\documentclass[reqno,12pt]{amsart}
\usepackage[colorlinks=true, linkcolor=blue, citecolor=blue]{hyperref}
    
\usepackage{amssymb}
\usepackage{amsmath, graphicx, rotating}
\usepackage{color}
\usepackage{soul}
\usepackage[dvipsnames]{xcolor}   
           
\usepackage{ifthen}
\usepackage{xkeyval}
\usepackage{todonotes}
\usepackage[T1]{fontenc}
\usepackage{lmodern}
\usepackage[english]{babel}

\usepackage{pifont}

\usepackage{ upgreek }
\usepackage{stmaryrd}
\SetSymbolFont{stmry}{bold}{U}{stmry}{m}{n}
\usepackage{amsthm}
\usepackage{float}

\usepackage{ bbm }
\usepackage{ stmaryrd }
\usepackage{ mathrsfs }
\usepackage{ frcursive }
\usepackage{ comment }

\usepackage{pgf, tikz}
\usetikzlibrary{shapes}
\usepackage{varioref}
\usepackage{enumitem}

\usepackage{mathtools}

\usepackage{dsfont}

\definecolor{rouge}{rgb}{0.7,0.00,0.00}
\definecolor{vert}{rgb}{0.00,0.5,0.00}
\definecolor{bleu}{rgb}{0.00,0.00,0.8}

\usepackage[margin=1.5in, tmargin=1.3in, bmargin=1.3in]{geometry}  

\newtheorem{theorem}{Theorem}[section]
\newtheorem*{theorem*}{Theorem}
\newtheorem{lemma}[theorem]{Lemma}

\newtheorem{corollary}[theorem]{Corollary}
\newtheorem{proposition}[theorem]{Proposition}

\labelformat{hypothesis}{\textbf{M\kern-0.1mm#1}}

\newtheorem{condition}{Condition}

\labelformat{conditionA}{\textbf{A\kern-0.1mm#1}}

\renewcommand\dots{\hbox to 1em{.\hss.\hss.}}

\theoremstyle{definition}

\newtheorem{remark}[theorem]{Remark}

\numberwithin{equation}{section}

\newcommand*{\abs}[1]{\left\lvert#1\right\rvert}

\def\bb#1{\mathbb{#1}}
\def\bf#1{\mathbf{#1}}
\def\geq{\geqslant}
\def\leq{\leqslant}

\def\bb{\mathbb}

\def\vphi{\varphi}

\def\..{\ldots} 
\def\…{\ldots}

\def\HH{H\"older}

\def\scr{\mathscr}

\newcommand\ee{\varepsilon}

\DeclareMathOperator{\Leb}{Leb}

\begin{document}

\title[Conditioned limit theorems]
{Conditioned limit theorems for hyperbolic dynamical systems}

\author{Ion Grama}
\author{Jean-Fran\c cois Quint}
\author{Hui Xiao}

\curraddr[Grama, I.]{ Universit\'{e} de Bretagne-Sud, LMBA UMR CNRS 6205, Campus de Tohannic 56017, Vannes, France.}
\email{ion.grama@univ-ubs.fr}

\curraddr[Quint, J.-F.]{ CNRS-Universit\'{e} de Bordeaux, 33405, Talence, France.}
\email{Jean-Francois.Quint@math.u-bordeaux.fr}

\curraddr[Xiao, H.]{ Universit\"{a}t Hildesheim, Institut f\"{u}r Mathematik und Angewandte Informatik, 31141, Hildesheim, Germany.}
\email{xiao@uni-hildesheim.de}


\begin{abstract}
Let $(\bb X, T)$ be a subshift of finite type equipped with the Gibbs measure $\nu$ and let 
$f$ be a real-valued H\"older continuous function on $\bb X$ such that $\nu(f) = 0$.
Consider the Birkhoff sums $S_n f = \sum_{k=0}^{n-1} f \circ T^{k}$,  $n\geq 1$. 
For any $t \in \bb R$, denote by $\tau_t^f$ the first time when the sum  $t+ S_n f$ 
leaves the positive half-line for some $n\geq 1$.
By analogy with the case of random walks with independent identically distributed increments, 
we study the asymptotic as $ n\to\infty $ of the probabilities
$ \nu(x\in \bb X: \tau_t^f(x)>n) $ and $ \nu(x\in \bb X: \tau_t^f(x)=n) $.
We also establish integral and local type limit theorems for the sum $t+ S_n f(x)$ 
conditioned on the set $\{ x \in \bb X: \tau_t^f(x)>n \}.$
 \end{abstract}

\date{\today}
\subjclass[2020]{Primary 37A30, 37A50, 37C05; Secondary 60F05, 60J05}
\keywords{Conditioned limit theorems; conditioned local limit theorems; exit time; dynamical systems}
\maketitle

\tableofcontents


\section{Statement of the results and motivation}

\subsection{Main results}
Consider a subshift of finite type $(\bb X, T)$ endowed with a Gibbs measure $\nu$ and 
let  $f$ be a real-valued H\"older continuous function on $\bb X$ 
(the precise definitions are given in Section \ref{sec-background}). 
Define the Birkhoff sums
\begin{align*}
S_n f = f + f \circ T + \ldots + f \circ T^{n-1},  \quad n\geq 1. 
\end{align*}
A fundamental result of the theory of dynamic systems 
is the celebrated Birkhoff ergodic theorem which asserts that $\nu$-almost surely,
\begin{align*}
\lim_{n \to \infty} \frac{S_nf}{n} = \int_{\bb X} f(x) \nu(dx) =: \nu(f). 
\end{align*}
Much effort was made to establish another important property -- the central limit theorem for $S_nf$.  
To formulate the corresponding statement, we first note that the following limit exists:
\begin{align*} 
\sigma_{f }^2 = \lim_{n\to \infty} \frac{1}{n}\int_{\bb X} (S_n f- n \nu(f))^2 d\nu. 
\end{align*}
It is known that $\sigma_f^2=0$ if and only if $f$ is a coboundary w.r.t.\ to $T$, which means that   
there exists a H\"{o}lder continuous function $g$ on $\bb X$ such that $f(x) = g(T x) - g(x)$ for any $x \in \bb X$.  
In the case when $\sigma_{f }>0$  (or equivalently when $f$ is not a coboundary) the following central limit theorem holds: 
for any bounded continuous function $F: \bb R \mapsto \bb R$, 
\begin{align} \label{CLT-forDS-001}
\lim_{n \to \infty} \int_{\bb X} F \left( \frac{S_n f(x)  -n \nu(f)    }{ \sigma_f \sqrt{n}} \right) \nu(dx)
= \frac{1}{\sqrt{2 \pi}} \int_{\bb R} F(t) e^{- \frac{t^2}{2}} dt. 
\end{align}

All these statements, which can be found in the excellent book of Parry and Pollicott \cite{PP90},  
are consequences of many successive works such as Sinai \cite{Sin60, Sin72}, Ratner \cite{Rat73}, Ruelle \cite{Rue78} and  
Denker-Phillip \cite{DP84}, 
to cite only a few. 
The goal of this paper is to complement the central limit theorem \eqref{CLT-forDS-001} by proving limit theorems for the Birkhoff sum $t+S_nf$ under the condition that the trajectory
$(t+S_kf)_{1\leq k\leq n}$ stays positive, where $t\in \bb R$ is a starting point.
A brief historical foray into the subject of conditioned limit theorems and our motivation are presented 
in Section \ref{sec: previous work - motivation}. 

To state our results assume that $\nu(f) = 0$ and that $f$ is not a coboundary. 
For any $t \in \bb R$, the following exit time is finite for $\nu$-almost every $x \in \bb X$: 
\begin{align*}
\tau_t^f(x) := \inf \left\{ k\geq 1, \, t + S_k f(x) < 0 \right\}.
\end{align*}
Thus, by definition, 
the set $\{x \in \bb X: \tau_t^f(x) >n\}$ is the set
where the trajectory $(t+S_kf)_{1\leq k\leq n}$ stays non-negative.

Our first theorem states the existence of a special Radon measure which will play a central role in the paper and 
will  be used in the formulations of the subsequent results. 

\begin{theorem}\label{Thm_Radon_Measure}
Let $f$ be a H\"older continuous function on $\bb X$ such that  $\nu (f) = 0$ and $f$ is not a coboundary.
Then, there exists a unique Radon measure $\mu^f$ on $\bb X \times \bb R$ such that 
for any continuous compactly supported function  $\vphi$ on $\bb X \times \bb R$,
\begin{align} \label{Harmonic_rho_f_Thm000}
\lim_{n \to \infty}  \int_{\bb X \times \bb R} \vphi(x, t)  S_{n} f (x) \mathds 1_{ \left \{ \tau_t^f(x) >n \right\} }  \nu(dx) dt
= \int_{\bb X \times \bb R}   \vphi(x, t) \mu^f(dx, dt).
\end{align}
Moreover, the Radon measure $\mu^f$ satisfies the following quasi-invariance property: 
for every continuous compactly supported function  $\vphi$ on $\bb X \times \bb R$,
\begin{align} \label{Harmonic_property_rho}
 \int_{\bb X \times \bb R}  \vphi(x, t) \mu^f(dx, dt)
 =    \int_{\bb X \times \bb R} \vphi \left( T^{-1} x, t - f(T^{-1} x) \right) \mathds 1_{\{ t \geq 0 \}} \mu^f(dx, dt). 
\end{align}
\end{theorem}

The limit \eqref{Harmonic_rho_f_Thm000} takes a simpler form when the function $\varphi$ does not depend on the first argument.
Indeed, we will show in Section \ref{Sect_harmonic_measure} that the marginal of $\mu^f$ on $\bb R$
is absolutely continuous with respect to the Lebesgue measure.
Its density function is a non-decreasing function on $\bb R$ that will be denoted by $V^f$.
In particular, by standard arguments,  the asymptotic \eqref{Harmonic_rho_f_Thm000} is valid for functions $\varphi$
of the form $\vphi(x, t) = \mathds 1_{[a, b]} (t)$ for $x \in \bb X$ and $t \in \bb R$. 
This leads to the following:

\begin{corollary}\label{Cor_Radon_Measure000}
Let $f$ be a H\"older continuous function on $\bb X$ such that  $\nu (f) = 0$ and $f$ is not a coboundary. 
Then, for any real numbers $-\infty < a < b < \infty$, we have, as $n\to \infty$, 
\begin{align}\label{Harmonic_rho_f_Thm000bb}
\lim_{n \to \infty} \int_a^b  \int_{\bb X}     S_{n} f (x) \mathds 1_{ \left \{ \tau_t^f(x) >n \right\} }  \nu(dx) dt
=   \mu^f(\bb X \times [a,b]) = \int_a^b V^f(t) dt. 
\end{align}
\end{corollary} 

Note that \eqref{Harmonic_rho_f_Thm000}  and \eqref{Harmonic_rho_f_Thm000bb} 
are stated in integral forms with respect to $t$. 
It is an open question whether it is possible to give an asymptotic of the integral 
$\int_{\bb X}  S_{n} f (x) \mathds 1_{ \{ \tau_t^f(x) >n \} }  \nu(dx)$ 
for a fixed value of $t$.

The Radon measure $\mu^f$ appearing in Theorem \ref{Thm_Radon_Measure}
will be called the {\it harmonic measure} associated to the dynamical system $(\bb X,  T, \nu)$
with the observable $f$. 
The reason for this is that the measure $\mu^f$ is related to the harmonicity property 
that appears in the study of killed random walks on the half line.
 We refer 
 to Section \ref{Sect_harmonic_measure} 
 for precise statements.  


The following results describe the limit behavior of the Birkhoff sum $t + S_n f$  
under the condition that the trajectory $(t + S_k f)_{1 \leq k \leq n}$ stays non-negative. 
We start by giving the equivalent of the probability that the trajectory $(t + S_k f)_{1 \leq k \leq n}$ stays non-negative. 
Denote by $\check \mu^f$ the harmonic measure related to the reversed dynamical system $(\bb X, T^{-1}, \nu)$ 
with the observable $f \circ T^{-1}$. 

\begin{theorem}\label{Thm-exit-time001} 
Let $f$ be a H\"older continuous function on $\bb X$ such that $f$ is not a coboundary and $\nu (f) = 0$.
Then, for any continuous compactly supported function $\vphi$ on $\bb X \times \bb R$, we have, as $n\to \infty$,
\begin{align*} 
 \int_{\bb X \times \bb R } \vphi(x, t)  \mathds 1_{\{\tau_t^f(x) >n\}} 
  \nu(dx) dt 
\sim \frac{2 }{\sigma_f \sqrt{2 \pi n} }  \int_{\bb X \times \bb R }  \vphi(x, t) \mu^{f}(dx, dt)  
\end{align*}
and    
\begin{align*} 
  \int_{\bb X \times \bb R } 
\vphi(T^n x, t + S_n f(x) )  &\mathds 1_{\{\tau_t^f(x) >n\}}  \nu(dx) dt 
\sim  \frac{2}{\sigma_f \sqrt{2 \pi n}}  \int_{\bb X \times \bb R }  \vphi(x, t) \check \mu^{(-f)}(dx, dt).
\end{align*}
\end{theorem}

In particular, as the measure $\mu^f$ has absolutely continuous marginal on $\bb R$, 
Theorem \ref{Thm-exit-time001} also applies to the function $\vphi(x, t) = \mathds 1_{[a, b]} (t)$ 
for $x \in \bb X$ and $t \in \bb R$. 
Therefore, this gives the following corollary. 

\begin{corollary}\label{Cor-exit-time001}
Let $f$ be a H\"older continuous function on $\bb X$ such that $f$ is not a coboundary and $\nu (f) = 0$.
Then, for any real numbers $-\infty < a < b < \infty$, we have, as $n\to \infty$,
\begin{align*} 
 \int_{a}^{b}  \nu \left( x \in \bb X:  \tau_t^f(x) >n \right) dt 
\sim  
 \frac{2}{\sigma_f \sqrt{2 \pi n} } \int_a^b V^f(t) dt.   
\end{align*}
\end{corollary}


Now we give a conditioned central limit theorem for the Birkhoff sum $S_n f$,
which states that the law of $S_n f$ conditioned to stay positive converges weakly to the Rayleigh law. 
Let $\phi^+$ be the Rayleigh density function:
\begin{align*} 
\phi^+(u) = u e^{-u^2/2} \mathds 1_{\{ u \geq 0 \}},  \quad  u \in \bb R.  
\end{align*}


\begin{theorem}\label{Thm-cnodi-lim-theor001} 
Let $f$ be a H\"older continuous function on $\bb X$ such that $f$ is not a coboundary and $\nu (f) = 0$.
Then, for any continuous compactly supported function $F$ on $\bb X \times \bb X \times \bb R \times \bb R$, 
we have, as $n\to\infty$,
\begin{align*}  
 \int_{\bb X \times \bb R }  
  F &\left(x, T^n x, t, \frac{S_n f (x)}{\sigma_f\sqrt{n}} \right) 
  \mathds 1_{\left\{ \tau_t^f(x) >n \right\} } \nu(dx)  dt  \nonumber \\
& \quad\quad\quad\sim  
\frac{2 }{\sigma_f  \sqrt{2 \pi n} }  
   \int_{\bb X \times \bb R }  \int_{\bb X \times \bb R }  F(x, x', t, t') \phi^+(t')    \nu(dx') dt' \mu^f(dx, dt). 
\end{align*}
\end{theorem}

As above, we can actually apply Theorem \ref{Thm-cnodi-lim-theor001} 
to the function $F(x, x', t, t') = \mathds 1_{[a, b]}(t) \mathds 1_{[a', b']}(t')$ for $x, x' \in \bb X$ and $t, t' \in \bb R$. 
Therefore, this implies the following corollary: 

\begin{corollary}
Let $f$ be a H\"older continuous function on $\bb X$ such that $f$ is not a coboundary and $\nu (f) = 0$.
Then, for any real numbers $-\infty < a < b < \infty$ and $-\infty < a' < b' < \infty$, we have, as $n\to \infty$,
\begin{align*}  
& \int_{a}^{b}  
 \nu \left( x \in \bb X:  \frac{S_n f (x)}{\sigma_f\sqrt{n}} \in [a',b'],  \tau_t^f(x) >n \right) dt   \notag\\
& \qquad\qquad\qquad \sim  
\frac{2 }{\sigma_f  \sqrt{2 \pi n} }  \int_a^b V^f(t) dt
   \int_{a'}^{b'}  \phi^+(t')  dt'. 
\end{align*}
\end{corollary}

Next we formulate a conditioned local limit theorem for $S_n f$,
which is a refinement of the previous result. 

\begin{theorem}\label{Thm-conLLT}  
Let $f$ be a H\"older continuous function on $\bb X$ such that  $\nu (f) = 0$.  
Assume that for any $p \neq 0$ and $q \in \bb R$, 
the function $p f + q$ is not cohomologous to a function with values in $\bb Z$. 
Then, for any continuous compactly supported function $F$ on $\bb X \times \bb X \times \bb R \times \bb R$, 
we have, as $n\to\infty$, 
\begin{align*}
    \int_{\bb X \times \bb R} F&(x, T^nx, t, t + S_n f(x))
    \mathds 1_{\left\{ \tau_t^f(x) >n - 1 \right\} }  \nu(dx) dt   \notag\\
&\quad\quad \sim  \frac{2}{\sqrt{2 \pi} \sigma_f^3  n^{3/2} } 
\int_{\bb X \times \bb R }  \int_{\bb X \times \bb R }  F(x, x', t, t') \mu^f(dx, dt)  \check \mu^{(-f)} (dx', dt'). 
\end{align*}
\end{theorem}

In Theorem \ref{Thm-conLLT}, we assumed the function $f$ to satisfy a non-arithmeticity condition.
When this is not the case but $f$ is still not cohomologous to $0$, 
we could still get an analogue of this result by the same method. 
 

In the particular case when the function $F$ has the form 
$F(x, x', t, t') = \mathds 1_{[a, b]}(t) \mathds 1_{[a', b']}(t')$ for $x, x' \in \bb X$ and $t, t' \in \bb R$,
from the previous theorem we obtain the following: 

\begin{corollary}\label{Cor-LLT-32-007}
Let $f$ be a H\"older continuous function on $\bb X$ such that  $\nu (f) = 0$.  
Assume that for any $p \neq 0$ and $q \in \bb R$, 
the function $p f + q$ is not cohomologous to a function with values in $\bb Z$. 
Then, for any real numbers $-\infty < a < b < \infty$ and $-\infty < a' < b' < \infty$, we have, as $n\to \infty$,
\begin{align*}
    \int_{a}^b  & \nu \left( x \in \bb X:   t + S_n f (x) \in [a', b'],  \tau_t^f(x) > n - 1 \right) dt   \notag\\
&\quad\quad \sim  \frac{2}{\sqrt{2 \pi} \sigma_f^3  n^{3/2} } 
    \int_a^b V^f(t) dt   \int_{a'}^{b'}  \check V^{(-f)}(t') dt'.
\end{align*}
\end{corollary}

In this corollary we have denoted by $\check V^{(-f)}$ the density function with respect to the Lebesgue measure
of the marginal on $\bb R$ of the Radon measure $\check \mu^{(-f)}$. 

As $f$ is bounded, by taking the interval $[a', b')$ to be $[-c, 0)$ for $c>0$ large enough, 
we get from Corollary \ref{Cor-LLT-32-007} the following:

\begin{corollary}\label{Cor-LLT-32-008}
Let $f$ be a H\"older continuous function on $\bb X$ such that  $\nu (f) = 0$.  
Assume that for any $p \neq 0$ and $q \in \bb R$, 
the function $p f + q$ is not cohomologous to a function with values in $\bb Z$. 
Then, for any real numbers $-\infty < a < b < \infty$, we have, as $n\to \infty$,
\begin{align*}
    \int_{a}^b   \nu \left( x \in \bb X:  \tau_t^f(x) = n  \right) dt   
 \sim  \frac{2}{\sqrt{2 \pi} \sigma_f^3  n^{3/2} } 
    \int_a^b V^f(t) dt   \int_{-\infty}^{0}  \check V^{(-f)}(t') dt'. 
\end{align*}
\end{corollary}

Our Corollary \ref{Cor-LLT-32-008} could be extended
without difficulties to the case when one only assumes that $f$ is not cohomologous to $0$. 
This  assertion could be deduced from a version of Theorem \ref{Thm-conLLT} 
for functions $f$ that are cohomologous to functions with values 
in a set of the form $\alpha \bb Z + \beta$ for some $\alpha, \beta \in \bb R$. 

Again, similarly to the comment after Corollary \ref{Cor_Radon_Measure000}, 
Theorems \ref{Thm-exit-time001}, \ref{Thm-cnodi-lim-theor001}  and \ref{Thm-conLLT} are stated in integral forms with respect to $t$. 
It is an open problem to obtain asymptotics for a fixed value of $t \in \bb R$ of the following probabilities: 
\begin{align*}
& \nu \left( x \in \bb X:  \tau_t^f(x) >n \right),   \quad \nu \left( x \in \bb X:  \frac{S_n f (x)}{\sigma_f\sqrt{n}} \in [a',b'],  \tau_t^f(x) >n \right),  \notag\\
& \nu \left( x \in \bb X:  S_n f (x) \in [a',b'],  \tau_t^f(x) >n \right). 
\end{align*}

Due to the theory of Markov partitions (see Appendix III of \cite{PP90} and Chapter 18.7 of \cite{Has-Katok-96}), 
Theorems \ref{Thm_Radon_Measure}, \ref{Thm-exit-time001}, \ref{Thm-cnodi-lim-theor001}
and \ref{Thm-conLLT}  can be applied without any changes to hyperbolic dynamical systems.
Finally, using the approach of this paper, one can obtain analogous results for hyperbolic flows. 
The latter is out of the scope of this article and will be done in another work. 



\subsection{Previous work and motivation} \label{sec: previous work - motivation}
The first examples of conditioned limit theorems for 
sums of independent random variables are due to the pioneering work of Spitzer \cite{Spitzer} and Feller \cite{Fel64}. 
Since then integral and local limit theorems  for random walks conditioned to stay positive attracted
a lot of attention.  
Very many authors contributed to this subject,  among them 
  Borovkov \cite{Borovk62, Borovkov04a, Borovkov04b}, 
Bolthausen \cite{Bolth}, Iglehart \cite{Igl74}, Eppel \cite{Eppel-1979},  Bertoin and Doney \cite{BertDoney94},  
Caravenna \cite{Carav05},  Vatutin and Wachtel \cite{VatWacht09}, Doney \cite{Don12}, Kersting and Vatutin \cite{KV17}. 
Most of this work is based on the Wiener-Hopf factorization and
various factorization identities. 
Recently Denisov and Wachtel \cite{Den Wachtel 2011, DW19} have studied the setting of random walks in cones and have developed a new approach 
based on the construction of the harmonic function thus avoiding
the use of the Wiener-Hopf factorization.
Following this method, in the case of dependent random variables very recent 
progress was made in 
\cite{GLP17, GLL18}, 
where conditioned integral limit theorems for products of random matrices and
for Markov chains satisfying spectral gap properties have been obtained.
In \cite{GLL20} a conditioned local limit theorem for a Markov chain with finite state space was considered. 

The question of establishing  
conditioned integral and local limit theorems for Birkhoff sums remained for a very long time 
a challenging problem mainly because the Wiener-Hopf factorisation 
does not work in the dynamical system framework. 
The emergence of new approaches for proving the integral and local limit theorems for Markov chains 
\cite{Den Wachtel 2011, GLP17, GLL18, GLL20}
 opens an opportunity to investigate these problems for dynamical systems.
However, there are two main difficulties which have motivated our study.

The first difficulty is that the 
alternative approaches described above are based on the existence of the harmonic function. 
For Markov chains its existence  is proved in \cite{GLP17, GLL18}. 
Or a subshift of finite type exhibits a dependence which in general is not of 
Markovian type and therefore the 
existence of the harmonic function is not granted, except in some very special cases. 
Indeed, one of the major findings of the paper is that in the case for the subshift of finite type $(\bb X, T,\nu)$
with a general H\"older continuous observable $f$ 
a more general object -- the harmonic measure $\mu^f$ -- has to be introduced. 
The conditioned integral and local limit theorems in this case are stated in terms of the harmonic measures 
$\mu^f$ and $\check \mu^{-f}$ (the second corresponding to the reversed subshift $(\bb X, T^{-1}, \nu)$ 
with observable $f \circ T^{-1}$), 
which therefore cannot be avoided and constitute essential characteristics of the model.
The construction of $\mu^f$ is performed first for a Birkhoff sum $t+S_nf$, where the observable 
$f$ depends only on the future coordinates 
(in the reversed setting it corresponds to studying a Markov chain). 
Then it is extended gradually to a function $f$ depending on the whole set of coordinates using some smoothing
techniques and a vague convergence argument, see Section \ref{sec: harmonic measure}.

The second difficulty resides in adapting the approaches 
from the Markov chains \cite{GLL18, GLL20} to the case of dynamical systems.
Our approach is close to that of the recent paper \cite{GX21}, where random walks with independent increments 
are considered.
We shall first establish the corresponding theorems for the Birkhoff sum $t+S_nf$ with an observable 
$f$ depending only on the future coordinates (which corresponds to dealing with some Markov chain). 
Then using the technique similar to that in the previous point we extend them to the general case.
The main difficulty here consists in the construction of the dual Markov chain, 
which is an important step in establishing a conditioned local limit theorem.
Fortunately, in the case of the subshift of finite type the dual Markov chain 
corresponds again to a dynamical system 
-- the reversed subshift $(\bb X, T^{-1}, \nu)$ with observable $f \circ T^{-1}$.

%

To put it a perspective, the approach developed in our paper can be applied to Birkhoff sums with drift, via a Cram\'er Type change of measure,
see Parry and Pollicott \cite{PP90} and  Waddington \cite{Wadd96}. 
It is also possible to apply the developed approach for  
studying local limit theorems for products of random matrices and more generally for Markov chains 
 with values in general state spaces in contrast to \cite{GLL20}, where the chain with finite state spaces have been considered.   
This will be the subject of a forthcoming paper.


%
%
%
\section{Background and auxiliary statements} \label{sec-background} 

\subsection{Subshift of finite type and Gibbs measure}\label{SubSec_Shift}
We start by precisely introducing the subshift of finite type. 
Let $k \geq 2$ be an integer and $A=\{1, 2, \ldots, k\}$. 
Let $M$ be a transition matrix on $A$, that is to say, 
$M = (M (i,j))_{ i, j \in A }$ is a matrix with coefficients in $\{0, 1\}$.  
We assume that the transition matrix $M$ is irreducible in the sense that
there exists an integer $p \geq 1$ such that all the coefficients of the matrix $M^p$ are strictly positive.   
Consider the associated subshift of finite type 
\begin{align*}
\bb X = \left\{ x = (x_n)_{n \in \bb Z} \in A^{\bb Z}: M(x_n, x_{n+1}) = 1,  n \in \bb Z \right\} \subset A^{\bb Z}, 
\end{align*}
equipped with the shift map $T$ defined by $(T x)_n = x_{n+1}$ for $x \in \bb X$ and $n \in \bb Z$.
The set $\{1, 2, \ldots, k\}$ 
is equipped with the discrete topology,
so the space $A^{\bb Z}$ is compact with the corresponding Tychonov product topology.
We equip $\bb X$ with the induced topology, which is also compact. 
For any $x = (x_n)_{n \in \bb Z} \in \bb X$ and $y = (y_n)_{n \in \bb Z} \in \bb X$, define
\begin{align*}
\omega (x, y) = \min \{k \geq 0: x_k \neq y_k \  \mbox{or} \  x_{-k} \neq y_{-k}  \}. 
\end{align*}
Note that for any constant $\alpha \in (0,1)$, the function $(x, y) \mapsto \alpha^{\omega(x, y)}$ is a distance on $\bb X$
which induces the natural product topology. 

The space of  real-valued continuous functions $f: \bb X \mapsto \bb R$ is denoted by $\mathcal C(\bb X)$.
For any function $f \in \mathcal C(\bb X)$, we say that $f$ is $\alpha$-H\"{o}lder continuous on $\bb X$ 
if there exist constants $C >0$ and $\alpha \in (0,1)$ such that for all $x, y \in \bb X$,
\begin{align*}
|f(x) - f(y)| \leq C \alpha^{\omega(x,y)}. 
\end{align*}
Denote by $\scr B_{\alpha}$ the space of all real-valued $\alpha$-H\"{o}lder  continuous functions on $\bb X$,
equipped with the following norm
\begin{align}\label{Def_HolderNorm1}
\| f \|_{\mathscr B_{\alpha}}: = \sup_{x \in \bb X} |f(x)| + \sup_{x, y \in \bb X: x \neq y} \frac{|f(x) - f(y)|}{\alpha^{\omega(x,y)}}. 
\end{align}
In the sequel we shall say that a real-valued function on $\bb X$ is H\"{o}lder continuous if there exists $\alpha\in (0,1)$ 
such that $f$ is $\alpha$-H\"{o}lder continuous. 
Let $\scr B = \cup_{0 < \alpha < 1} \scr B_{\alpha}$ be the set of all real-valued H\"{o}lder continuous functions on $\bb X$.

For any $f \in \scr B$, we consider the Birkhoff sum process $(S_nf)_{n\geq 0}$ by setting $S_0f=0$ and 
\begin{align*}
S_n f = f + \ldots + f\circ T^{n-1}, \quad n \geq 1. 
\end{align*}
Let us denote by $\bb X^+ \subset A^{\bb N}$ the set 
\begin{align*}
\bb X^+ = \left\{ x = (x_n)_{n \in \bb N} \in A^{\bb N}: M(x_n, x_{n+1}) = 1,  n \in \bb N \right\} \subset A^{\bb N}.  
\end{align*}
The set of continuous functions on $\bb X^+$ is denoted by $\mathcal{C}(\bb X^+)$. 
By abuse of notation, the one-sided shift map $\bb X^+ \mapsto \bb X^+$ will still be denoted by $T$. 

The Ruelle operator $\mathcal L_f: \mathcal C(\bb X^+) \mapsto \mathcal C(\bb X^+)$ related to $f \in \mathcal C(\bb X^+)$
is defined as follows: 
for any $g \in \mathcal C(\bb X^+)$, 
\begin{align}
\mathcal L_{f} g(x) = \sum_{y: \, Ty = x} e^{-f(y)} g(y),  \quad x \in \bb X^+.  \label{Def_OperatorP}
\end{align}
One can easily see that $\mathcal L_{f}$ is a bounded linear operator on $\mathcal C(\bb X^+)$.
From \eqref{Def_OperatorP}, by iteration, it follows that for any $n \geq 1$,
\begin{align*}
\mathcal L_{f}^n g(x) = \sum_{y: \, T^n y = x} e^{-S_nf(y)} g(y),  \quad x \in \bb X^+.
\end{align*}
Also, if $h \in \mathcal{C}(\bb X^+)$, we have the conjugacy relation 
\begin{align}\label{ConjugacyRelation}
\mathcal L_{f + h \circ T - h} g = e^{-h} \mathcal L_{f}  \left(e^{h} g\right),
\end{align}
which tells us that the spectral properties of the transfer operator $\mathcal L_{f}$ only depend on the cohomology class of $f$. 
We say that a real valued and H\"older continuous function $\psi$ on $\bb X^+$ is normalized if  $\mathcal L_{\psi} \mathds 1 = \mathds 1$. 
By \cite[Chapter 2, Theorem 2.2]{PP90}, there exist a real valued H\"older continuous function $h$ and a real number $\lambda$
such that $\mathcal L_{\psi} e^{h} = e^{\lambda + h}$.
From the conjugacy relation \eqref{ConjugacyRelation}, this tells us that the function $\psi - h\circ T + h +\lambda$ is also normalized. 
 Therefore,  all over the paper, we will assume that $\psi$ is normalized. 
In this case, it is well known (e.g. \cite{PP90}) that the operator $\mathcal L_{\psi}$ admits a unique invariant probability measure
$\nu^+$. 
The measure $\nu^+$ is called the Gibbs measure related to the potential $\psi$.
Since $\psi$ is normalized, the measure $\nu^+$ is $T$-invariant, that is, for any $f\in \mathcal C(\bb X^+)$,
\begin{align} \label{transl-invar001}
\nu^+ (f\circ T) = \nu^+ (f),
\end{align}
see \cite[Chapter 2]{PP90}. 


Note that $\nu^+$  is also $T$-ergodic, meaning that any $T$-invariant Borel subset $B$ 
of $\bb X^+$ has $\nu^+$ measure $0$ or $1$:
\begin{align*} 
T^{-1}B=B  \quad \Rightarrow \quad  \nu^+(B)\in \{0,1\}.
\end{align*}

Thanks to the following lemma, the measure $\nu^+$ allows to define a $T$-invariant measure on $\bb X$.

\begin{lemma}\label{Lem_Unique_nu}
Let $\nu^+$ be a Borel probability measure on $\bb X^+$ which is $T$-invariant.
Then there exists a unique $T$-invariant Borel probability measure $\nu$ on $\bb X$
such that the image of $\nu$ under the natural projection map $\bb X \mapsto \bb X^+$ is equal to $\nu^+$.  
\end{lemma}

The proof of this lemma is just a consequence of Kolmogorov's extension theorem. 
We will actually give an explicit formula for the measure $\nu$ below. 

\subsection{Conditional measures on the past}

For any $z\in \bb X^+ $, we shall construct 
a measure $\nu_z^- $, which is the conditional measure of $\nu$ with respect to the map 
$x\in \bb X \mapsto  x_+ \in \bb X^+$.  
To this end, for any $a \in A$, let 
\begin{align*}
\bb X^-_a 
= \left\{ y \in A^{- \bb N^*}: M(y_{-1}, a) = 1, M(y_{-n-1}, y_{-n}) = 1, \forall n \geq 1 \right\}, 
\end{align*}
where $M$ is the transition matrix on the set $A$ 
which was used to define the finite type subshift $\bb X \subset A^{\bb Z}$. 
For any $z \in \bb X^+$, we set $\bb X^-_z = \bb X^-_{z_0}$, where $z_0$ is the first coordinate of $z \in \bb X^+$.  
We have the decomposition 
\begin{align*}
\bb X = \bigcup_{z \in \bb X^+} \bb X^-_z \times \{z\}. 
\end{align*}
The point $z$ may be thought of as the future of the trajectory whereas the elements of $\bb X^-_z$ describe
the pasts which are compatible with this future.  
Let us introduce some notation related to this decomposition. 
For any $z \in \bb X^+$ and $y \in \bb X^-_z$, we denote $y\cdot z = (y, z) \in \bb X$.
For $z\in \bb X^+ $ and $k \geq 1$, we set 
\begin{align*}
A^k_z  &  = \Big\{ (y_{-k}, \ldots, y_{-1})   \in A^{\{ -k, \ldots, -1 \} }:   \nonumber\\
 &  \qquad M(y_{-1}, z_0) = 1, M(y_{-n - 1}, y_{-n}) = 1, \forall  1 \leq n \leq k -1  \Big\}. 
\end{align*}
For $(y_{-k}, \ldots, y_{-1}) \in A^k_z$, we set 
$
y_{-k} \ldots y_{-1} \cdot z
$
to be the element $w \in \bb X^+$ defined by 
\begin{equation*}
w_n = 
\begin{cases}
y_{n -k}  &  \mbox{if} \ 0 \leq n \leq k - 1  \\
z_{n- k}  &  \mbox{if} \ n \geq k. 
\end{cases}
\end{equation*}
For $a \in A^k_z$, let 
\begin{align} \label{Def-CYLIND-001}
\bb C_{a, z} = \{ y \in \bb X^-_z:  y_{-k} = a_{-k}, \ldots,  y_{-1} = a_{-1} \}
  \end{align} 
  be the associated cylinder of length $k$ in $\bb X^-_z$. 

Recall that the two-sided shift map 
$T: \bb X \mapsto \bb X $ and its inverse $T^{-1}$ are defined by 
$(T x)_n = x_{n +1}$ and $(T^{-1} x)_n = x_{n-1}$ for any $x \in \bb X$ and $n \in \bb Z$. 
By abuse of notation, the one-sided forward shift map will be denoted by $T:  \bb X^+ \mapsto \bb X^+ $ 
and is defined by $T(x) =(x_{1},x_{2},x_{3},\ldots),$ 
for any $x=(x_{0},x_{1},x_{2},\ldots)\in \bb X^+$.
Let us define conditional measures on the past of trajectories.
For $k \geq 0$, define $\nu^k_{z}$ as a function on cylinders of length $k$ in $\bb X^-_z$ 
by the formula 
\begin{align}\label{Def_nu_kz}
\nu^k_z (\bb C_{a, z}) =   \exp( - S_k \psi (a \cdot z)  ), 
\end{align}
for $a\in A_z^k$.
Since $\mathcal{L}_\psi \mathds 1 = 1$, 
we have that for any $a \in A^k_z$, 
\begin{align}\label{Extension_measure_nu-initial}
\nu^k_z (\bb C_{a, z}) 
= \sum_{\substack{b \in A \\ M(b, a_{-k}) = 1}}  \nu^{k+1}_z (\bb C_{b \cdot a, z}).  
\end{align}
By Kolmogorov's extension theorem, from
equation \eqref{Extension_measure_nu-initial} it follows that
there exists a unique Borel probability measure $\nu^-_z$ on $\bb X^-_z$ 
such that for any $k \geq 0$, 
$\nu^k_z$ is the restriction of $\nu^-_z$ to cylinders of length $k$. 

We can now give an explicit formula for the measure $\nu$ in terms of the measures $\nu^+$ and $\nu_z^-$.

\begin{lemma}\label{Lem_Fubini}
Let $\vphi \in \mathcal{C(\bb X)}$. Then we have 
\begin{align*}
\nu(\varphi) = \int_{\bb X^+} \int_{\bb X^-_z} \vphi(y \cdot z) \nu^-_z(dy) \nu^+(dz). 
\end{align*}
\end{lemma}

\begin{proof}
By Lemma \ref{Lem_Unique_nu}, it suffices to prove that the measure $\nu$ on $\bb X$
defined by the above equation is $T$-invariant. 
This property is a direct consequence of the definition of the measures $\nu_z^-$, $z \in \bb X^+$,
and of the fact that $\nu^+$ is $\mathcal{L}_{\psi}$-invariant. 
\end{proof}

We shall use the fact that the measures $\nu^-_z$ and $\nu^-_{z'}$ are equivalent. 

\begin{lemma}\label{Lem_Absolute_Contin}
There exists a real-valued continuous function $\theta$ on the set 
\begin{align*}
\bb X_3 :=  
\left\{ (y, z, z') \in A^{-\bb N^*} \times \bb X^+ \times \bb X^+: z_0 = z_0',  \,  y \in \bb X^-_z = \bb X^-_{z'}  \right\}
\end{align*}
such that for any $z, z' \in \bb X^+$ and any continuous function $\varphi$ on $\bb X^-_z$, one has
\begin{align*}
\int_{ \bb X^-_z } \varphi(y) \nu^-_{z'}(dy) =  \int_{ \bb X^-_z } \varphi(y) e^{\theta (y, z, z')} \nu^-_{z}(dy). 
\end{align*}
Besides, there exists a constant $c>0$ such that for any $(y, z, z') \in \bb X_3$, 
\begin{align*}
|\theta(y, z, z')| \leq c \alpha^{\omega(z,z')}.   
\end{align*}
\end{lemma}

\begin{proof}
Indeed, it suffices to set 
\begin{align*}
\theta(y, z, z') = \sum_{k=1}^{\infty} \left( \psi(T^{-k} (y \cdot z)) - \psi(T^{-k} (y \cdot z')) \right). 
\end{align*}
\end{proof}



\subsection{General properties of exit times}\label{General-Gibbs-Mea}

From the following lemma it follows that the function $x \mapsto \tau_t^f(x)$ is finite $\nu$-almost surely. 

\begin{lemma}\label{Lem_tau_finite}
Let $f \in \scr B$ with $\nu(f) = 0$. Assume that $f$ is not a coboundary. 
Then for $\nu$-almost every $x \in \bb X$, one has
\begin{align*}
\inf_{n \geq 1} S_n f(x) = - \infty. 
\end{align*}
\end{lemma}

\begin{proof}
Consider the Borel set
\begin{align*}
A = \left\{ x \in \bb X:  \inf_{n \geq 1} S_n f(x) > - \infty \right\}. 
\end{align*}
It is clear that the set $A$ is $T$-invariant. Therefore $\nu(A) = 0$ or $\nu(A) = 1$. 
Assume that $\nu(A) = 1$, then let us show that $f$ is a coboundary.
Indeed, for any $x \in A$, we have that $h(x): = \liminf_{n \to \infty} S_n f(x) > -\infty$. 
Since $\nu(f) = 0$, it is well known that $S_n f(x)$ admits finite limit points for $\nu$-almost every $x \in \bb X$,
so that $h(x) < \infty$. 
Now by definition, we have $h(Tx) = h(x) - f(x)$, hence $f$ is a coboundary as a measurable function on $\bb X$. 
Therefore, by \cite[Proposition 6.2]{PP90}, we get that $f$ is a coboundary as a H\"{o}lder continuous function on $\bb X$. 
\end{proof}

For notational reasons, it is more convenient to study objects defined by the reverse shift $T^{-1}$.
Note that the two studies are equivalent.

Indeed,  let us define the flip map $\iota: A^{\bb Z} \mapsto A^{\bb Z}$ by the following relation: for any 
$x = (\ldots, x_{-1}, x_0, x_{1}, \ldots)\in A^{\bb Z}$ it holds $\iota(x) = (\ldots,x_{1}, x_0, x_{-1},\ldots ) \in A^{\bb Z}$,
that is $(\iota x)_n = x_{-n}$ for $n \in \bb Z$. 
The following lemma is classical (see \cite[Chapter 2]{PP90}). 

\begin{lemma}\label{Lem-iotanu-01}
The set $\iota \bb X$ is a subshift of finite type 
and the measure $\iota_* \nu$ is a Gibbs measure on $\iota \bb X$. 
\end{lemma}

%

For $f \in \scr B$, consider the reversed Birkhoff sum process $(\check S_n f)_{n \geq 1}$ 
which is defined as follows: for any $x\in \bb X$, 
\begin{align*}
\check S_n f (x) = f(T^{-1} x) + f(T^{-2} x)  +  \ldots + f( T^{-n} x ) = S_n f (T^{-n} x), \quad n \geq 1. 
\end{align*}
In the same way, denote by $ \check \tau_t^f(x)$ the first time when $t + \check S_n f(x)$ becomes negative:
for any $x\in \bb X$, 
\begin{align} \label{timecheck001}
\check \tau_t^f(x) := \inf \left\{ k\geq 1, \, t + \check S_k f(x) < 0 \right\}. 
\end{align}
Then the relation between the exit times $\tau_t^f$ and $\check\tau_t^{f \circ \iota}$ is given by 
\begin{align*} 
\tau_t^{f} (Tx) = \check \tau_t^{f \circ \iota} (\iota x),  \quad x \in \bb X. 
\end{align*}
%
%
%
In the present paper we deal with the measure 
\begin{align} \label{probabexit001}
 \nu \left(x  \in \bb X:  \tau_t^f(x) >n \right) 
\end{align}
which, by the discussion above, is equivalent to studying the measure 
\begin{align} \label{probabexit002}
  \nu \left(x  \in \bb X:  \check \tau_t^f(x) >n \right). 
\end{align}
In turn, Lemma \ref{Lem_Fubini} shows that in order to study \eqref{probabexit002}, 
it suffices to investigate
\begin{align} \label{equiv-formulation001}
\nu^-_z \big( y \in \bb X^-_z: \check \tau_t^f(y \cdot z) >n \big),
\end{align}
for $z \in \bb X^+$.
We will do it 
by using tools from the theory of Markov chains \cite{GLL18}. 
In particular, we will make use of the martingale approximation for the process $(\check S_n f)_{n \geq 1}$. 



\subsection{Martingale approximation} \label{Sec_CohomologicalEq}

Recall that $\scr B = \cup_{0 < \alpha < 1} \scr B_{\alpha}$, where $\scr B_{\alpha}$ is
the space of  real-valued $\alpha$-H\"{o}lder continuous  functions on $\bb X$ endowed with the norm \eqref{Def_HolderNorm1}. 
In the same way, we denote by $\scr B^+_{\alpha}$ 
the space of  real-valued $\alpha$-H\"{o}lder continuous  functions on $\bb X^+$ endowed with the norm
\begin{align*}
\| f \|_{\mathscr B^+_{\alpha}}: = \sup_{x \in \bb X^+} |f(x)| + \sup_{x, y \in \bb X^+: x \neq y} \frac{|f(x) - f(y)|}{\alpha^{\omega(x,y)}}. 
\end{align*}
Let $\scr B^+ = \cup_{0 < \alpha < 1} \scr B_{\alpha}^+$. 
Note that every H\"older continuous function $f$ on $\bb X^+$ can be extended to a 
\HH \ continuous function on $\bb X$  defined by
\begin{align*} 
x = (\ldots, x_{-1}, x_0, x_{1}, \ldots) \in \bb X \mapsto f(x_0, x_{1}, \ldots),
\end{align*}
so we can identify $\scr B^+$ with a subspace of $\scr B$.

Let $f\in \scr B$. 
Define the cohomology class of $f$ as the following set of H\"older continuous functions:
\begin{align*} 
\scr C(f)= \{ f_0 \in \scr B  \ |  \  f_0 = f - h\circ T + h, \ h \in \scr B   \}. 
\end{align*}
The following proposition tells us that there exists a natural choice in  $\scr C(f)$.

\begin{proposition}\label{lemma-martingale001}
Let $f\in \scr B$ be such that $\nu (f)=0$.
Then there exists a unique 
function $f_0 \in \scr B^+ $ such that $\mathcal L_{\psi} f_0 = 0$
and its extension on $\bb X$ belongs to $\scr C (f)$. 
\end{proposition}
\begin{proof}
First we prove the existence of $f_0$. 
By Proposition 1.2 in \cite{PP90},  
there exists a H\"older continuous function $g$ on $\bb X^+$, whose extension to $\bb X$ is cohomologous to $f$.
As $\nu (f)=0$, we have $\nu^+ (g)=0$.
Now we choose $\alpha \in (0,1)$ close enough to $1$ so that
$\mathcal L_{\psi}$ is bounded on $\scr B_\alpha$ and $g\in \scr B_\alpha$.
By the spectral gap property for the operator $\mathcal L_{\psi}$ (see Theorem 2.2 of \cite{PP90}), 
there exists a H\"older continuous function $h \in \scr B_\alpha$ such that
\begin{align}
h - \mathcal L_{\psi} h = \mathcal L_{\psi} g. 
\end{align}
Since $h = h \mathcal L_{\psi} 1 = \mathcal L_{\psi} (h \circ T)$, 
it follows that 
\begin{align*}
\mathcal L_{\psi} (g - h \circ T + h) = 0. 
\end{align*}
Hence there exists a function $f_0: = g - h \circ T + h \in \scr C(f)$ satisfying $\mathcal L_{\psi} f_0 = 0$.

Now we prove the uniqueness of $f_0$. 
Suppose that there exist $f_0, f_0' \in \scr C(f)$  such that 
$\mathcal L_{\psi} f_0 = \mathcal L_{\psi} f_0' = 0$. 
Then $f_0 - f_0'$ is a coboundary, namely, there exists $h_1 \in \scr B$ such that
$f_0 - f_0' = h_1 \circ T - h_1$. 
As $f_0$ and $f_0'$ depend only on the future coordinates,
it is well known that $h_1$ depends only on the future coordinates.
It follows that $\mathcal L_{\psi} (h_1 \circ T - h_1) = 0$ and hence $\mathcal L_{\psi} h_1 = h_1$. 
This implies that $h_1$ is a constant and therefore $f_0' = f_0$. 
\end{proof}


For any $z \in \bb X^+$, we have defined a probability measure $\nu^-_z$ on the set $\bb X^-_z \subset A^{- \bb N^*}$ of past sequences 
which are compatible with $z$. 
For $n \geq 1$, we let $\scr F_n$ denote the $\sigma$-algebra of subsets of $A^{- \bb N^*}$ generated by the coordinate maps
$y \mapsto (y_{-1}, \ldots, y_{-n})$. 
By convention, we also define $\scr F_0$ as the trivial $\sigma$-algebra. 
And we let $\scr F_n^z$ be the $\sigma$-algebra induced on $\bb X^-_z$. 
The following proposition is a classical result from \cite{PP90}:

\begin{proposition}\label{lemma-martingale002}
Let $f_0\in \mathcal{C} (\bb X^+)$. Then $\mathcal L_{\psi} f_0 = 0$ if and only if for any $z\in \bb X^+$,
the sequence of random variables 
\begin{align*} 
y\in \bb X^-_z \mapsto \check S_n f_0 ( y \cdot z ), \ \ n \geq 0 
\end{align*}
is a martingale on $\bb X^-_z$ equipped with the probability measure $\nu^-_z$
w.r.t.\ the filtration $(\scr F_n^z)_{n \geq 0}$. 
\end{proposition}

\begin{proof}
Denote by $g_n^z: \bb X^-_z \mapsto \bb R$ the function $y \mapsto \check S_n f_0 ( y \cdot z )$. 
Then for $y \in \bb X^-_z$ and $n \geq 1$, we have by the definition of the measure $\nu^-_z$, 
\begin{align*}
\nu^-_z \left( g_n^z \big| \scr F_{n-1}^z  \right)(y)
 = g_{n-1}^z (y) + \mathcal L_{\psi} f_0 (T^{-n} (y \cdot z)). 
\end{align*}
From this identity, the assertion follows. 
\end{proof}

The following result shows that the difference $\check S_n f - \check S_n g$ is bounded,
for $f$ and $g$ in the same cohomology class. 

\begin{lemma} \label{Martingale approx} 
Let $f\in \scr B$ and $g\in\scr C(f) $. Let $h\in \scr B$ be such that $f-g= h\circ T-h$. 
Then, for any $x\in\bb X$ and any $n \geq 1$
we have  
\begin{align*}
\left| \check S_n f ( x ) - \check S_n g ( x ) \right| \leq c=2 \| h \|_{\infty}.
\end{align*}
\end{lemma}

\begin{proof}
Indeed, we have $S_n f - S_n g = h \circ T^{n} - h.$
Since $\check S_n f = (S_n f) \circ T^{-n}$, we get
$\check S_n f - \check S_n g = h - h \circ T^{-n}$, which proves the assertion. 
\end{proof}

\subsection{The H\"older continuity and approximation}

We establish several technical results which will be used in the proofs of the main results of the paper. 
In particular, they will allow us to prove that several convergences hold uniformly in $z \in \bb X^+$. 

\begin{lemma}\label{Lem_sum_Ine}
For any $g \in \scr B$,   
there exists a constant $c_0 >0$ such that for any $n \geq 1$, $z, z' \in \bb X^+$
with $z_0 = z_0'$ and $y \in \bb X^-_z (= \bb X^-_{z'})$, one has
\begin{align}\label{Second_Ine_Sn_z_z}
\left| \check S_n g(y \cdot z) - \check S_n g(y \cdot z')  \right| \leq  c_0 \alpha^{ \omega (z, z')}. 
\end{align}
In particular, 
for any $g \in \scr B$,   
there exists a constant $c_0 >0$ such that for any $n \geq 1$, 
$z, z' \in \bb X^+$
with $z_0 = z_0'$ and $y \in \bb X^-_z (= \bb X^-_{z'})$,  it holds
\begin{align*}
\check S_n g(y \cdot z) \leq \check S_n g(y \cdot z') + c_0. 
\end{align*}
\end{lemma}

\begin{proof}
Since $g \in \scr B$, there exists a constant $L_{g}$ such that for any $x, x' \in \bb X$, 
\begin{align*} 
|g(x) - g(x')| \leq L_{g} \alpha^{ \omega (x,x')},
\end{align*}
where $0< \alpha <1$. 
Hence for any $z, z' \in \bb X^+$
with $z_0 = z_0'$ and $y \in \bb X^-_z$,  and $n \geq 1$, one has
\begin{align*} 
\left| \check S_n g(y \cdot z) - \check S_n g(y \cdot z') \right|  
& \leq \sum_{k=0}^{n-1}  L_{g} \alpha^{n- k + \omega(z,z')}    \nonumber\\
& \leq  L_{g} \frac{\alpha^{1 + \omega(z,z')}}{1- \alpha} = : c_0 \alpha^{w(z, z')}.
\end{align*}
The desired result follows. 
\end{proof}

\begin{corollary}\label{Cor_Holder_minimum}  
For any $g \in \scr B$,   
there exist constants $c_0 >0$ and $\alpha \in (0,1)$ such that for any $n \geq 1$, $z, z' \in \bb X^+$
with $z_0 = z_0'$ and $y \in \bb X^-_z (= \bb X^-_{z'})$, we have
\begin{align}\label{Second_Ine_Sn_z_z_min}
\left| \min_{1 \leq j \leq n} \check S_j g(y \cdot z) - \min_{1 \leq j \leq n} \check S_j g(y \cdot z')  \right|
 \leq  c_0 \alpha^{w(z, z')}.
\end{align}
\end{corollary}

\begin{proof}
By Lemma \ref{Lem_sum_Ine},  there exist constants $c_0 >0$ and $\alpha \in (0,1)$ such that for any $n \geq j \geq 1$,  
\begin{align*}
\min_{1 \leq j \leq n} \check S_j g(y \cdot z) 
\leq \check S_j g(y \cdot z)   \leq   \check S_j g(y \cdot z') +  c_0 \alpha^{w(z, z')}. 
\end{align*}
Taking the minimum over $1 \leq j \leq n$ on the right hand side, we get
\begin{align}\label{Upper_bound_DS}
\min_{1 \leq j \leq n} \check S_j g(y \cdot z) 
   \leq  \min_{1 \leq j \leq n} \check S_j g(y \cdot z') +  c_0 \alpha^{w(z, z')}. 
\end{align}
In the same way, again by Lemma \ref{Lem_sum_Ine},  there exist constants $c_0 >0$ and $\alpha \in (0,1)$ such that for any $n \geq j \geq 1$,  
\begin{align*}
\check S_j g(y \cdot z)   \geq   \check S_j g(y \cdot z') -  c_0 \alpha^{w(z, z')}
 \geq  \min_{1 \leq j \leq n} \check S_j g(y \cdot z') -  c_0 \alpha^{w(z, z')}. 
\end{align*}
Taking the minimum over $1 \leq j \leq n$ on the left hand side, we get
\begin{align}\label{Lower_bound_DS}
\min_{1 \leq j \leq n} \check S_j g(y \cdot z) 
   \geq  \min_{1 \leq j \leq n} \check S_j g(y \cdot z') -  c_0 \alpha^{w(z, z')}. 
\end{align}
Combining \eqref{Upper_bound_DS} and \eqref{Lower_bound_DS}, we conclude the proof of \eqref{Second_Ine_Sn_z_z_min}.
\end{proof}

We will also need the following technical lemma that allows us to approximate the function $g \in \scr B$
by a function $x \mapsto g_m(x)$ on $\bb X$ which only depends  on the coordinates $\{x_k\}_{k \geq -m}$. 

\begin{lemma}\label{Lem_Appro_g}
Let $g \in \scr B$. Then there exist constants $\alpha \in (0,1)$, $c_1>0$ 
and a sequence of H\"older continuous functions $(g_m)_{m \geq 0}$ on $\bb X$ 
which only depend  on the coordinates $\{x_k\}_{k \geq -m}$ such that $\mathcal{L}_{\psi} g_0 = 0$ and for any $m \geq 0$, 
\begin{align} \label{Basic_Inequality}
\sup_{n \geq 1} \left\| \check S_n g_m -  \check S_n g  \right\|_{\infty} \leq  c_1 \alpha^m.
\end{align}
\end{lemma}

\begin{proof}
By Proposition \ref{lemma-martingale001}, there exist $g_0 \in \scr B^+$ and $h \in \scr B$ with $\mathcal{L}_{\psi} g_0 = 0$
and  
\begin{align}\label{Cohomological_Equ_g}
g_0 = g - h\circ T + h. 
\end{align}
Since $h \in \scr B$, there is $\alpha \in (0,1)$ such that $h \in \scr B_\alpha$. 
Then, for any $m\geq 0$, there exists
a H\"older continuous function $h_m$ on $\bb X$ which only depends on the coordinates $\{x_k\}_{k \geq -m}$ such that
\begin{align} \label{Bound_h_hp_001}
\| h-h_m   \|_{\infty} \leq c_1 \alpha^{m},
\end{align}
where $c_1 >0$ is a fixed constant not depending on $p$, 
and by convention $h_0 = 0$.  
We define for any $m \geq 0$,
\begin{align}\label{Bound_gp}
g_m = g_0 + h_{m} \circ T -  h_{m}.
\end{align}
From \eqref{Cohomological_Equ_g}, \eqref{Bound_h_hp_001} and \eqref{Bound_gp}, we get \eqref{Basic_Inequality}.
\end{proof}

\subsection{Duality}

The next duality property will be crucial in the proof of the main results.

\begin{lemma} \label{Corollary_Duality} 
Let $g \in \scr B$. 
For any $n\geq 1$ and any non-negative measurable function 
$F: \bb X \times \bb R \times \bb X \times \bb R \mapsto \mathbb R$, 
we have
\begin{align*} 
& \int_{\mathbb R} \int_{\bb X}  F\left( x, t, T^{-n} x,  t + \check S_n g(x) \right) 
  \mathds 1_{ \left\{ \check \tau_{t}^g (x) > n - 1 \right\}}  \nu(dx) dt
\notag \\
&= \int_{\mathbb R}  \int_{\bb X}  F \big( T^{n} x, u -  S_n g(x), x, u \big) 
   \mathds 1_{ \left\{ \tau_{u}^{-g} (x) > n - 1 \right\}} \nu(dx) du. 
\end{align*}
\end{lemma}

\begin{proof}
By a change of variable $t = u - \check S_n g(x)$, it follows that 
\begin{align*}  
I:  & =   \int_{\mathbb R} \int_{\bb X}  F\left( x, t, T^{-n} x,  t + \check S_n g(x) \right)  
  \mathds 1_{\left\{ t + \check S_{n-1} g(x) \geq 0,  \ldots,   t + \check S_1 g(x) \geq 0  \right\}} \nu(dx) dt
      \notag \\
& =   \int_{\mathbb R} \int_{\bb X}  F \left( x, u - \check S_n g(x),  T^{-n} x,  u \right)    \notag\\
& \qquad\quad  \times  \mathds 1_{ \left\{  u - g(T^{-n} x) \geq 0, 
          \ldots, u - g(T^{-n} x)-  \ldots - g(T^{-2} x) \geq 0 \right\}}   \nu(dx) du. 
\end{align*}
Since the measure $\nu$ is $T^{-1}$-invariant,  we obtain
\begin{align*}
& I = \int_{\mathbb R}  \int_{\bb X} F \big( T^{n} x, u -  S_n g(x), x, u \big)  
   \mathds 1_{ \left\{ u -  S_1 g(x) \geq 0, \ldots, u -  S_{n-1} g(x) \geq 0 \right\} } \nu(dx) du, 
\end{align*}
which ends the proof of the lemma. 
\end{proof}


\section{Harmonicity for dynamical system}\label{sec: harmonic measure}

\subsection{Existence of the harmonic function}

The aim of this section is to prove the existence of a function $V^{f}$ on the state space $\bb R$ 
which we call the harmonic function of $f$ by analogy with the theory developed for Markov chains in \cite{GLL18}.
Our main result is the following theorem:

\begin{theorem}\label{Proposition exist harm func}
Let $f$ be a H\"older continuous function on $\bb X$ such that $f$ is not a coboundary and $\nu (f) = 0$.
Then there exists a unique non decreasing and right continuous function $V^{f}: \bb R \mapsto \bb R_+$ such that
for any continuous compactly supported function  $\vphi$ on $\bb R$,
\begin{align} \label{Harmonic_V_f}
\lim_{n \to \infty} \int_{\bb R} \vphi(t) \int_{\bb X}   S_{n} f (x) \mathds 1_{ \left \{ \tau_t^f(x) >n \right\} }  \nu(dx) dt
= \int_{\bb R} \vphi(t) V^{f}(t) dt.
\end{align}
Besides, there exists a constant $c>0$ such that for any $t\in \bb R$, 
\begin{align}\label{Bound_V_f}
\max \{t - c, 0\} \leq V^f(t) \leq  \max \{t, 0 \} +  c.
\end{align}
\end{theorem}

Note that, the bound \eqref{Bound_V_f} implies that $V^f(t)/t \to 1$ as $t \to \infty$.


The proof of Theorem \ref{Proposition exist harm func} will be given at the end of this section. 
At this point we start by giving an explicit formula for the harmonic function 
in the case where the observable only depends on future coordinates. 
Let $g \in \scr B^+$ with $\nu(g) = 0$ and assume that $g$ is not a coboundary. 
Let $g_0$ be the unique element of $\scr B^+$ such that $\mathcal{L}_{\psi} g_0 = 0$ and $g_0$ is cohomologous to $g$,
as in Proposition \ref{lemma-martingale001}. 
For $z \in \bb X^+$ and $t \in \bb R$, we set
\begin{align}\label{Def_V_g_001}
\check V^{g}(z, t) 
= - \int_{\bb X^-_z} \check{S}_{\check \tau_{t}^{g} (y \cdot z)} g_0 (y \cdot z)  \nu^-_z (dy). 
\end{align}
This integral makes sense. Indeed, first, by Lemma \ref{Lem_Limit_Tau}, 
the stopping time $y \mapsto \check \tau_{t}^{g} (y \cdot z)$ is finite $\nu^-_z$-almost everywhere.
Second, the Birkhoff sum $t + \check{S}_{\check \tau_{t}^{g} (y \cdot z)} g (y \cdot z)$ 
takes values in the interval $[ -\|g\|_{\infty}, 0]$ 
when $t$ is non-negative, and in the interval $[t-\|g\|_{\infty}, 0]$ when $t$ is negative. 
Third, by Lemma \ref{Martingale approx}, the difference of the Birkhoff sums for $g$ and $g_0$ is uniformly bounded.  

The function $\check V^{g}(z, \cdot)$ plays a crucial role in proving conditioned limit theorems for 
products of random matrices and more generally for Markov chains, see \cite{GLP17} and \cite{GLL18}. 
From the results of  \cite{GLL18} it follows that $\check V^{g}(z, \cdot)$ has the following harmonicity property. 

\begin{lemma}\label{Lem_Harmonic_equation}
Let $g$ be in $\scr B^+$ such that $\nu^+(g) = 0$ and $g$ is not a coboundary. 
Then for any $(z,t) \in \bb X^+ \times \bb R$, we have 
\begin{align}\label{Harmonicity_Property_V}
\check V^g(z, t) = \sum_{z' \in \bb X^+: T(z') = z} e^{- \psi(z')} \mathds 1_{\{ t + g(z') \geq 0 \}}  \check V^g(z', t + g(z')). 
\end{align}
\end{lemma}

The proof of the existence of the harmonic function $\check V^g$ given in \cite{GLL18} is rather difficult. 
In the case of the subshift of finite type (since the jumps are bounded) it is possible to give a much shorter direct proof, which is not included because of the space limitations.

We shall  extend the definition of $\check V^{g}(z, \cdot)$ to the case of any function $g \in \scr B$,
that is, the case of a function $g$ that depends on both the past and the future coordinates. 
We will use the following technical assertion:

\begin{lemma}\label{Lem_Limit_Tau}
Let $g \in \scr B$ such that $\nu(g) = 0$ and $g$ is not a coboundary with respect to T. 
Then, for any $t \in \bb R$, it holds uniformly in $z \in \bb X^+$ that 
\begin{align*}
\lim_{n \to \infty} \nu^-_z \left( y \in \bb X^-_z:  \check \tau_t^g(y \cdot z) >n \right) = 0. 
\end{align*}
\end{lemma}

\begin{proof}
Let $c_0 >0$ be as in Lemma \ref{Lem_sum_Ine}. 
By Lemma \ref{Lem_tau_finite}  and Fubini's theorem, 
for any $a \in A$, 
there exists $z' \in \bb X^+$ such that $z'_0 = a$ and the function
$y \mapsto \tau_{t+c_0}^g(y \cdot z')$ on $\bb X^-_{z'}$ is finite $\nu^-_{z'}$-almost everywhere. 
Then for any $z \in \bb X^+$ with $z_0 = a$, we have 
\begin{align*}
\left\{ y \in \bb X^-_z: \check \tau_t^g(y \cdot z) >n \right\}  
\subseteq  \left\{ y \in \bb X^-_{z'}: \check \tau_{t+c_0}^g(y \cdot z') >n \right\}. 
\end{align*}
From Lemma \ref{Lem_Absolute_Contin}, we get that for some constant $c>0$, 
\begin{align*}
\nu^-_{z} \left\{ y \in \bb X^-_z: \check \tau_t^g(y \cdot z) >n \right\}
\leq c  \nu^-_{z'}\left\{ y \in \bb X^-_{z'}: \check \tau_{t+c_0}^g(y \cdot z')  \right\}.  
\end{align*}
Thus, the lemma follows from the fact that 
$\nu^-_{z'}\left\{ y \in \bb X^-_{z'}: \check \tau_{t+c_0}^g(y \cdot z') >n \right\}$
converges to $0$ as $n \to \infty$. 
\end{proof}

Now we give an alternative definition of the function $\check V^{g}(z, \cdot)$ for $g \in \scr B^+$,
where the key point is that in this case, the function $y \mapsto \check \tau_t^{g} (y \cdot z)$ is a stopping time 
with respect to the filtration $\{ \scr{F}_k^z \}_{k \geq 0}$.

\begin{lemma}\label{Lem_def_V}
Let $g \in \scr B^+$ with $\nu(g) = 0$ and assume that $g$ is not a coboundary.
Let $g_0$ be the unique element of $\scr B^+$ such that $\mathcal{L}_{\psi} (g_0) = 0$ and $g_0$ is cohomologous to $g$. 
Then, for any $z \in \bb X^+$ and $t \in \bb R$, we have
\begin{align} \label{limit V}
\check V^{g}(z, t) 
& = \lim_{n \to \infty} \int_{\bb X^-_z} \left( t + \check S_{n} g (y \cdot z) \right)
    \mathds 1_{ \left\{ \check \tau_t^{g} (y \cdot z) >n  \right\} }  \nu^-_z(dy)   \nonumber\\
& = \lim_{n \to \infty} \int_{\bb X^-_z} \check S_{n} g_0 (y \cdot z) 
    \mathds 1_{ \left\{ \check \tau_t^{g} (y \cdot z) >n  \right\} }  \nu^-_z(dy).  
\end{align}
In addition, there is a constant $c>0$ such that, for any $z \in \bb X^+$ and $t \in \bb R$,
\begin{align}\label{Uniform_bound_V}
\int_{\bb X^-_z} \check S_{n} g (y \cdot z)
    \mathds 1_{ \left\{ \check \tau_t^{g} (y \cdot z) >n  \right\} }  \nu^-_z(dy)
 \leq \max\{t, 0\} + c, 
\end{align}
for any $z \in \bb X^+$ and $t \in \bb R_+$,
\begin{align}\label{Bound_Har_V_g0}
t - c \leq  \check V^{g}(z, t) \leq t + c, 
\end{align}
and for any $z \in \bb X^+$ and $t < -c$, it holds that $\check V^{g_0}(z, t) = 0$. 

Moreover, for any $z \in \bb X^+$, the function $\check V^{g}(z, \cdot)$ is  non decreasing on $\mathbb{R}$. 
\end{lemma}

\begin{proof}
By Lemmas \ref{Martingale approx} and \ref{Lem_Limit_Tau}, 
the limits in \eqref{limit V} coincide  since the difference of the Birkhoff sums for $g$ and $g_0$ is uniformly bounded and 
 $\nu^-_z(y \in \bb X^-_z: \check \tau_t^{g} (y \cdot z) >n) \to 0$ as $n \to \infty$. 
By the optional stopping theorem, 
\begin{align*}
& \int_{\bb X^-_z} \check S_{n} g_0 (y \cdot z) \mathds 1_{ \left\{ \check \tau_t^{g} (y \cdot z) >n \right\} } \nu^-_z(dy)  \nonumber\\
& = \int_{\bb X^-_z} \check S_{n} g_0 (y \cdot z) \nu^-_z(dy)  
 - \int_{\bb X^-_z} \check S_{n} g_0 (y \cdot z) \mathds 1_{ \left\{ \check \tau_t^{g} (y \cdot z) \leq n  \right\} } \nu^-_z(dy)
  \nonumber\\
& =  - \int_{\bb X^-_z} \check S_{\check \tau_t^{g} (y \cdot z)} g_0 (y \cdot z) 
    \mathds 1_{ \left\{ \check \tau_t^{g} (y \cdot z) \leq n  \right\} } \nu^-_z(dy). 
\end{align*}
By Lemma \ref{Lem_Limit_Tau}, we have $\check \tau_t^{g} (y \cdot z) < + \infty$ for $\nu^-_z$-almost every $y \in \bb X^-_z$.
Using the Lebesgue convergence theorem,
we obtain \eqref{limit V} as well as \eqref{Uniform_bound_V}. 

Since for any $z \in \bb X^+$ and $t \in \bb R_+$, 
the function $y \mapsto t + \check S_{\check \tau_t^{g} (y \cdot z)} g_0 (y \cdot z)$ takes values in $[-\|g\|_{\infty} - c, c]$, 
where $c>0$ is from Lemma \ref{Martingale approx}, 
we get $\check V^{g}(z, t) \in [t-c, t + \|g\|_{\infty} + c ]$. 
Besides, if $t < - \|g\|_{\infty}$, we get $\check \tau_{t}^{g} (y \cdot z) = 1$ everywhere, for all $z \in \bb X^+$, 
thus by \eqref{Def_nu_kz} and \eqref{Def_OperatorP} we have 
\begin{align*}
\check V^{g}(z, t) = \int_{A^1_z} \left( g_0(y_{-1} \cdot z) \right) \nu^1_z(dy_{-1}) = \mathcal{L}_{\psi} (g_0) (z) = 0. 
\end{align*}

As $\check \tau_{t_1}^{g} \leq \check \tau_{t_2}^{g}$ for any $t_1 \leq t_2$,
and $t_2 + \check S_n g \geq 0$ on the set $\{\check \tau_{t_2}^{g} >n\}$, it follows that
\begin{align*}
& \int_{\bb X^-_z} \left( t_1 + \check S_n g (y \cdot z) \right)  
  \mathds 1_{ \left\{ \check \tau_{t_1}^{g} (y \cdot z) > n  \right\} } \nu^-_z(dy) \nonumber\\
& \leq \int_{\bb X^-_z} \left( t_2 + \check S_n g (y \cdot z) \right)  
   \mathds 1_{ \left\{ \check \tau_{t_2}^{g} (y \cdot z) > n  \right\} } \nu^-_z(dy).
\end{align*}
Letting $n\to \infty$ yields that the function $\check V^{g}(z, \cdot)$ is non decreasing on $\mathbb{R}$. 
\end{proof}

By using Lemma \ref{Lem_def_V}, we can now give a definition of $\check{V}^g$ 
for a function $g$ only depending on finitely many negative coordinates. 

\begin{lemma}\label{Lem_Harmonic_Function_V}
Let $g \in \scr B$ such that $\nu(g) = 0$ and $g$ is not a coboundary. 
Assume that $g$ only depends on $m$ negative coordinates for some $m \geq 0$.
In other words, the function $h = g \circ T^{m} \in \scr B^+$.
Then, for any $t \in \bb R$, we have uniformly in $z \in \bb X^+$, 
\begin{align*}
\lim_{n \to \infty} \int_{\bb X^-_z} \check{S}_n g (y \cdot z) 
 \mathds 1_{ \left\{ \check{\tau}_t^g (y \cdot z) >n \right\} } \nu^-_z(dy) 
 = \mathcal{L}_{\psi}^m \left( \check{V}^h (\cdot, t) \right)(z). 
\end{align*} 
\end{lemma}

Let $g$ and $h$ be as in Lemma \ref{Lem_Harmonic_Function_V}. We set for $z \in \bb X^+$ and $t \in \bb R$,  
\begin{align*}
\check{V}^g (z,t) = \mathcal{L}_{\psi}^m \left( \check{V}^h (\cdot, t) \right)(z). 
\end{align*}
Lemma \ref{Lem_Harmonic_Function_V} implies that this notation is coherent with that introduced in \eqref{Def_V_g_001}.  

\begin{proof}[Proof of Lemma \ref{Lem_Harmonic_Function_V}]
By conditioning over the $m$ first coordinates of $y$, we get for $n \geq 0$,  
\begin{align*}
& \int_{\bb X^-_z} \check{S}_n g (y \cdot z) 
 \mathds 1_{ \left\{ \check{\tau}_t^g (y \cdot z) >n \right\} } \nu^-_z(dy) \notag\\
 & = \sum_{a \in A_z^m} \exp(- S_m \psi(a \cdot z))   
  \int_{\bb X^-_{a \cdot z}}   \check{S}_n  g (  (y \cdot a)  \cdot z) 
 \mathds 1_{ \left\{ \check{\tau}_t^{g} ( (y \cdot a)  \cdot z) > n \right\} }  \nu^-_{a \cdot z}(dy)   \notag\\
 & = \sum_{a \in A_z^m} \exp(- S_m \psi(a \cdot z))   
  \int_{\bb X^-_{a \cdot z}}   \check{S}_n  h (T^{-m}  (y \cdot a)  \cdot z) 
 \mathds 1_{ \left\{ \check{\tau}_t^{h \circ T^{-m}} ( (y \cdot a)  \cdot z) > n \right\} }  \nu^-_{a \cdot z}(dy)   \notag\\
& = \sum_{a \in A_z^m} \exp(- S_m \psi(a \cdot z))   
  \int_{\bb X^-_{a \cdot z}}   \check{S}_{n} h ( y \cdot (a \cdot z))   
   \mathds 1_{ \left\{ \check{\tau}_t^h (y \cdot (a \cdot z)) >n \right\} }   \nu^-_{a \cdot z}(dy), 
\end{align*} 
where we have used the relations $(y \cdot a) \cdot z = T^m (y \cdot (a \cdot z))$
and $\check{\tau}_t^{h \circ T^{-m}} = \check{\tau}_t^{h} \circ T^{-m}$.  
The conclusion now follows from Lemma \ref{Lem_def_V} and the definition of the transfer operator $\mathcal L_{\psi}^m.$ 
\end{proof}

We will prove that the convergence in Lemma \ref{Lem_Harmonic_Function_V} 
holds in a weak sense for every function $g \in \scr B$.
The key step to prove Theorem \ref{Proposition exist harm func} is the following technical lemma which shows that
the convergence of Lemma \ref{Lem_Harmonic_Function_V} holds for all functions $g \in \scr B$ in a weak sense. 

\begin{lemma}\label{Lem_harm_func_001}
Assume that $g\in \scr B$ is not a coboundary w.r.t.\ to $T$ and $\nu (g)=0$. 
Then, for any continuous compactly supported function  $\vphi$ on $\bb R$,  
uniformly in $z \in \bb X^+$, the following limit exists and is finite: 
\begin{align*} 
\lim_{n \to \infty} \int_{\bb R} \vphi(t) \int_{\bb X^-_z} 
     \check S_{n} g (y \cdot z) \mathds 1_{ \left\{ \check \tau_t^g(y \cdot z) >n \right\} } \nu^-_z(dy)  dt. 
\end{align*}
\end{lemma}


\begin{proof}
Assume that $g\in \scr B$.  Let $(g_m)_{m \geq 0}$, $c_1>0$ and $\alpha \in (0,1)$ be as in Lemma \ref{Lem_Appro_g}. 
Set 
\begin{align*}
W_n (z, t) 
= \int_{\bb X^-_z} \left( t +  \check S_{n} g (y \cdot z) \right) 
   \mathds 1_{ \left\{ \check \tau_t^g(y \cdot z) >n \right\} } \nu^-_z(dy) 
\end{align*}
and 
\begin{align*}
W_{n,m} (z, t) 
= \int_{\bb X^-_z} \left( t +  \check S_{n} g_m (y \cdot z) \right)
     \mathds 1_{ \left\{ \check \tau_t^{g_m} (y \cdot z) >n \right\} } \nu^-_z(dy). 
\end{align*}
By \eqref{Basic_Inequality}, we have the inclusions
\begin{align*}
\left\{  \check \tau_{t-2c_1  \alpha^{m}}^{g_{m}} > n  \right\}
\subseteq    \left\{  \check \tau_t^{g} > n \right\}   
\subseteq  \left\{  \check \tau_{t+2c_1  \alpha^{m}}^{g_{m}} > n \right\}, 
\end{align*}
which imply that
\begin{align}\label{Pf_W_n_Wnp}
W_{n,m} (z, t - 2 c_1 \alpha^m) \leq  W_{n} (z, t)  \leq  W_{n,m} (z, t + 2 c_1 \alpha^m). 
\end{align}
In the same way, we have
\begin{align}\label{Pf_Wnp_02}
W_{n,m} (z, t) \leq  W_{n, 0} (z, t + 2 c_1)  \leq  c_2 + 2 \max \{t, 0\}, 
\end{align}
where the last bound follows from \eqref{Uniform_bound_V}. 

By lemma \ref{Lem_Harmonic_Function_V}, for fixed $m \geq 0$, as $n \to \infty$, 
the function $W_{n,m} (z, t)$ converges to $\check{V}^{g_m} (z, t)$, uniformly in $z \in \bb X^+$. 
From \eqref{Pf_W_n_Wnp} we get
\begin{align*}
\check{V}^{g_m} (z, t - 2 c_1 \alpha^m) 
\leq \liminf_{n \to \infty}  W_{n} (z, t) 
\leq \limsup_{n \to \infty}  W_{n} (z, t)
\leq \check{V}^{g_m} (z, t + 2 c_1 \alpha^m). 
\end{align*}
Now we have 
\begin{align*}
& \int_{\bb R} \vphi(t) \left[ \check{V}^{g_m} (z, t + 2 c_1 \alpha^m) - \check{V}^{g_m} (z, t - 2 c_1 \alpha^m) \right] dt  \nonumber\\
& =  \int_{\bb R} \left[ \vphi(t - 2 c_1 \alpha^m) - \vphi(t + 2 c_1 \alpha^m) \right] \check{V}^{g_m} (z, t) dt. 
\end{align*}
Using \eqref{Pf_Wnp_02} and Lemma \ref{Lem_Harmonic_Function_V}, we have that $\check{V}^{g_m} (z, t) \leq c_2 + \max \{t, 0\}$. 
As $\vphi$ is continuous on $\bb R$ with compact support, 
by the Lebesgue dominated convergence theorem, we get that uniformly in $z \in \bb X^+$, 
\begin{align*}
\lim_{m \to \infty} 
\int_{\bb R} \vphi(t) \left[ \check{V}^{g_m} (z, t + 2 c_1 \alpha^m) - \check{V}^{g_m} (z, t - 2 c_1 \alpha^m) \right] dt 
= 0. 
\end{align*}
This tells us that $\int_{\bb R} \vphi(t) W_n(z,t) dt$ has a uniform limit as $n \to \infty$. 
\end{proof}

We will use the previous lemma to build a function $\check V^g(z,t)$.
The existence of this function will be deduced from the following elementary fact from the  theory of distributions. 

\begin{lemma}\label{Lem_Conver_Func}
Let $(V_n)_{n \geq 1}$ be a sequence of non decreasing functions on $\bb R$.
Assume that for every continuous compactly supported function $\vphi$ on $\bb R$, 
the sequence $\int_{\bb R} V_n(t) \vphi (t) dt$ admits a finite limit. 
Then there exists a unique right continuous and non decreasing function $V$ on $\bb R$ such that 
for any continuous compactly supported function $\vphi$, we have
\begin{align*}
\lim_{n \to \infty} \int_{\bb R} V_n(t) \vphi (t) dt = \int_{\bb R} V(t) \vphi (t) dt. 
\end{align*}
\end{lemma}

Now we construct the function $\check V^g(z,t)$ for any $g \in \scr B$. 

\begin{lemma}\label{Lemma exist of harm func(copy)}
Assume that $g\in \scr B$ is not a coboundary w.r.t.\ to $T$ and $\nu (g)=0$.
Then, for any $z \in \bb X^+$, 
there exists a unique non decreasing  and  right continuous function $\check V^{g}(z, \cdot)$ on $\bb R$ such that

\noindent
1. For any continuous compactly supported function  $\vphi$ on $\bb R$,  
uniformly in $z \in \bb X^+$,
\begin{align} 
\lim_{n \to \infty} \int_{\bb R} \vphi(t) \int_{\bb X^-_z} 
   \check S_{n} g (y \cdot z) \mathds 1_{ \left\{ \check \tau_t^g(y \cdot z) >n \right\} } \nu^-_z(dy)  dt  
  = \int_{\bb R} \vphi(t) \check V^{g}(z, t) dt.    \label{limit V002}
\end{align}

\noindent
2. For any continuous compactly supported function  $\vphi$ on $\bb R$, 
 the mapping $z \mapsto \int_{\bb R} \vphi(t) \check V^{g}(z,  t) dt$ is continuous on $\bb X^+$. 



\noindent
3. There exists a constant $c >0$ such that for any $z \in \bb X^+$ and $t \in \bb R_+$, 
\begin{align}\label{Bound_Vg_01}
t-c \leq \check V^{g}(z, \cdot) \leq t + c. 
\end{align}
In addition, for any $z \in \bb X^+$ and $t \leq -c$, we have $\check V^{g}(z, t) = 0$. 

\end{lemma}

By Lemma \ref{Lem_def_V}, in the case $g \in \scr B^+$, the notation $\check V^{g}(z, \cdot)$ 
is coherent with the one in \eqref{Def_V_g_001}.

\begin{proof}[Proof of Lemma \ref{Lemma exist of harm func(copy)}]
%

Fix $z \in \bb X^+$. 
By Lemmas \ref{Lem_Limit_Tau} and \ref{Lem_harm_func_001}, the following limit exists: 
for any continuous compactly supported function $\vphi$ on $\bb R$, 
\begin{align} 
\lim_{n \to \infty} \int_{\bb R} \vphi(t) \int_{\bb X^-_z}  
    \left( t + \check S_{n} g (y \cdot z) \right)  \mathds 1_{ \left\{ \check \tau_t^g(y \cdot z) >n  \right\} } \nu^-_z(dy)  dt. 
\end{align}
For $t \in \bb R$, set 
\begin{align}\label{Def_Vn_g_001}
\check V_n^g (z, t) = \int_{\bb X^-_z}  
    \left( t + \check S_{n} g (y \cdot z) \right)  \mathds 1_{ \left\{ \check \tau_t^g(y \cdot z) >n  \right\} } \nu^-_z(dy). 
\end{align}
Then the function $\check V_n^g (z, \cdot)$ is non decreasing on $\bb R$. 
By Lemma \ref{Lem_Conver_Func}, there exists a unique non decreasing and right continuous function $\check V^g (z, \cdot)$ on $\bb R$
such that for any continuous function $\vphi$ on $\bb R$ with compact support, 
\begin{align*}
\lim_{n \to \infty} \int_{\bb R} \check V_n^g (z, t) \vphi (t) dt = \int_{\bb R} \check V^g (z, t) \vphi (t) dt. 
\end{align*}

Note that for $t < - \|g\|_{\infty}$,
we have $\check \tau^{g}_t = 1$ everywhere. Hence $\check V^{g}(z, t) = 0$ for $t \leq - c$. 

We now prove \eqref{Bound_Vg_01}. 
By Proposition \ref{lemma-martingale001}, there exists $g_0 \in \scr B^+$ 
such that $\mathcal{L}_{\psi} (g_0) = 0$ and $g$ is cohomologous to $g_0$. 
By Lemma \ref{Lem_Appro_g}, we can choose a constant $c>0$ large enough such that for any $n \geq 1$, 
it holds that $\| \check S_n g - \check S_n g_0 \|_{\infty} \leq c$. 
By Lemmas \ref{Lem_Limit_Tau} and \ref{Lem_harm_func_001}, we have, 
for any continuous non negative function $\vphi$ on $\bb R$ with compact support,  
\begin{align*}
\lim_{n \to \infty} \int_{\bb R} \vphi (t)  \int_{\bb X^-_z}
    \left( t + c + \check S_{n} g_0 (y \cdot z) \right) \mathds 1_{ \left\{ \check \tau_t^g(y \cdot z) >n \right\} } \nu^-_z(dy)  dt  
  = \int_{\bb R} \check V^g (z, t) \vphi (t) dt. 
\end{align*}
Note that from Lemma \ref{Lem_Limit_Tau}, we have $\nu^-_z( y \in \bb X^-_z: \check \tau_t^g(y \cdot z) >n) \to 0$ as $n \to \infty$. 
As we have the following inclusion: for any $t \in \bb R$, 
\begin{align*}
\left\{ \check \tau_{t-c}^{g_0} >n  \right\} 
\subset  \left\{ \check \tau_t^g >n  \right\} 
\subset  \left\{ \check \tau_{t+c}^{g_0} >n  \right\},  
\end{align*}
and as $t + c + \check S_{n} g_0 \geq  0$ on the set $\{ \check \tau_{t+c}^{g_0} >n \}$,
we get 
\begin{align*}
\int_{\bb R} \check V^{g_0} (z, t-c) \vphi (t) dt
\leq \int_{\bb R} \check V^g (z, t) \vphi (t) dt
\leq \int_{\bb R} \check V^{g_0} (z, t+c) \vphi (t) dt. 
\end{align*}
Since this holds for any continuous non-negative test function $\vphi$ on $\bb R$,
we obtain 
\begin{align}\label{Inequality_Vg_Vg0}
\check V^{g_0} (z, t-c) \leq  \check V^g (z, t)  \leq  \check V^{g_0} (z, t+c). 
\end{align}
This, together with Lemma \ref{Bound_Har_V_g0}, concludes the proof of \eqref{Bound_Vg_01}. 

We now want to prove the continuity in $z \in \bb X^+$ of the function $z \mapsto \int_{\bb R} \vphi(t) \check V^g (z,t) dt$. 
To this aim, we establish a uniform bound for the quantity $\check V_n^g (z,t)$ defined in \eqref{Def_Vn_g_001}. 
Indeed, as usual, we have $\check V_n^g (z,t) \leq \check V_n^{g_0} (z,t+c)$. 
Now the optional stopping theorem gives 
\begin{align*}
 \check V_n^{g_0} (z,t)   
&  = t \nu^-_z\left(y \in \bb X^-_z:  \check \tau_{t}^{g_0}(y \cdot z)  > n \right)   \notag\\
& \quad  - \int_{\bb X^-_z}  \check S_{\check \tau_{t}^{g_0} (y \cdot z)} g_0 (y \cdot z)  
      \mathds 1_{ \left\{  \check \tau_{t}^{g_0} (y \cdot z)  \leq n  \right\} } \nu^-_z(dy)  \nonumber\\
&  \leq  |t| + |t| + \|g_0\|_{\infty}  = 2|t| + \|g_0\|_{\infty}. 
\end{align*}
From \eqref{Inequality_Vg_Vg0} we get 
\begin{align}\label{Bound_Vng_02}
\check V_n^{g} (z,t) \leq 2 |t| + 2c + \|g_0\|_{\infty}.
\end{align}

It remains to prove that for any continuous compactly supported function  $\vphi$ on $\bb R$, 
 the mapping $z \mapsto \int_{\bb R} \vphi(t) \check V^{g}(z,  t) dt$ is continuous on $\bb X^+$. 
It suffices to prove that for any $n \geq 1$, 
the mapping $z \mapsto \int_{\bb R} \vphi(t) \check V_n^{g}(z,  t) dt$ is continuous on $\bb X^+$. 

A priori, for fixed $t \in \bb R$, the function $z \mapsto \check V_n^{g}(z,  t)$ is not continuous. 
Nevertheless, we claim that it satisfies the following weak continuity property:
for $\ee >0$, there exists $k \in \bb N$ such that for any $z, z' \in \bb X^+$ with $w(z, z') \geq k$
we have
\begin{align*}
\check V_n^{g}(z, t-\ee)  \leq  \check V_n^{g}(z', t) \leq \check V_n^{g}(z, t + \ee). 
\end{align*}
Indeed, this follows from the inequality \eqref{Second_Ine_Sn_z_z} in Lemma \ref{Lem_sum_Ine}.
This, together with the bound \eqref{Bound_Vng_02} and the uniform continuity of the function $\vphi$, implies that 
the mapping $z \mapsto \int_{\bb R} \vphi(t) \check V_n^{g}(z,  t) dt$ is continuous on $\bb X^+$. 
\end{proof}

The previous statements can be summarized as follows: 

\begin{theorem}\label{Proposition exist harm func_aaa}
Let $g$ be a H\"older continuous function on $\bb X$ such that  $\nu (g) = 0$ and $g$ is not a coboundary.
Then there exists a unique non decreasing and right continuous function $\check V^{g}: \bb R \mapsto \bb R_+$ with the following properties:

\noindent 1. For any continuous compactly supported function  $\vphi$ on $\bb R$,
\begin{align} \label{Harmonic_V_f_aaa}
\lim_{n \to \infty} \int_{\bb R} \vphi(t) \int_{\bb X}  \check S_{n} g (x) \mathds 1_{ \left \{ \check \tau_t^g(x) >n \right\} }  \nu(dx) dt
= \int_{\bb R} \vphi(t) \check V^{g}(t) dt.
\end{align}

\noindent  2. There exists a constant $c>0$ such that for any $t\in \bb R$ it holds 
\begin{align}\label{Bound_V_f_aaa}
\max \{t - c, 0\} \leq \check V^g(t) \leq  \max \{t, 0 \} +  c.
\end{align}

\end{theorem}

\begin{proof}
Let $g \in \scr B$. For $t \in \bb R$, we set
\begin{align*}
\check V^g (t) = \int_{\bb X^+}  \check V^g (z, t) \nu^+(dz). 
\end{align*}
Then the points 1 and 2 of Theorem \ref{Proposition exist harm func_aaa} 
follow from \eqref{limit V002} and \eqref{Bound_Vg_01} in  Lemma \ref{Lemma exist of harm func(copy)}, respectively. 
\end{proof}

\begin{proof}[Proof of Theorem \ref{Proposition exist harm func}]
It is easy to see that Theorem \ref{Proposition exist harm func} is equivalent to Theorem \ref{Proposition exist harm func_aaa}
for the reversed dynamics, i.e.\ 
by replacing $f$ with $g = f \circ T^{-1} \circ \iota = f \circ \iota \circ T$, and $\nu$ with $\iota_* \nu$. 
\end{proof}

\subsection{Properties of the harmonic function}

The goal of this section is to give some additional properties of the harmonic function $\check V^{g}$
which will be necessary for the proof of Theorem \ref{Thm-exit-time001}. 
We start with a continuity result on the cohomology class of the function $g$. 

\begin{lemma}\label{Lem_Continuity_01}
Let $g \in \scr B$ with $\nu(g) = 0$. Assume that $g$ is not a coboundary. 
Let $\alpha \in (0,1)$ and $(h_n)_{n \geq 0}$ be a sequence of element of $\scr B_{\alpha}$ that converges to $0$ 
with respect to the H\"{o}lder norm $\|\cdot \|_{\alpha}$. 
For $n \geq 0$, set $g_n = g + h_n \circ T - h_n$.
Then, there exists a constant $c >0$ such that for any $n \geq 0$, $z \in \bb X^+$ and $t \in \bb R$, one has
\begin{align}\label{Bound_V_gn_1000}
\check V^{g_n} (z,t) \leq  \max\{t, 0 \} + c. 
\end{align}
Moreover, for any continuous compactly supported function  $\vphi$ on $\bb R$, we have,
uniformly in $z \in \bb X^+$, 
\begin{align}\label{Func_V_200}
\lim_{n \to \infty} \int_{\bb R} \vphi(t) \check V^{g_n} (z,t) dt = \int_{\bb R} \vphi(t) \check V^{g} (z,t) dt. 
\end{align} 
\end{lemma}

\begin{proof}
The bound \eqref{Bound_V_gn_1000} follows from \eqref{Bound_Vg_01} and the relation $g_n = g + h_n \circ T - h_n$. 
The construction of the function $\check V^{g}$ in \eqref{Func_V_200} 
 can be performed in the same way as in Lemmata \ref{Lem_harm_func_001} and  \ref{Lemma exist of harm func(copy)}. 
\end{proof}

We can also describe how the function $\check V^{g}$ behaves when the function $g$ is shifted by the dynamics. 

\begin{lemma}\label{Lem_property_V_nn}
Let $g \in \scr B$ with $\nu(g) = 0$. Assume that $g$ is not a coboundary. 
Then, for any $z \in \bb X^+$ and $t \in \bb R$, we have 
\begin{align*}
\check V^{g \circ T^{-1}} (z,t)  = \mathcal{L}_{\psi} \left( V^g(\cdot, t) \right)(z). 
\end{align*}
\end{lemma}

\begin{proof}
By Lemma \ref{Lemma exist of harm func(copy)}, for any continuous compactly supported function  $\vphi$ on $\bb R$, we have
\begin{align*}
& \int_{\bb R} \vphi(t) \check V^{g \circ T^{-1}} (z,t) dt  
 = \lim_{n \to \infty} \int_{\bb R} \vphi(t) \int_{\bb X^-_z}
     \check S_{n} g (T^{-1}(y \cdot z)) \mathds 1_{ \{ \check \tau_t^{g \circ T^{-1}}(y \cdot z) >n \} } \nu^-_z(dy) dt.   
\end{align*}
By conditioning on the coordinate $y_{-1}$, we get
\begin{align*}
& \int_{\bb X^-_z}
     \check S_{n} g (T^{-1}(y \cdot z)) \mathds 1_{ \big\{ \check \tau_t^{g \circ T^{-1}}(y \cdot z) >n \big\} } \nu^-_z(dy)  \nonumber\\
& = \int_{A^1_z}  \int_{\bb X^-_z}
     \check S_{n} g (w \cdot (y_{-1} \cdot z) )  \mathds 1_{ \left\{ \check \tau_t^{g}( w \cdot (y_{-1} \cdot z ) ) >n \right\} } \nu^-_z(dy)
     \nu^1_z (dy_{-1}). 
\end{align*}
Again by Lemma \ref{Lemma exist of harm func(copy)}, we obtain
\begin{align*}
\check V^{g \circ T^{-1}} (z,t) = \int_{A^1_z} \check V^{g} (y_{-1} \cdot z, t)  \nu^1_z(dy_{-1})
 = \mathcal{L}_{\psi} \left( V^g(\cdot, t) \right)(z),  
\end{align*}
as desired. 
\end{proof}

\subsection{The harmonic measure and the proof of Theorem \ref{Thm_Radon_Measure}}  \label{Sect_harmonic_measure}


In the case when $g$ depends only on the future ($g \in \scr B^+$), 
the function $\check V^g$ satisfies the harmonicity equation \eqref{Harmonicity_Property_V}.
In general when $g$ depends also on the past, this property may not hold. 
It turns out that equation \eqref{Harmonicity_Property_V} 
can be reinterpreted as a kind of invariance property of a certain Radon measure, which we will introduce at the end of this section.
Indeed, we have: 

\begin{lemma}\label{explanation-invariance001}
Let $g$ be in $\scr B^+$ and let $V$ be a locally integrable non-negative function on $\bb X^+ \times \bb R$. 
Then the following are equivalent: 
\begin{enumerate}
\item
For $\nu^+ \otimes dt$ almost every $(z, t)$ in  $\bb X^+ \times \bb R$, 
we have 
\begin{align*}
V(z,t) =  \sum_{z' \in \bb X^+: T(z') = z} e^{- \psi(z')} \mathds 1_{\{ t + g(z') \geq 0 \}}  V (z', t + g(z')). 
\end{align*}

\item
For any continuous compactly supported function $\varphi$ on $\bb X^+ \times \bb R$, 
we have 
\begin{align}\label{Harmonic_rho_cc}
\int_{\bb X^+ \times \bb R} \vphi(z, t)  V(z, t) \nu^+(dz) dt 
= \int_{\bb X^+} \int_{0}^{\infty}  \vphi(Tz, t - g(z))  V(z, t)  \nu^+(dz) dt.  
\end{align}

\end{enumerate}
\end{lemma}



\begin{proof}
The proof is a direct computation. Indeed, for any continuous compactly supported function $\varphi$ on $\bb X^+ \times \bb R$, 
 by a change of variable, the right hand side of \eqref{Harmonic_rho_cc} 
can be written as 
\begin{align*}
\int_{\bb R} \int_{\bb X^+}  \vphi(Tz, t)  \mathds 1_{\{ t + g(z) \geq 0 \}} V(z, t + g(z)) \nu^+(dz) dt. 
\end{align*}
As $\nu^+$ is $\mathcal L_{\psi}$ invariant, 
by using \eqref{transl-invar001}, 
we get for $t \in \bb R$, 
\begin{align*}
& \int_{\bb X^+}  \vphi(Tz, t)  \mathds 1_{\{ t + g(z) \geq 0 \}} V(z, t + g(z)) \nu^+(dz)  \notag\\
& =  \int_{\bb X^+} \varphi(z,t)  \left( \sum_{Tz' = z}  e^{- \psi(z')}  \mathds 1_{\{ t + g(z') \geq 0 \}} V(z', t + g(z'))  \right)   \nu^+(dz).  
\end{align*}
This proves the lemma. 
\end{proof}

We will now show that the functions $\check V^{g}$ and $V^{g}$ can be seen as the densities with respect to the Lebesgue measure on $\bb R$
of the projections on $\bb R$ 
of certain natural Radon measures $\check\mu^g$ and $\mu^{g}$ on $\bb X \times \bb R$,
which satisfy an invariance property similar to \eqref{Harmonic_rho_cc}.  
Those measures will play a key role in the statement of the conditioned local limit theorem.
The purpose of this subsection is to build them. This construction will follow the same lines as the one of the harmonic functions.
We will first use Markov chain arguments to define these objects when $g \in \scr B^+$ and then use approximation arguments 
to extend the definition to the general case. 

We first assume that $g$ is in $\scr B^+$. In that case, for $(z, t) \in \bb X^+ \times \bb R$ with $\check V^{g} (z,t) >0$, 
let us introduce a Borel probability measure $\check\mu^{g, -}_{z, t}$ on $\bb X^-_z$. 
To do this, for $n \geq 1$, let $A^n_z$ be as in the definition \eqref{Def_nu_kz}. 
For $a \in A^n_z$, let us write $a \cdot z$ for the element $\bb X^+$
whose $n$ first coordinates are $a_{-n}, \ldots, a_{-1}$ and whose $k$-th coordinate is $z_{k-n}$ for $k \geq n$.


\begin{lemma}\label{Def_rho_zt}
Let $g$ be in $\scr B^+$ such that $\nu^+(g) = 0$ and $g$ is not cohomologous to $0$. 
Let  $(z, t)$ be in $\bb X^+ \times \bb R$ with $\check V^{g} (z,t) >0$. 
Then, there exists a unique Borel probability measure $\check\mu^{g, -}_{z, t}$ on $\bb X^-_z$ 
such that for any $n \geq 0$ and any $a \in A^n_z$ 
we have 
\begin{align}\label{Measure_rho}
& \check\mu^{g, -}_{z, t} (\{ y \in \bb X^-_z:  y_{-n} = a_{-n}, \ldots,  y_{-1} = a_{-1} \})  \notag\\
& = \frac{1}{\check V^{g} (z,t)} 
\exp( - S_n \psi (a \cdot z)  )
\check V^{g} (a \cdot z, t + S_n g (a \cdot z) ),  
\end{align}
as soon as $t +  S_k g (T^k(a \cdot z)) \geq 0$ for all $1 \leq k \leq n$. 
\end{lemma}

\begin{proof}
The proof is a translation of the general construction of the Markov measures 
on the set of trajectories of a Markov chain. 

Recall that, for $a \in A^n_z$, we denoted by $\bb C_{a, z}$
(see \eqref{Def-CYLIND-001}) 
the associated cylinder of length $n$ in $\bb X^-_z$. 
For $n \geq 0$, define $\check\mu^{g, n}_{z, t}$ as a function on cylinders of length $n$ in $\bb X^-_z$ 
by the formula 
\begin{align*}
\check\mu^{g, n}_{z, t} (\bb C_{a, z}) 
=  \frac{1}{\check V^{g} (z,t)}  \exp( - S_n \psi (a \cdot z)  )
\check V^{g} (a \cdot z, t + S_n g (a \cdot z) ), 
\end{align*}
if $t +  S_k g (T^k(a \cdot z)) \geq 0$ for all $1 \leq k \leq n$; 
if not, we set $\check\mu^{g, n}_{z, t} (\bb C_{a, z})  = 0$, 
(compare with \eqref{Def_nu_kz}). 
We claim that for any $a \in A^n_z$, we have 
\begin{align}\label{Extension_measure_rho}
\check\mu^{g, n}_{z, t} (\bb C_{a, z}) 
= \sum_{\substack{b \in A \\ M(b, a_{-n}) = 1}}  \check\mu^{g, n+1}_{z, t} (\bb C_{b \cdot a, z}),  
\end{align}
(compare with \eqref{Extension_measure_nu-initial}).
Indeed, this follows from the harmonicity property of the function $\check V^g$ established in Lemma \ref{Lem_Harmonic_equation}. 
By Kolmogorov's extension theorem, equation \eqref{Extension_measure_rho} implies that 
there exists a unique Borel probability measure $\check\mu^{g, -}_{z, t}$ on $\bb X^-_z$ 
such that for any $n \geq 0$, 
$\check\mu^{g, n}_{z, t}$ is the restriction of $\check\mu^{g, -}_{z, t}$ to cylinders of length $n$. 
The lemma follows. 
\end{proof}

In the same way as for the function $\check V^{g}$, we can 
give an alternative definition of the measures $\check\mu^{g, -}_{z, t}$, which relies on 
a convergence property.

\begin{lemma}\label{Lem_def_rho}
Let $g \in \scr B^+$ with $\nu(g) = 0$ and assume that $g$ is not a coboundary.
Let $(z, t)$ be in $\bb X^+ \times \bb R$
and $\varphi$ be a continuous function on $\bb X^-_z$.
Then,  we have 
\begin{align} \label{limit_rho}
\check\mu^{g, -}_{z, t} (\varphi) \check V^{g}(z, t) 
& = \lim_{n \to \infty} \int_{\bb X^-_z}  \check S_{n} g (y \cdot z)  \varphi(y) 
    \mathds 1_{ \left\{ \check \tau_t^{g} (y \cdot z) >n  \right\} }  \nu^-_z(dy).  
\end{align}
%
\end{lemma}

\begin{proof}
By Lemma \ref{Lem_Limit_Tau}, the limit in equation \eqref{limit_rho} is the same as the limit of 
\begin{align*}
\int_{\bb X^-_z}  \left( t + \check S_{n} g (y \cdot z)  \right) \varphi(y) 
    \mathds 1_{ \left\{ \check \tau_t^{g} (y \cdot z) >n  \right\} }  \nu^-_z(dy). 
\end{align*}
The latter quantity is non-negative whenever $\varphi$ is non-negative. 
Besides, if $\varphi = 1$, the convergence follows from Lemma \ref{Lem_def_V}. 
Therefore, it suffices to check the convergence when $\varphi$ is the indicator function of a cylinder set.
Thus, let $m \geq 0$ be an integer. Pick $a \in A^m_z$ and let $\bb C_{a,z}$ be the associated cylinder in $\bb X^-_z$.
If $S_k g (T^k(a \cdot z)) < 0$ for some $1 \leq k \leq m$, we have for $n \geq m$, 
\begin{align*}
\int_{\bb X^-_z}  \left( t + \check S_{n} g (y \cdot z)  \right)  \mathds 1_{\bb C_{a,z} }(y) 
    \mathds 1_{ \left\{ \check \tau_t^{g} (y \cdot z) >n  \right\} }  \nu^-_z(dy)  
 = 0.  
\end{align*}
If not, we have for $n \geq m$,  
\begin{align*}
& \int_{\bb X^-_z}  \left( t + \check S_{n} g (y \cdot z)  \right)  \mathds 1_{\bb C_{a,z} }(y) 
    \mathds 1_{ \left\{ \check \tau_t^{g} (y \cdot z) >n  \right\} }  \nu^-_z(dy)  
     = \exp(- S_m \psi (a \cdot z))    \notag\\
& \quad  \times  \int_{\bb X^-_{a \cdot z}}  \left( t + \check S_{n - m} g (y \cdot a \cdot z) 
   + S_m g(a \cdot z) \right)  
    \mathds 1_{ \left\{ \check \tau_{t + S_m g(a \cdot z)}^{g} (y \cdot a \cdot z) > n - m  \right\} }  \nu^-_{a \cdot z}(dy). 
\end{align*}
By Lemma \ref{Lem_def_V}, as $n \to \infty$, this converges to 
\begin{align*}
\exp(- S_m \psi (a \cdot z))  \check V^g(a \cdot z, t + S_m g(a \cdot z)), 
\end{align*}
which, by  the definition of $\check\mu^{g, -}_{z, t}$ in Lemma \ref{Def_rho_zt}, 
is equal to $\check\mu^{g, -}_{z, t} (\bb C_{a,z}) \check V^g (z,t)$.  
\end{proof}

Using Lemma \ref{Lem_def_rho}, we can now give a definition of $\check\mu^{g, -}_{z, t}$ 
for a function $g$ only depending on finitely many negative coordinates. 

\begin{lemma}\label{Lem_Harmonic_Function_rho}
Let $g \in \scr B$ such that $\nu(g) = 0$ and $g$ is not a coboundary. 
Assume that $g$ only depends on $m$ negative coordinates for some $m \geq 0$.
In other words, the function $h = g \circ T^{m} \in \scr B^+$.
Let $(z, t)$ be in $\bb X^+ \times \bb R$ and $\varphi$ be a continuous function on $\bb X^-_z$. 
For $a \in A^m_z$, set $\varphi_a$ to be the function $y \mapsto \varphi(y \cdot a)$ on $\bb X^-_{a \cdot z}$.
Then,  we have 
\begin{align*}
& \lim_{n \to \infty} \int_{\bb X^-_z} \check{S}_n g (y \cdot z)  \varphi(y) 
 \mathds 1_{ \left\{ \check{\tau}_t^g (y \cdot z) >n \right\} } \nu^-_z(dy)  \notag\\
& = \sum_{a \in A^m_z} \exp( - S_m \psi(a \cdot z))  \check{V}^h (a \cdot z, t) \check\mu^{h, -}_{a \cdot z, t} (\varphi_a).  
\end{align*} 
\end{lemma}

Before proving this lemma, we recall some useful facts. 
 Let $g$ and $h$ be as in Lemma \ref{Lem_Harmonic_Function_rho}. 
For $z \in \bb X^+$ and $t \in \bb R$,  
\begin{align*}
\check{V}^g (z,t) = \mathcal{L}_{\psi}^m \left( \check{V}^h (\cdot, t) \right)(z)
 = \sum_{a \in A^m_z} \exp( - S_m \psi(a \cdot z))  \check{V}^h (a \cdot z, t). 
\end{align*}
If $\check{V}^g (z,t) >0$ and $\varphi$ is a continuous function on $\bb X^-_z$, we set
\begin{align}\label{Def_rho_gzt}
\check\mu^{g, -}_{z, t} (\varphi) =  \frac{1}{ \check{V}^g (z,t) }
\sum_{a \in A^m_z} \exp( - S_m \psi(a \cdot z))  \check{V}^h (a \cdot z, t) \check\mu^{h, -}_{a \cdot z, t} (\varphi_a). 
\end{align}
Lemma \ref{Lem_Harmonic_Function_rho} implies that 
the notation \eqref{Def_rho_gzt} 
is coherent with that introduced in Lemma \ref{Def_rho_zt}. 

\begin{proof}[Proof of Lemma \ref{Lem_Harmonic_Function_rho}]
As in the proof of  Lemma \ref{Lem_Harmonic_Function_V}, 
by conditioning over the $m$ first coordinates of $y$, we get for $n \geq 0$,  
\begin{align*}
& \int_{\bb X^-_z} \check{S}_n g (y \cdot z)  \varphi(y)
 \mathds 1_{ \left\{ \check{\tau}_t^g (y \cdot z) >n \right\} } \nu^-_z(dy) \notag\\
 & = \sum_{a \in A_z^m} \exp(- S_m \psi(a \cdot z))   
  \int_{\bb X^-_{a \cdot z}}   \check{S}_n  g (  (y \cdot a)  \cdot z)  \varphi(y \cdot a) 
 \mathds 1_{ \left\{ \check{\tau}_t^{g} ( (y \cdot a)  \cdot z) > n \right\} }  \nu^-_{a \cdot z}(dy)   \notag\\
 & = \sum_{a \in A_z^m} \exp(- S_m \psi(a \cdot z))   
  \int_{\bb X^-_{a \cdot z}}   \check{S}_n  h (T^{-m}  (y \cdot a)  \cdot z)  \varphi_a(y)
 \mathds 1_{ \left\{ \check{\tau}_t^{h \circ T^{-m}} ( (y \cdot a)  \cdot z) > n \right\} }  \nu^-_{a \cdot z}(dy)   \notag\\
& = \sum_{a \in A_z^m} \exp(- S_m \psi(a \cdot z))   
  \int_{\bb X^-_{a \cdot z}}   \check{S}_{n} h ( y \cdot (a \cdot z))   \varphi_a(y) 
   \mathds 1_{ \left\{ \check{\tau}_t^h (y \cdot (a \cdot z)) >n \right\} }   \nu^-_{a \cdot z}(dy), 
\end{align*} 
where we have used the relations $(y \cdot a) \cdot z = T^m (y \cdot (a \cdot z))$
and $\check{\tau}_t^{h \circ T^{-m}} = \check{\tau}_t^{h} \circ T^{-m}$.  
The conclusion now follows from Lemma \ref{Lem_def_rho}.
\end{proof}

Now we prove that the convergence in Lemma \ref{Lem_Harmonic_Function_rho} 
holds in a weak sense for every function $g \in \scr B$.

\begin{lemma}\label{Lem_harm_func_001_rho}
Assume that $g\in \scr B$ is not a coboundary w.r.t.\ to $T$ and $\nu (g)=0$. 
Then, for any $z \in \bb X^+$, 
for any continuous compactly supported function  $\vphi$ on $\bb X^-_z \times \bb R$,  
 the following limit exists and is finite: 
\begin{align*} 
\lim_{n \to \infty} \int_{\bb R}  \int_{\bb X^-_z} 
     \check S_{n} g (y \cdot z) \varphi(y, t)  \mathds 1_{ \left\{ \check \tau_t^g(y \cdot z) >n \right\} } \nu^-_z(dy)  dt. 
\end{align*}
\end{lemma}


\begin{proof}
First let us assume that $\varphi$ is of the form $(y, t) \mapsto \varphi_1(y) \varphi_2(t)$,
where $\varphi_1$ and $\varphi_2$ are non-negative continuous functions on $\bb X^-_z$ and $\bb R$,
and $\varphi_2$ is compactly supported. 
In that case, let $(g_m)_{m \geq 0}$, $c_1>0$ and $\alpha \in (0,1)$ be as in Lemma \ref{Lem_Appro_g}. 
Set 
\begin{align*}
W_n (z, t) 
= \int_{\bb X^-_z} \left( t +  \check S_{n} g (y \cdot z) \right)  \varphi_1(y)   
   \mathds 1_{ \left\{ \check \tau_t^g(y \cdot z) >n \right\} } \nu^-_z(dy) 
\end{align*}
and 
\begin{align*}
W_{n,m} (z, t) 
= \int_{\bb X^-_z} \left( t +  \check S_{n} g_m (y \cdot z) \right)  \varphi_1(y)  
     \mathds 1_{ \left\{ \check \tau_t^{g_m} (y \cdot z) >n \right\} } \nu^-_z(dy). 
\end{align*}
By \eqref{Basic_Inequality}, we have the inclusions
\begin{align*}
\left\{  \check \tau_{t-2c_1  \alpha^{m}}^{g_{m}} > n  \right\}
\subseteq    \left\{  \check \tau_t^{g} > n \right\}   
\subseteq  \left\{  \check \tau_{t+2c_1  \alpha^{m}}^{g_{m}} > n \right\}, 
\end{align*}
which imply that
\begin{align}\label{Pf_W_n_Wnp_rho}
W_{n,m} (z, t - 2 c_1 \alpha^m) \leq  W_{n} (z, t)  \leq  W_{n,m} (z, t + 2 c_1 \alpha^m). 
\end{align}

By lemma \ref{Lem_Harmonic_Function_rho}, for fixed $m \geq 0$, as $n \to \infty$, 
the function $W_{n,m} (z, t)$ converges to $\check\mu^{g_m, -}_{z, t} (\varphi_1) \check{V}^{g_m} (z, t)$.
From \eqref{Pf_W_n_Wnp_rho} we get
\begin{align*}
\check\mu^{g_m, -}_{z, t - 2 c_1 \alpha^m} (\varphi_1)
\check{V}^{g_m} (z, t - 2 c_1 \alpha^m) 
& \leq \liminf_{n \to \infty}  W_{n} (z, t) 
\leq \limsup_{n \to \infty}  W_{n} (z, t)   \notag\\
& \leq \check\mu^{g_m, -}_{z, t + 2 c_1 \alpha^m} (\varphi_1) \check{V}^{g_m} (z, t + 2 c_1 \alpha^m). 
\end{align*}
Now we have 
\begin{align}\label{Equation_V_rho}
& \int_{\bb R} \vphi_2(t) 
\left[ \check\mu^{g_m, -}_{z, t + 2 c_1 \alpha^m} (\varphi_1) \check{V}^{g_m} (z, t + 2 c_1 \alpha^m) 
      -  \check\mu^{g_m, -}_{z, t - 2 c_1 \alpha^m} (\varphi_1)  \check{V}^{g_m} (z, t - 2 c_1 \alpha^m) \right] dt  
       \nonumber\\
& =  \int_{\bb R} \left[ \vphi_2(t - 2 c_1 \alpha^m) - \vphi_2(t + 2 c_1 \alpha^m) \right] 
       \check\mu^{g_m, -}_{z, t \alpha^m} (\varphi_1)  \check{V}^{g_m} (z, t) dt. 
\end{align}
Using \eqref{Pf_Wnp_02} and Lemma \ref{Lem_Harmonic_Function_V}, we have that $\check{V}^{g_m} (z, t) \leq c_2 + \max \{t, 0\}$. 
As $\vphi_2$ is continuous on $\bb R$ with compact support, 
by the Lebesgue dominated convergence theorem, we get 
that the left hand side of \eqref{Equation_V_rho} converges to $0$ as $m \to \infty$. 
This tells us that $\int_{\bb R} \vphi_2(t) W_n(z,t) dt$ has a limit as $n \to \infty$. 
In other words, the lemma holds for the function $\varphi(y,t) = \varphi_1(y) \varphi_2(t)$.
This is also true when $\varphi_1$ and $\varphi_2$ are not necessarily non-negative. 

The general case follows from a standard but tedious approximation argument.  
Indeed, we can find a continuous compactly supported function $\theta$ on $\bb R$ with support $K$ 
such that for any $\ee >0$, there exist an integer $p \geq 0$ and 
continuous functions $\varphi_{i, 1}$ on $\bb X^-_z$
and  continuous compactly supported functions $\varphi_{i, 2}$ on $\bb R$ with support included in $K$, 
$1 \leq i \leq p$,  with 
\begin{align}\label{Approximation_varphi_yt}
\sup_{y \in \bb X^-_z} 
\left| \varphi(y, t) - \varphi_\ee (y,t)   \right|  
\leq \ee  \theta(t),  \quad  t \in \bb R, 
\end{align}
where $\varphi_{\ee} (y,t) = \sum_{i = 1}^p \varphi_{i, 1}(y)  \varphi_{i, 2}(t)$. 
We set $t_0 = \sup_{t \in K} |t|$. By Lemma \ref{Lem_Limit_Tau}, we need to show that 
\begin{align*}
U_n = \int_{\bb R}  \int_{\bb X^-_z} 
   \left( t_0 +  \check S_{n} g (y \cdot z)  \right)   \varphi(y, t) 
   \mathds 1_{ \left\{ \check \tau_t^g(y \cdot z) >n \right\} } \nu^-_z(dy)  dt
\end{align*}
has a limit as $n \to \infty$. 
By the first case, we know that 
\begin{align*}
U_{n,\ee} = \int_{\bb R}  \int_{\bb X^-_z} 
   \left( t_0 +  \check S_{n} g (y \cdot z)  \right)   \varphi_{\ee} (y, t) 
   \mathds 1_{ \left\{ \check \tau_t^g(y \cdot z) >n \right\} } \nu^-_z(dy)  dt
\end{align*}
has a limit $U_{\ee}$ as $n \to \infty$. 
Besides, by Lemma \ref{Lem_harm_func_001}, we get that
\begin{align*}
\int_{\bb R}  \int_{\bb X^-_z} 
   \left( t_0 +  \check S_{n} g (y \cdot z)  \right)   \theta (t) 
   \mathds 1_{ \left\{ \check \tau_t^g(y \cdot z) >n \right\} } \nu^-_z(dy)  dt
\end{align*}
converges to $\int_{\bb R} \check V^g (z,t) \theta(t) dt$. 
By \eqref{Approximation_varphi_yt}, we have 
\begin{align*}
U_{\ee} - \ee \int_{\bb R} \check V^g (z,t) \theta(t) dt 
\leq  \liminf_{n \to \infty} U_n \leq  \limsup_{n \to \infty} U_n
\leq  U_{\ee}  + \ee \int_{\bb R} \check V^g (z,t) \theta(t) dt, 
\end{align*}
which gives 
\begin{align*}
\limsup_{n \to \infty} U_n -  \liminf_{n \to \infty} U_n \leq 2 \ee \int_{\bb R} \check V^g (z,t) \theta(t) dt. 
\end{align*}
Hence the proof of Lemma \ref{Lem_harm_func_001_rho} is complete. 
\end{proof}

Now we use the previous lemma to build a Radon measure $\check \mu_{z}^{g, -}$ 
on $\bb X_z^- \times \bb R$ for any $g \in \scr B$. 

\begin{lemma}\label{Lemma exist of harm func(copy)Radon}
Assume that $g\in \scr B$ is not a coboundary w.r.t.\ to $T$ and $\nu (g)=0$.
Then, for any $z \in \bb X^+$, 
there exists a unique Radon measure $\check \mu_{z}^{g, -}$ on $\bb X_z^- \times \bb R$ such that
for any continuous compactly supported function  $\vphi$ on $\bb X_z^- \times \bb R$,  
\begin{align} 
\lim_{n \to \infty} \int_{\bb R}  \int_{\bb X^-_z}   \vphi(y, t)  
   \check S_{n} g (y \cdot z) \mathds 1_{ \left\{ \check \tau_t^g(y \cdot z) >n \right\} } \nu^-_z(dy)  dt  
  = \int_{\bb R}  \int_{\bb X^-_z}  \vphi(y, t) \check \mu_{z}^{g, -}(dy, dt).    \label{limit V002_RadonMeasure}
\end{align}
Besides, the marginal measure of $\check \mu_{z}^{g, -}$ on $\bb R$ under the natural projection map 
is the absolutely continuous measure $\check V^g(z,t) dt$. 
\end{lemma}


\begin{proof} 
By Lemma \ref{Lem_harm_func_001_rho}, the limit on the left hand side of \eqref{limit V002_RadonMeasure} exists.
By Lemma \ref{Lem_Limit_Tau}, the limit is the same as the one of 
\begin{align*}
\lim_{n \to \infty} \int_{\bb R}  \int_{\bb X^-_z}   \vphi(y, t)  
 \left(  t_0 + \check S_{n} g (y \cdot z) \right) 
  \mathds 1_{ \left\{ \check \tau_t^g(y \cdot z) >n \right\} } \nu^-_z(dy)  dt,
\end{align*}
where $t_0 >0$ is arbitrarily large.
In particular, this limit is non-negative. 
By Riesz representation theorem, it may be written as $\check \mu_{z}^{g, -}(\varphi)$, 
where $\check \mu_{z}^{g, -}$ is a Radon measure on $\bb X_z^- \times \bb R$. 
By Lemma \ref{Lemma exist of harm func(copy)}, the marginal measure of $\check \mu_{z}^{g, -}$ on $\bb R$ 
under the natural projection map 
is the absolutely continuous measure $\check V^g(z,t) dt$. 
\end{proof}

We define the Radon measure $\check \mu^g$ on $\bb X \times \bb R$ by setting,
 for any continuous compactly supported function $\varphi$ on $\bb X \times \bb R$, 
 \begin{align*}
\check \mu^g (\varphi) 
= \int_{\bb X^+} \int_{\bb R} \int_{\bb X^-_z} \varphi(y \cdot z, t) \check \mu_{z}^{g,-} (dy, dt) \nu^+(dz). 
\end{align*}
The main result of this section is stated as follows.

\begin{theorem}\label{Proposition exist harm func_rho}
Let $g$ be a H\"older continuous function on $\bb X$ such that  $\nu (g) = 0$ and $g$ is not a coboundary.
Then, for any continuous compactly supported function $\varphi$ on $\bb X \times \bb R$, we have 
\begin{align} \label{Harmonic_rho_f}
\lim_{n \to \infty}  \int_{\bb X} \int_{\bb R}   \vphi(x, t)  
 \check S_{n} g (x) \mathds 1_{ \left \{ \check \tau_t^g(x) >n \right\} }  \nu(dx) dt
= \int_{\bb X \times \bb R} \vphi(x, t)  \check \mu^g(dx, dt).
\end{align}
Moreover, the following harmonicity property holds:
\begin{align}\label{Harmonic_rho_aa}
\int_{\bb X \times \bb R} \vphi(x, t)  \check \mu^g(dx, dt)
= \int_{\bb X} \int_{0}^{\infty}  \vphi(Tx, t - g(x))  \check \mu^g(dx, dt).  
\end{align}
\end{theorem}

\begin{proof}
We can assume that $\varphi$ is non-negative. 
By Lemma \ref{Lemma exist of harm func(copy)Radon}, for every $z \in \bb X^+$ and $t \in \bb R$, 
we have 
\begin{align*}
& \lim_{n \to \infty}  \int_{\bb X^-_z} \int_{\bb R}   \vphi(y \cdot z, t)  \left( t + \check S_{n} g (y \cdot z) \right) 
   \mathds 1_{ \left \{ \check \tau_t^g(y \cdot z) >n \right\} }  \nu^-_z(dy) dt  \notag\\
& =  \int_{\bb X^-_z} \int_{\bb R}  \vphi(y \cdot z, t) \check \mu_{z}^{g, -}(dy, dt).  
\end{align*}
Thanks to the dominated convergence theorem, this will imply \eqref{Harmonic_rho_f}. 
Indeed, for $t \in \bb R$, set $\theta(t) = \sup_{x \in \bb X}  \vphi(x, t)$, so that 
$\theta$ is a continuous compactly supported function on $\bb R$. 
Note that 
\begin{align*}
& \int_{\bb X^-_z} \int_{\bb R}   \vphi(y \cdot z, t)  \left( t + \check S_{n} g (y \cdot z) \right) 
   \mathds 1_{ \left \{ \check \tau_t^g(y \cdot z) >n \right\} }  \nu^-_z(dy) dt  \notag\\
& \leq   \int_{\bb X^-_z} 
    \int_{\bb R}   \theta(t)   \left( t + \check S_{n} g (y \cdot z) \right) 
   \mathds 1_{ \left \{ \check \tau_t^g(y \cdot z) >n \right\} }  \nu^-_z(dy) dt. 
\end{align*}
By Lemma \ref{Lemma exist of harm func(copy)}, we have, uniformly in $z \in \bb X^+$, 
\begin{align*}
\lim_{n \to \infty} \int_{\bb R}   \theta(t)  \left( t + \check S_{n} g (y \cdot z) \right) 
   \mathds 1_{ \left \{ \check \tau_t^g(y \cdot z) >n \right\} }  \nu^-_z(dy) dt
   =  \int_{\bb R}   \theta(t)  \check V^g(z, t) dt.  
\end{align*}
By the dominated convergence theorem, we get \eqref{Harmonic_rho_f}.  

Now we prove \eqref{Harmonic_rho_aa}. 
By \eqref{Harmonic_rho_f}, 
\begin{align*}
\int_{\bb X \times \bb R} \vphi(x, t)  \check \mu^g(dx, dt)
=  \lim_{n \to \infty}  \int_{\bb X} \int_{\bb R}   \vphi(x, t)  \check S_{n} g (x) 
    \mathds 1_{ \left \{ \check \tau_t^g(x) >n \right\} }  \nu(dx) dt. 
\end{align*}
As $\nu$ is $T$-invariant, we have 
\begin{align*}
& \int_{\bb X} \int_{\bb R}   \vphi(x, t)  \check S_{n} g (x) 
    \mathds 1_{ \left \{ \check \tau_t^g(x) >n \right\} }  \nu(dx) dt  \notag\\
& =  \int_{\bb X} \int_{\bb R}   \vphi(Tx, t)  \left( \check S_{n-1} g (x)  + g(x) \right)
    \mathds 1_{ \{ \check \tau_{t + g(x)}^g(x) > n -1 \} }  \mathds 1_{ \{ t + g(x) \geq 0 \} }   \nu(dx) dt  \notag\\
 & =  \int_{\bb X} \int_{\bb R}   \vphi(Tx, t - g(x))  \left( \check S_{n-1} g (x)  + g(x) \right)
    \mathds 1_{  \{ \check \tau_{t}^g(x) > n -1 \} }  \mathds 1_{ \{ t \geq 0 \} }   \nu(dx) dt. 
\end{align*}
By Lemma \ref{Lem_Limit_Tau}, the latter has the same limit, as $n \to \infty$, as 
\begin{align*}
\int_{\bb X} \int_{\bb R}   \vphi(Tx, t - g(x))   \check S_{n-1} g (x)  
    \mathds 1_{  \{ \check \tau_{t}^g(x) > n -1 \} }  \mathds 1_{ \{ t \geq 0 \} }   \nu(dx) dt. 
\end{align*}
We will prove below that we can apply \eqref{Harmonic_rho_f} 
to the function $(x, t) \mapsto \vphi(Tx, t - g(x)) \mathds 1_{ \{ t \geq 0 \} }$
to get 
\begin{align*}
& \lim_{n \to \infty} \int_{\bb X} \int_{\bb R}   \vphi(Tx, t - g(x))   \check S_{n-1} g (x)  
    \mathds 1_{  \{ \check \tau_{t}^g(x) > n -1 \} }  \mathds 1_{ \{ t \geq 0 \} }   \nu(dx) dt  \notag\\
&    =  \int_{\bb X} \int_{0}^{\infty}  \vphi(Tx, t - g(x))  \check \mu^g(dx, dt), 
\end{align*}
which proves \eqref{Harmonic_rho_aa}. 

To finish the proof, we need to show that \eqref{Harmonic_rho_f}  implies 
that for any continuous compact supported function $\varphi$ on $\bb X \times \bb R$, 
as $n \to \infty$, the quantity 
\begin{align*}
I_n = \int_{\bb X} \int_{0}^{\infty}    
 \vphi(x, t)  \check S_{n} g (x) \mathds 1_{ \left \{ \check \tau_t^g(x) >n \right\} }  \nu(dx) dt
\end{align*}
converges to 
\begin{align}\label{Limit_Approxima_rho}
\int_{\bb X} \int_{0}^{\infty}  \vphi(x, t)  \check \mu^g(dx, dt).
\end{align}
This is a standard argument by an  approximation. 
Indeed, for $\ee >0$ and $t \in \bb R$, set $\chi_{\ee}^-(t) = 0$ if $t<0$; $\chi_{\ee}^-(t) = \frac{t}{\ee}$ if $0\leq t \leq \ee$
and $\chi_{\ee}^-(t) = 1$ if $t > \ee$. 
Define also $\chi_{\ee}^+(t) = \chi_{\ee}^-(t + \ee)$. 
Then, for any $n \geq 0$, we have 
\begin{align*}
& \int_{\bb X \times \bb R}  \chi_{\ee}^-(t) 
 \vphi(x, t)  \left( t + \check S_{n} g (x) \right) \mathds 1_{ \left \{ \check \tau_t^g(x) >n \right\} }  \nu(dx) dt  \notag\\
& \leq I_n  \leq   \int_{\bb X \times \bb R}  \chi_{\ee}^+(t) 
 \vphi(x, t)  \left( t + \check S_{n} g (x) \right) \mathds 1_{ \left \{ \check \tau_t^g(x) >n \right\} }  \nu(dx) dt. 
\end{align*}
By \eqref{Harmonic_rho_aa} and Lemma \ref{Lem_Limit_Tau}, we get 
\begin{align*}
& \int_{\bb X \times \bb R}  \chi_{\ee}^-(t) \vphi(x, t)  \check \mu^g(dx, dt)   \notag\\
&  \leq  \liminf_{n \to \infty} I_n 
  \leq \limsup_{n \to \infty} I_n   
  \leq \int_{\bb X \times \bb R}  \chi_{\ee}^+(t)  \vphi(x, t)  \check \mu^g(dx, dt). 
\end{align*}
We claim that the left and right hand sides of the latter inequality converge to the integral in \eqref{Limit_Approxima_rho}.
Indeed, for $(x, t) \in \bb X \times \bb R$, we have that $|\chi_{\ee}^+(t) \vphi(x, t)|$ and $|\chi_{\ee}^-(t) \vphi(x, t)|$ 
are dominated by $|\vphi(x, t)|$. 
The conclusion now follows from the dominated convergence theorem. 
%
\end{proof}

As for the function $\check V^g$, the measure $\check \mu^g$ enjoys the following 
continuity property on cohomology classes. 

\begin{lemma}\label{Lem_Continuity_rho_g}
Let $g \in \scr B$ with $\nu(g) = 0$. Assume that $g$ is not a coboundary. 
Let $\alpha \in (0,1)$ and $(h_m)_{m \geq 0}$ be a sequence of element of $\scr B_{\alpha}$ that converges to $0$ 
with respect to the H\"{o}lder norm $\|\cdot \|_{\alpha}$. 
For $m \geq 0$, set $g_m = g + h_m \circ T - h_m$.
Then, for any continuous compactly supported function  $\vphi$ on $\bb X \times\bb R$, we have,
\begin{align}\label{Func_V_300}
\lim_{m \to \infty} \int_{\bb X \times \bb R} \vphi(x, t) \check \mu^{g_m} (dx, dt) 
= \int_{\bb X \times \bb R} \vphi(x, t) \check \mu^{g} (dx, dt). 
\end{align} 
\end{lemma}

\begin{proof}
We can assume that $\varphi$ is non-negative. 
By Theorem \ref{Proposition exist harm func_rho}, for $m \geq 0$, we have 
\begin{align*}
\int_{\bb X \times \bb R} \vphi(x, t)  \mu^{g_m} (dx, dt) 
= \lim_{n \to \infty}   \int_{\bb X \times \bb R} \vphi(x, t)  \left( t + \check S_n g_m(x) \right)  
   \mathds 1_{\left\{ \check \tau_t^{g_m}(x) > n \right\}}  \nu(dx) dt. 
\end{align*}
For any $n \geq 0$, we have 
$S_n g_m \leq S_n g + 2 \|h_m\|_{\infty}$. Hence, for $t \in \bb R$, 
we have $\check \tau_t^{g_m} \leq \check \tau_{t + 2 \|h_m\|_{\infty}}^{g}$.
We get 
\begin{align*}
& \int_{\bb X \times \bb R} \vphi(x, t)  \left( t + \check S_n g_m(x) \right)  
   \mathds 1_{\left\{ \check \tau_t^{g_m}(x) > n \right\}}  \nu(dx) dt   \notag\\
&   \leq  \int_{\bb X \times \bb R} \vphi(x, t)  \left( t + \check S_n g_m(x) \right)  
   \mathds 1_{\left\{ \check \tau_{t + 2 \|h_m\|_{\infty}}^{g}  > n \right\}}  \nu(dx) dt.  
\end{align*}
 Again by Theorem \ref{Proposition exist harm func_rho}, as $n \to \infty$, the latter quantity converges to 
 \begin{align*}
\int_{\bb X \times \bb R} \vphi(x, t - 2 \|h_m\|_{\infty})  \check \mu^{g} (dx, dt).  
\end{align*}
Thus we have 
\begin{align*}
\int_{\bb X \times \bb R} \vphi(x, t)  \check \mu^{g_m} (dx, dt) 
\leq   \int_{\bb X \times \bb R} \vphi(x, t - 2 \|h_m\|_{\infty})  \check \mu^{g} (dx, dt).  
\end{align*}
In the same way, one also has
\begin{align*}
\int_{\bb X \times \bb R} \vphi(x, t)  \check  \mu^{g_m} (dx, dt) 
\geq   \int_{\bb X \times \bb R} \vphi(x, t + 2 \|h_m\|_{\infty})  \check \mu^{g} (dx, dt).  
\end{align*}
As $\varphi$ is continuous, the conclusion follows from the dominated convergence theorem. 
\end{proof}

\begin{proof}[Proof of Theorem \ref{Thm_Radon_Measure}]
So far we have proved Theorem \ref{Proposition exist harm func_rho}
which is an analogue of Theorem \ref{Thm_Radon_Measure}
for the reversed dynamical system $(\bb X, T^{-1}, \nu)$.
By Lemma \ref{Lem-iotanu-01}, this dynamical system is isomorphic to a subshift of finite type equipped with a Gibbs measure.  
Therefore, Theorem \ref{Thm_Radon_Measure} is actually equivalent to Theorem \ref{Proposition exist harm func_rho}. 
Formally, the former can be obtained from the latter 
by replacing $f$ with $g = f \circ T^{-1} \circ \iota = f \circ \iota \circ T$, and $\nu$ with $\iota_* \nu$. 
\end{proof}

The reader may notice that \eqref{Harmonic_rho_cc} is a particular case of \eqref{Harmonic_property_rho}, 
which is the reason to call the Radon measure $\mu^f$ harmonic. 


\section{Conditioned limit theorems}\label{Sect_CCLT}
In this section we prove Theorems \ref{Thm-exit-time001} and \ref{Thm-cnodi-lim-theor001}. 


\subsection{Proof of Theorem \ref{Thm-exit-time001}}

As in the construction of the harmonic function $\check{V}^g$ and the harmonic measure $\check \mu^g$, 
we will prove Theorem \ref{Thm-exit-time001} in several steps.
The first step is to deal with the case of functions $g$ 
depending only on the future. The following result follows from the general result 
for Markov chains established in \cite[Theorem 2.3]{GLL18}. 
The assumptions of this statement can be checked to hold thanks to the spectral gap properties  of the Ruelle
operator formulated in  Section \ref{Spectral gap theory}.

\begin{lemma}\label{Lem_Markov_chain}
Let $g \in \scr B^+$ with $\nu(g) = 0$ and assume that $g$ is not a coboundary. 
Then, for $t \in \bb R$, we have, uniformly in $z \in \bb X^+$, 
\begin{align*}
\lim_{n \to \infty} \sigma_g \sqrt{2 \pi n} \,  \nu^-_z\left(y \in \bb X^-_z: \check \tau_t^g(y \cdot z) >n \right) 
= 2 \check V^{g}(z, t).
\end{align*} 
\end{lemma}


We have to strengthen Lemma \ref{Lem_Markov_chain} by proving the following integral form:

\begin{lemma}\label{Lem_Markov_chain_Target}
Let $g \in \scr B^+$ with $\nu(g) = 0$ and assume that $g$ is not a coboundary. 
Then, for any $(z, t) \in \bb X^+ \times \bb R$ and 
for any continuous compactly supported function $\varphi$ on $\bb X^-_z$, we have 
\begin{align*}
\lim_{n \to \infty} \sigma_g \sqrt{2 \pi n}  \int_{\bb X^-_z} \varphi(y) 
\mathds 1_{ \left\{ \check \tau_t^g(y \cdot z) >n \right\} } \nu^-_z(dy) 
= 2 \check V^{g}(z, t) \int_{\bb X^-_z} \varphi(y) \check \mu^{g,-}_{z,t} (dy).
\end{align*} 
\end{lemma}

\begin{proof}
It suffices to prove this result when $\varphi$ is the indicator function of a cylinder set in $\bb X^-_z$,
since the general case follows by a standard approximation argument. 
Thus, let $m \geq 0$ and $a \in A_z^m$ and, as before, denote by $\bb C_{a,z}$  
the associated cylinder in $\bb X^-_z$ (see \eqref{Def-CYLIND-001}). 

If $t + S_k g(T^{m-k} (a \cdot z)) \geq 0$ for every $1 \leq k \leq m$, we have 
\begin{align*}
& \sigma_g \sqrt{2 \pi n} \int_{\bb X^-_z}  
\mathds 1_{\bb C_{a,z}} (y)  \mathds 1_{ \left\{ \check \tau_t^g(y \cdot z) >n \right\} } \nu^-_z(dy)   \notag\\
&  =  \sigma_g \sqrt{2 \pi n} \exp(-S_m \psi(a \cdot z))  \int_{\bb X^-_{a \cdot z} }   
   \mathds 1_{ \left\{ \check \tau_{t + S_m g( a \cdot z) }^g(y \cdot (a \cdot z)) >n - m \right\} } 
    \nu^-_{a \cdot z}(dy). 
\end{align*}
By Lemma \ref{Lem_Markov_chain}, as $n \to \infty$, the latter quantity converges to 
\begin{align*}
2 \check V^{g}(a \cdot z, t + S_m g( a \cdot z ) )  \exp(-S_m \psi(a \cdot z)), 
\end{align*}
which, by definition, is equal to $2 \check V^{g}(z, t)  \check \mu^{g,-}_{z,t} (\bb C_{a,z}).$ 

If there exists $1 \leq k \leq m$ with $t + S_k g(T^{m-k} (a \cdot z)) <0$, 
we have $\check \mu^{g,-}_{z,t} (\bb C_{a,z}) = 0$
and
\begin{align*}
\int_{\bb X^-_z}  
\mathds 1_{\bb C_{a,z}} (y)  \mathds 1_{ \left\{ \check \tau_t^g(y \cdot z) >n \right\} } \nu^-_z(dy) = 0, 
\end{align*}
for $n >k$. 
The conclusion follows. 
\end{proof}

From Lemmas \ref{Lem_Markov_chain} and \ref{Lem_Markov_chain_Target},  
we deduce the analogous result for functions which depend only on finitely many negative coordinates. 

\begin{lemma}\label{Corollary_Markov_chain}
Let $g \in \scr B$ be such that $\nu(g) = 0$ and there exists $m \geq 0$ with $g \circ T^m \in \scr B^+$. 
Assume that $g$ is not a coboundary. 
Then, for $t \in \bb R$, we have, uniformly in $z \in \bb X^+$, 
\begin{align*}
\lim_{n \to \infty} \sigma_g \sqrt{2 \pi n} \,  \nu^-_z\left(y \in \bb X^-_z: \check \tau_t^g(y \cdot z) >n \right) 
= 2 \check V^{g}(z, t).
\end{align*} 
Moreover, for any $(z, t) \in \bb X^+ \times \bb R$ and 
for any continuous compactly supported function $\varphi$ on $\bb X^-_z$, we have 
\begin{align*}
\lim_{n \to \infty} \sigma_g \sqrt{2 \pi n}  \int_{\bb X^-_z} \varphi(y) 
\mathds 1_{ \left\{ \check \tau_t^g(y \cdot z) >n \right\} } \nu^-_z(dy) 
= 2 \check V^{g}(z, t) \int_{\bb X^-_z} \varphi(y) \check \mu^{g,-}_{z,t} (dy).
\end{align*} 
\end{lemma}

\begin{proof}
As in Lemma \ref{Lem_Harmonic_Function_rho}, 
for $a \in A_z^m$, let $\varphi_a$ be the continuous function $y \mapsto \varphi(a \cdot y)$ on $\bb X^-_{a \cdot z}$. 
We have, by setting $h = g \circ T^m$, 
\begin{align*}
& \int_{\bb X^-_z} \varphi(y)  \mathds 1_{ \left\{ \check \tau_t^g(y \cdot z) >n \right\} } \nu^-_z(dy)  \notag\\
&  = \sum_{a \in A^m_z} \exp(- S_m \psi(a \cdot z) )  \int_{\bb X^-_{a \cdot z}} \varphi_a(y) 
\mathds 1_{ \left\{ \check \tau_t^h (y \cdot (a \cdot z)) >n \right\} } \nu^-_{a \cdot z}(dy).   
\end{align*}
The conclusion now follows from Lemmas \ref{Lem_Markov_chain} and \ref{Lem_Markov_chain_Target} and \eqref{Def_rho_gzt}. 
\end{proof}

Now we use the same approximation argument as before to deduce from Lemma \ref{Corollary_Markov_chain}
a slightly weaker statement that works for every function $g$ in $\scr B$.
This is the main result of this section. 

\begin{theorem}\label{Lemma_Harmonic_V_10}
Let $g \in \scr B$ be such that $\nu(g) = 0$. 
Assume that $g$ is not a coboundary. 
Then, for any continuous compactly supported function $\vphi$ on $\bb R$, 
we have, uniformly in $z \in \bb X^+$, 
\begin{align*}
\lim_{n \to \infty} \sigma_g \sqrt{2 \pi n} 
\int_{\bb R} \vphi(t)  \nu^-_z\left(y \in \bb X^-_z: \check \tau_t^g(y \cdot z) >n \right)  dt
= 2  \int_{\bb R} \vphi(t) \check V^{g}(z, t) dt.
\end{align*} 
Moreover, 
for any continuous compactly supported function $\varphi$ on $\bb X \times \bb R$, we have 
\begin{align*}
\lim_{n \to \infty} \sigma_g \sqrt{2 \pi n} 
\int_{\bb X \times \bb R} \vphi(x, t)  
\mathds 1_{ \left\{ \check \tau_t^g(x) >n \right\} } \nu(dx) dt
= 2  \int_{\bb X \times \bb R} \vphi(x, t) \check \mu^g(dx, dt).
\end{align*} 
\end{theorem}

\begin{proof}
For $(z,t) \in \bb X^+ \times \bb R$, denote
\begin{align*}
\check{V}_n^g (z,t) = \frac{1}{2} \sigma_g \sqrt{2 \pi n} \,  \nu^-_z\left(y \in \bb X^-_z: \check \tau_t^g(y \cdot z) >n \right).
\end{align*} 
Let $(g_m)_{m \geq 0}$ be the sequence of H\"{o}lder continuous functions as in Lemma \ref{Lem_Appro_g}. 
For $z \in \bb X^+$ and $t \in \bb R$, we have 
\begin{align*}
\check{V}_n^{g_m} (z,t - 2 c_1 \alpha^m) \leq \check{V}_n^{g} (z,t)  \leq \check{V}_n^{g_m} (z,t + 2 c_1 \alpha^m). 
\end{align*} 
By taking the limit as $n \to \infty$, we get by Lemma \ref{Corollary_Markov_chain}, 
\begin{align*} 
\check{V}^{g_m} (z,t - 2 c_1 \alpha^m) 
\leq \liminf_{n \to \infty} \check{V}_n^{g} (z,t) 
\leq \limsup_{n \to \infty} \check{V}_n^{g} (z,t)  
\leq \check{V}^{g_m} (z,t + 2 c_1 \alpha^m). 
\end{align*}
The first part of the lemma now follows from Lemma \ref{Lem_Continuity_01}. 

Let now $\varphi$ be a non-negative continuous compactly supported function on $\bb X \times \bb R$.
For $m, n \geq 0$, we have
\begin{align*}
 \int_{\bb X \times \bb R} \vphi(x, t)  \mathds 1_{ \left\{ \check \tau_{t - 2c_1 \alpha^m}^{g_m}(x) > n \right\} } \nu(dx) dt
& \leq  \int_{\bb X \times \bb R} \vphi(x, t)  \mathds 1_{ \left\{ \check \tau_t^g(x) >n \right\} } \nu(dx) dt   \notag\\
& \leq  \int_{\bb X \times \bb R} \vphi(x, t)  \mathds 1_{ \left\{ \check \tau_{t + 2c_1 \alpha^m}^{g_m}(x) > n \right\} } \nu(dx) dt.  
\end{align*}
The conclusion follows from Lemmas \ref{Lem_Continuity_rho_g} and \ref{Corollary_Markov_chain}.  
\end{proof}

%
%

Now we prove Theorem \ref{Thm-exit-time001}. 

\begin{proof}[Proof of Theorem \ref{Thm-exit-time001}]
The first two assertions of Theorem \ref{Thm-exit-time001} follow from Theorem \ref{Lemma_Harmonic_V_10} by
replacing $f$ with $g = f \circ T^{-1} \circ \iota = f \circ \iota \circ T$, and $\nu$ with $\iota_* \nu$.
The third assertion 
is equivalent to the second one by Lemma \ref{Corollary_Duality}. 
\end{proof}

From Theorem \ref{Thm-exit-time001}, 
we get the following coarse domination which will be used in the proof of the 
conditioned local limit theorem (Theorem \ref{Thm-conLLT}). 

\begin{corollary}\label{Cor_CoarseBound}
Let $g$ be in $\scr B^+$ with $\nu(g) = 0$.
Assume that $g$ is not cohomologous to $0$. 
Let $G$ be a continuous compactly supported function on $\bb X^+ \times \bb R$.
Then there exists a constant $c>0$ such that for any $n \geq 1$, 
\begin{align*}
\int_{\bb R} \sup_{z \in \bb X^+} \int_{\bb X_z^-} G( T^{-n} (y \cdot z)_+,  t + \check S_n g( y \cdot z ) ) 
   \mathds 1_{ \left\{ \check \tau_t^g (y \cdot z) > n \right\} } \nu_z^-(dy)  dt
   \leq \frac{c}{\sqrt{n}}. 
\end{align*}
\end{corollary}

\begin{proof}
By replacing $G$ with the function $\sup_{z \in \bb X^+} |G(z, t)|$,
we can assume that $G$ does not depend on the first coordinate. 
Let $c_0$ be as in Lemma \ref{Lem_sum_Ine}. 
For $t \in \bb R$, set $G_1(t) = \sup_{|t' - t| \leq c_0}  |G(t')|$.
Then for any $t \in \bb R$ and $z, z' \in \bb X^+$ with $z_0 = z_0'$,
we have 
\begin{align*}
& \int_{\bb X_z^-} G(t + \check S_n g(  y \cdot z) ) 
   \mathds 1_{ \left\{ \check \tau_t^g (y \cdot z) > n \right\} } \nu_z^-(dy)  \notag\\
& \leq  \int_{\bb X_z^-} G_1 (t + \check S_n g(  y \cdot z') ) 
   \mathds 1_{ \left\{ \check \tau_{t+c_0}^g (y \cdot z') > n \right\} } \nu_z^-(dy)  \notag\\
&  \leq  c \int_{\bb X_{z'}^-} G_1 (t + \check S_n g( y \cdot z') ) 
   \mathds 1_{ \left\{ \check \tau_{t+c_0}^g (y \cdot z') > n \right\} } \nu_{z'}^-(dy), 
\end{align*}
for some constant $c>0$ coming from Lemma \ref{Lem_Absolute_Contin}. 
By integrating over $z' \in \bb X^+$, we get
\begin{align*}
& \sup_{z \in \bb X^+} \int_{\bb X_z^-} G(t + \check S_n g( y \cdot z ) ) 
   \mathds 1_{ \left\{ \check \tau_t^g (y \cdot z) > n \right\} } \nu_z^-(dy)   \notag\\
&  \leq  \frac{c}{ c_1 }  
  \int_{\bb X}   G_1(t - c_0 + \check S_n g(x) ) 
   \mathds 1_{ \left\{ \check \tau_t^g (x) > n \right\} } \nu(dx),  
\end{align*}
where $c_1 = \inf_{a \in A} \nu^+ \{z' \in \bb X^+: z'_0 = a \}$. 
Integrating over $t \in \bb R,$ 
we get the result by Theorem \ref{Thm-exit-time001}.  
\end{proof}

\subsection{Proof of Theorem \ref{Thm-cnodi-lim-theor001}}

Again we start with the case of Markov chains.
As in the previous section, using the argument of \cite[Theorem 2.5]{GLL18}, 
we get the following result. 

\begin{lemma}\label{Lem_CondiCLT}
Let $g \in \scr B^+$ with $\nu(g) = 0$ and assume that $g$ is not a coboundary. 
Then, 
for any $t \in  \bb R$ and continuous compactly supported function $F$ on $\bb X^+ \times \bb R$, 
we have, uniformly in $z \in \bb X^+$, 
\begin{align*} 
\lim_{n \to \infty}  \sigma_g  \sqrt{2 \pi n}   \int_{\bb X^-_z} 
& F \left( \left( T^{-n}(y \cdot z) \right)_+,  \frac{\check S_n g (y \cdot z)}{\sigma_g \sqrt{n}}  \right)
    \mathds 1_{ \left\{ \check \tau_t^g(y \cdot z) >n \right\} }   \nu^-_z(dy) \nonumber \\
&= 2 \check V^{g}(z, t)   \int_{\bb X^+ \times \bb R} F(z', u) \phi^+(u) du \nu^+(dz').
\end{align*}
\end{lemma}

We shall extend the previous lemma to allow functions $F$ depending on the past coordinates in $\bb X$. 

\begin{lemma}\label{Lem_CondiCLT_target1}
Let $g \in \scr B^+$ with $\nu(g) = 0$ and assume that $g$ is not a coboundary. 
Then, 
for any $t \in  \bb R$ and continuous compactly supported function $F$ on $\bb X \times \bb R$, 
we have, uniformly in $z \in \bb X^+$, 
\begin{align*} 
\lim_{n \to \infty}  \sigma_g  \sqrt{2 \pi n}   \int_{\bb X^-_z} 
& F \left( T^{-n}(y \cdot z),  \frac{\check S_n g (y \cdot z)}{\sigma_g \sqrt{n}}  \right)
    \mathds 1_{ \left\{ \check \tau_t^g(y \cdot z) >n \right\} }   \nu^-_z(dy) \nonumber \\
&=  2 \check V^{g}(z, t)   \int_{\bb X \times \bb R} F(x, u) \phi^+(u) du \nu(dx). 
\end{align*}
\end{lemma}

\begin{proof}

By a standard approximation argument, 
it suffices to prove the result for the set of functions $F$  
of the form 
\begin{align*}
(x,t) \mapsto \mathds 1_{\{ x_{-m} = a_{-m}, x_{-m+1} = a_{-m+1}, \ldots, x_{-1} = a_{-1} \}} F_1(x_+, t), 
\end{align*}
where $F_1$ is a continuous compactly supported function on $\bb X^+ \times \bb R$,
and $a_{-m}, \ldots, a_{-1} \in A$ with $M(a_{i-1}, a_{i}) = 1$ for $-m+1 \leq i \leq -1$.  
We want to determine the limit as $n \to \infty$, 
\begin{align*}
I_n: = \int_{\bb X^-_z} 
F \left( T^{-n}(y \cdot z),  \frac{\check S_n g (y \cdot z)}{\sigma_g \sqrt{n}}  \right)
    \mathds 1_{ \left\{ \check \tau_t^g(y \cdot z) >n \right\} }   \nu^-_z(dy). 
\end{align*}
Note that in this integral, all the terms only depend on the coordinates $y_{-n}, y_{-n+1}, \ldots, y_{-1}$
except $T^{-n}(y \cdot z)$. 
By integrating first  over the deep past coordinates $\ldots, y_{-n-2}, y_{-n-1}$, we get by using Lemma \ref{Lem_Fubini}, 
\begin{align*}
& I_n =  \int_{\bb X^-_z} 
F_2 \left( \left( T^{-n}(y \cdot z) \right)_+,  \frac{\check S_n g (y \cdot z)}{\sigma_g \sqrt{n}}  \right)
    \mathds 1_{ \left\{ \check \tau_t^g(y \cdot z) >n \right\} }   \nu^-_z(dy), 
\end{align*}
where, for $(z',t) \in \bb X^+ \times \bb R$,
\begin{align*}
F_2(z',t) =  \exp(- S_m \psi(a_{-m} \ldots a_{-1} \cdot z')) F_1(z', t). 
\end{align*}
Lemma \ref{Lem_CondiCLT} gives 
\begin{align*}
\lim_{n \to \infty}  \sigma_g  \sqrt{2 \pi n}   \int_{\bb X^-_z} 
&F_2 \left( \left( T^{-n}(y \cdot z) \right)_+,  \frac{\check S_n g (y \cdot z)}{\sigma_g \sqrt{n}}  \right)
    \mathds 1_{ \left\{ \check \tau_t^g(y \cdot z) >n \right\} }   \nu^-_z(dy) \nonumber \\
&= 2 \check V^{g}(z, t)   \int_{\bb X^+ \times \bb R} F_2(z', u) \phi^+(u) du \nu^+(dz'). 
\end{align*}
By construction of the measure $\nu$ in Lemma \ref{Lem_Fubini}, we have 
\begin{align*}
\int_{\bb X^+ \times \bb R} F_2(z', u) \phi^+(u) du \nu^+(dz')
= \int_{\bb X \times \bb R} F(x, u) \phi^+(u) du \nu(dx),  
\end{align*}
which ends the proof of the lemma. 
\end{proof}

As for Theorem \ref{Thm-exit-time001}, we get the following version of Lemma \ref{Lem_CondiCLT_target1},
where we add a source target function. 

\begin{lemma}\label{Lem_CondiCLT_target2}
Let $g \in \scr B^+$ with $\nu(g) = 0$ and assume that $g$ is not a coboundary. 
Then, 
for any $(z, t) \in \bb X^+ \times \bb R$ and continuous compactly supported function $F$ 
on $\bb X_z^- \times \bb X \times \bb R$, 
we have 
\begin{align*} 
\lim_{n \to \infty}  \sigma_g  \sqrt{2 \pi n}   \int_{\bb X^-_z} 
& F \left(y,  T^{-n}(y \cdot z),  \frac{\check S_n g (y \cdot z)}{\sigma_g \sqrt{n}}  \right)
    \mathds 1_{ \left\{ \check \tau_t^g(y \cdot z) >n \right\} }   \nu^-_z(dy) \nonumber \\
&= 2 \check V^{g}(z, t)   
 \int_{\bb X_z^- \times \bb X \times \bb R} F(y', x, u) \phi^+(u) \check \mu^{g,-}_{z,t} (dy') \nu(dx) du.
\end{align*}
\end{lemma}

The proof of Lemma \ref{Lem_CondiCLT_target2} can be carried out
in the same way as that of Lemma \ref{Lem_Markov_chain_Target} and therefore is left to the reader. 
By using again conditioning and Lemma \ref{Lem_Harmonic_Function_rho}, 
we extend the previous lemma to functions $g$ depending on finitely many coordinated of the past.

\begin{lemma}
Let $g \in \scr B$ be such that $\nu(g) = 0$ and there exists $m \geq 0$ with $g \circ T^m \in \scr B^+$. 
Assume that $g$ is not a coboundary. 
Then, 
for any $t \in  \bb R$ and continuous compactly supported function $F$ on $\bb X \times \bb R$, 
we have, uniformly in $z \in \bb X^+$, 
\begin{align*} 
\lim_{n \to \infty}  \sigma_g  \sqrt{2 \pi n}   \int_{\bb X^-_z} 
& F \left( T^{-n}(y \cdot z),  \frac{\check S_n g (y \cdot z)}{\sigma_g \sqrt{n}}  \right)
    \mathds 1_{ \left\{ \check \tau_t^g(y \cdot z) >n \right\} }   \nu^-_z(dy) \nonumber \\
&= 2 \check V^{g}(z, t)   \int_{\bb X \times \bb R} F(x, u) \phi^+(u)  \nu(dx) du.
\end{align*}
Moreover, for any $(z, t) \in \bb X^+ \times \bb R$ and 
for any continuous compactly supported function $F$ on $\bb X^-_z \times \bb X \times \bb R$, we have 
\begin{align*}
\lim_{n \to \infty}  \sigma_g  \sqrt{2 \pi n}   \int_{\bb X^-_z} 
& F \left(y,  T^{-n}(y \cdot z),  \frac{\check S_n g (y \cdot z)}{\sigma_g \sqrt{n}}  \right)
    \mathds 1_{ \left\{ \check \tau_t^g(y \cdot z) >n \right\} }   \nu^-_z(dy) \nonumber \\
&= 2 \check V^{g}(z, t)   
 \int_{\bb X_z^- \times \bb X \times \bb R} F(y', x, u) \phi^+(u) \check \mu^{g,-}_{z,t} (dy')  \nu(dx) du.
\end{align*} 
\end{lemma}

\begin{proof}
We prove only the second assertion, since the first one is a particular case of the second.
As in Lemma \ref{Lem_Harmonic_Function_rho}, 
for $a \in A_z^m$, set $F_a$ to be the function on $\bb X^-_{a \cdot z} \times \bb X \times \bb R$
defined by $F_a (y, x, t) = F (y \cdot a, T^m x, t)$. 
We have, by setting $h = g \circ T^m$,  
\begin{align*}
& \int_{\bb X^-_z}  F \left(y,  T^{-n}(y \cdot z),  \frac{\check S_n g (y \cdot z)}{\sigma_g \sqrt{n}}  \right)
  \mathds 1_{ \left\{ \check \tau_t^g(y \cdot z) >n \right\} } \nu^-_z(dy)  \notag\\
&  = \sum_{a \in A^m_z} \exp(- S_m \psi(a \cdot z) )  
  \int_{\bb X^-_{a \cdot z}} 
  F_a \left(y,  T^{-n}( y \cdot (a \cdot z) ),  \frac{\check S_n h ( y \cdot (a \cdot z) )}{\sigma_g \sqrt{n}}  \right)    \notag\\
& \qquad \times \mathds 1_{ \left\{ \check \tau_t^h (y \cdot (a \cdot z)) >n \right\} } \nu^-_{a \cdot z}(dy).   
\end{align*}
The conclusion now follows from Lemma \ref{Lem_CondiCLT_target2} and \eqref{Def_rho_gzt}. 
\end{proof}

The same technique as in Lemma \ref{Lemma_Harmonic_V_10} gives 

\begin{lemma}\label{Lem_Condi_CLT}
Let $g \in \scr B$ with $\nu(g) = 0$ and assume that $g$ is not a coboundary. 
Then, 
for any continuous compactly supported function $F$ on $\bb X \times \bb X \times \bb R \times \bb R$, 
we have 
\begin{align*} 
\lim_{n \to \infty}  \sigma_g  \sqrt{2 \pi n}  \int_{\bb X \times \bb R}   
& F \left(x, T^{-n} x, t,  \frac{ \check S_n g (y \cdot z)}{\sigma_g \sqrt{n}}  \right)
    \mathds 1_{ \left\{ \check \tau_t^g(x) >n \right\} }   \nu(dx) dt \nonumber \\
&= 2  \int_{\bb X  \times  \bb R \times  \bb X  \times \bb R} 
   F(x, x', t, t') \phi^+(t') \nu(dx') dt' \check \mu^g(dx, dt).
\end{align*}
\end{lemma}

Theorem \ref{Thm-cnodi-lim-theor001} easily follows from Lemma \ref{Lem_Condi_CLT}.

\section{Effective local limit theorems}

So far we have adapted some results from the theory of Markov chains to the case of hyperbolic dynamical systems
 by constructing the analogs of the harmonic functions $V^g$ and $\check V^g$ 
 and building the harmonic measures $\mu^g$ and $\check \mu^g$.  
In the remaining part of the paper we shall use these objects to establish conditioned limit theorems,
by adapting the strategy from the case of sums of independent random variables \cite{GX21}. 
We start with formulating an effective version of the ordinary local limit theorem which is adapted to our needs. 

\subsection{Spectral gap theory} \label{Spectral gap theory}
Fix $\alpha \in (0,1)$ such that $\psi \in \mathscr B^+_\alpha$, where $\psi$ is the potential function used for the construction of the Gibbs measure $\nu$ (see Section \ref{SubSec_Shift}).
Denote by $\mathscr L(\scr B^+_\alpha, \scr B^+_\alpha)$ 
the set of all bounded linear operators from $\scr B^+_\alpha$ to $\scr B^+_\alpha$
equipped with the standard operator norm $\left\| \cdot \right\|_{\mathscr B^+_\alpha \to \mathscr B^+_\alpha}$. 
From the general construction of the Ruelle operator, 
every $ f \in \scr B^+_\alpha$ gives rise to a family of perturbed operators 
$(\mathcal{L}_{\psi + \bf{i} t f})$ 
defined as follows: 
for any $\varphi \in \scr B^+_\alpha$, 
\begin{align} \label{operator Rsz}
\mathcal{L}_{\psi + \bf{i} t f} \varphi(z) 
=\sum_{z': \, Tz' = z} e^{-\psi(z') - \bf{i} t f(z')} \vphi(z'),  \quad z \in \bb X^+,\ t \in \bb R.
\end{align}
By iteration, it follows that for any $\psi, f \in \scr B_\alpha$ and $t \in \bb R$,
\begin{align*} 
\mathcal{L}^n_{\psi + \bf{i} t f} \varphi(z) 
=\sum_{z': \, T^nz' = z} e^{-S_n(\psi + \bf{i} t f) (z') } \vphi(z'),  \quad z \in \bb X^+.
\end{align*}
The following result (see \cite{PP90}) provides the spectral gap properties for the perturbed operator $\mathcal{L}_{\psi + \bf{i} t f}$. 
For similar statements in the case of Markov chains we refer to \cite{HH01}. 

\begin{lemma} \label{Lem_Perturbation}
Assume that $f\in \scr B^+_\alpha$ is not a coboundary and that $\nu (f)=0$.
Then, there exists a constant $\delta > 0$ such that for any $t \in (-\delta, \delta)$,
\begin{align}
\mathcal{L}^n_{\psi + \bf{i} t f} = \lambda^{n}_{t} \Pi_{t} + N^{n}_{t},  
\quad  n \geq 1, \label{perturb001}
\end{align}
where the mappings $t \mapsto \Pi_{ t}: (-\delta, \delta) \to \mathscr L(\scr B^+_\alpha, \scr B^+_\alpha)$
and $z \mapsto N_{ t}: (-\delta, \delta) \to \mathscr L(\scr B^+_\alpha, \scr B^+_\alpha)$ are analytic
in the operator norm topology, 
$\Pi_{ t}$ is a rank-one projection with 
$\Pi_{0}(\varphi)(z) = \nu^+ (\varphi)$ for any $\varphi \in \mathscr{B}^+_\alpha$ and $z \in \bb X^+$,
$\Pi_{ t} N_{ t} = N_{ t} \Pi_{ t} = 0$. 
Moreover, there exist $n_0\geq 1$ and $q \in (0,1)$ such that for any $t\in (-\delta,\delta)$
 the $\|N_{t}^{n_0}\|_{\mathscr B^+_\alpha \to \mathscr B^+_\alpha} \leq q$. 
\end{lemma}

The eigenvalue $\lambda_{t}$ has the asymptotic expansion: as $t \to 0$,   
\begin{align} \label{decomp-lambda001}
\lambda_{ t} = 1 - \frac{\sigma_f^2}{2} t^2 + O(|t|^3). 
\end{align}
Note that since $f$ is not a coboundary w.r.t.\ to $T$,
 the asymptotic variance $\sigma_f^2$ appearing in \eqref{decomp-lambda001} is strictly positive.

%
%
%

%


\begin{lemma}\label{Lem_StrongNonLattice}
Let $f\in \scr B^+_\alpha$ and $t \neq 0$. 
Assume that for any $p \neq 0$ and $q \in \bb R$, 
the function $p f + q$ is not cohomologous to a function with values in $\bb Z$.  
Then,
for any $t \neq 0$, the operator $\mathcal L_{\psi +{\bf i} t f }$ has spectral radius strictly less than $1$
in $\mathscr B^+_\alpha$. 
More precisely,  for any compact set $K \subset \bb R \setminus \{0\}$, 
there exist constants $c_K, c_K' >0$ such that for any $\varphi\in \mathscr{B}^+_{\alpha}$ and $n \geq 1$, 
\begin{align}\label{Spectral_Radius_Uniform}
\sup_{t \in K}   \|\mathcal L_{\psi +{\bf i}t f }^n \varphi \|_{\scr B^+_{\alpha}}
\leq  c_K' e^{- c_K n}  \|\varphi\|_{\scr B^+_{\alpha}}.
\end{align}
\end{lemma}

\begin{proof}
The proof of the first assertion can be found in \cite[Theorem 4.5]{PP90}. 
Now we prove \eqref{Spectral_Radius_Uniform}. 
For every $t \in K$, there exist $n_0(t) \geq 1$ and $\alpha(t) \in (0,1)$ such that 
$\|\mathcal L_{\psi +{\bf i}t f }^{n_0(t)} \|_{\mathscr{B}^+_{\alpha} \to  \mathscr{B}^+_{\alpha}} < \alpha(t)$. 
As the operator $\mathcal L_{\psi +{\bf i}t f }$ depends continuously on $t$ for the operator norm topology,
there exists $\delta = \delta(t)$ such that for any $s \in (t - \delta(t), t + \delta(t))$, we still have 
$\|\mathcal L_{\psi +{\bf i}s f }^{n_0(t)} \|_{\mathscr{B}^+_{\alpha} \to  \mathscr{B}^+_{\alpha}} < 1$. 
In particular, for every $n \geq 0$ we have 
$\|\mathcal L_{\psi +{\bf i}s f }^{n} \|_{\mathscr{B}^+_{\alpha} \to  \mathscr{B}^+_{\alpha}} 
\leq  c(t) \alpha(t)^{n/n_0(t)}$, for some $c(t) > 0$. 
By compactness, we can find $t_1, \ldots, t_r \in K$ such that $K \subset \bigcup_{i =1}^r (t_i - \delta(t_i), t_i + \delta(t_i))$.
In particular, by setting $c = \max_{1 \leq i \leq r} c(t_i)$, $\alpha = \max_{1 \leq i \leq r} \alpha(t_i)$
and $n_0 = \max_{1 \leq i \leq r} n_0(t_i)$, we get for any $s \in K$ and $n \geq 0$, 
$\|\mathcal L_{\psi +{\bf i}s f }^{n} \|_{\mathscr{B}^+_{\alpha} \to  \mathscr{B}^+_{\alpha}}  \leq  c \alpha^{n/n_0}$. 
\end{proof}

\subsection{Local limit theorem for smooth target functions}

In the following we establish a local limit theorem for Markov chains with a precise estimation of the remainder term.
Let $F$ be a measurable non-negative bounded target function on $\bb X \times \mathbb{R}$.
The probability we are interested in can be written as follows: for any $z \in \bb X^+$, 
\begin{align*} 
 \int_{\bb X^-_z} F \left( (T^{-n} y \cdot z)_+, \check S_n g (y\cdot z) \right) \nu^-_z(dy).
\end{align*}
The main difficulty is to give a local limit theorem with the explicit dependence of the remainder terms on $F$. 

We first describe the kind of target functions that we will use. 


\begin{lemma}\label{Lem_Measurability}
Let $X$ be a compact metric space. 
Let $F$ be a real-valued function on $X \times \bb R$ such that
\begin{enumerate}[label=\arabic*., leftmargin=*]
\item For any $t \in \bb R$, the function $z \mapsto F(z, t)$ is $\alpha$-H\"older continuous on $X$.  
\item For any $z \in X$, the function $t \mapsto F(z, t)$ is measurable on $\bb R$.
 \end{enumerate}
Then, the function $(z, t) \mapsto F(z,t)$ is measurable on $X \times \bb R$ and 
the function $t \mapsto \| F (\cdot, t) \|_{\alpha}$ is measurable on $\bb R$,
where the norm $\|\cdot \|_{\alpha}$ is the usual norm on the space of $\alpha$-H\"older continuous functions on $X$. 
Moreover, if the integral $\int_{\bb R} \| F (\cdot, t) \|_{\alpha} dt$ is finite, 
we define the partial Fourier transform $\widehat F$ of $F$ by setting for any $z \in X$ and $u \in \bb R$,
$$\widehat F(z, u) = \int_{\bb R} e^{-itu} F (z, t) dt.$$ 
This is a continuous function on $X \times \bb R$. 
In addition, for every $u \in \bb R$, 
the function $z \mapsto \widehat F(z, u)$ is $\alpha$-H\"older continuous
and $\| \widehat F (\cdot, u) \|_{\alpha} \leq \int_{\bb R} \| F (\cdot, t) \|_{\alpha} dt$. 
\end{lemma}

\begin{proof}
Since the space $X$ is separable and the function $z \mapsto F (z, t)$ is continuous on $X$ for any $t \in \bb R$,  
the supremum $\sup_{z\in X} | F(z,t) |$ 
can be taken over a countable dense subset, so that 
$t \mapsto \sup_{z\in X} | F(z,t) |$ is measurable. 
In the same way, 
since the function $z \mapsto F (z, t)$ is $\alpha$-H\"older continuous on  $X$ for any $t \in \bb R$, 
one can also verify that $\sup_{z, z' \in X}  \frac{| F \left(z, t \right)  - F \left(z', t \right)  |}{\alpha^{\omega(z, z')}}$
is a measurable function in $t$. 

In case the integral $\int_{\bb R} \| F (\cdot, t) \|_{\alpha} dt$ is finite, 
the partial Fourier transform $\widehat F$ is well defined and continuous by the dominated convergence theorem. 
The norm domination is obvious. 
\end{proof}

We denote by $\mathscr H_{\alpha}^+$ the set of real-valued functions on $\bb X^+ \times \bb R$ 
such that conditions (1) and (2) of Lemma \ref{Lem_Measurability} hold and 
the integral $\int_{\bb R} \| F (\cdot, t) \|_{\mathscr B^+_\alpha} dt$ is finite. 
%
%
For any compact set $K\subset \bb R $, 
denote by $\mathscr H_{\alpha, K}^+$ the set of functions $F \in \mathscr H_{\alpha}^+$
such that the Fourier transform $\widehat F(z,\cdot)$ has a support contained in $K$  for any $z \in \bb X^+$.

\begin{theorem} \label{Theor-LLT-bound001}
Let $\alpha \in (0,1)$. 
Assume that $g \in \scr B^+_{\alpha}$  such that 
$\nu^+ (g)=0$
and  for any $p \neq 0$ and $q \in \bb R$, 
the function $p g + q$ is not cohomologous to a function with values in $\bb Z$.  
Let $K \subset \bb R$ be a compact set. 
Then there exists a constant $c_K >0$ 
such that for any $F\in \mathscr H_{\alpha, K}^+$, 
$n\geq 1$ and $z \in \bb X^+$,
\begin{align}\label{LLT_First}
& \bigg| \sqrt{n} \int_{\bb X^-_z}  F \left( (T^{-n} y \cdot z)_+,  \check S_n g  (y\cdot z) \right) \nu^-_z(dy)   \notag\\
& \quad  -  \int_{\bb X^+ \times \bb R}  \frac{1}{\sigma_g}  \phi \left( \frac{u}{\sigma_g \sqrt{n}} \right)
    F \left(z', u\right)  du  \nu^+(dz') \bigg|  
 \leq  \frac{c_K}{\sqrt{n} }   \int_{\bb R} \| F (\cdot, t) \|_{\mathscr B^+_\alpha} dt.  
\end{align} 
\end{theorem}

\begin{proof}
Without loss of generality, we assume that $\sigma_g = 1$. 
By the Fourier inversion formula, the Fubini theorem and a change of variable $t$ to $\frac{t}{\sqrt{n}}$, 
we get
\begin{align*}
& \sqrt{n}  \int_{\bb X^-_z} F \left( (T^{-n} y \cdot z)_+, \check S_n g (y\cdot z) \right) \nu^-_z(dy)   \nonumber\\
& = \frac{\sqrt{n}}{2 \pi}  
    \int_{\bb X^-_z \times \bb R}  e^{- \bf{i} t \check S_n g (y\cdot z)} \widehat{F} ((T^{-n} y \cdot z)_+,  t)  \nu^-_z(dy)  dt  \nonumber\\
& = \frac{1}{2\pi } 
  \int_{\bb X^-_z \times \bb R}  e^{ -\frac{\bf{i} t}{\sqrt{n}} \check S_n g (y\cdot z)} 
  \widehat{F} \left((T^{-n} y \cdot z)_+,  \frac{t}{\sqrt{n}} \right)  \nu^-_z(dy) dt  =: I. 
\end{align*}
Note that the Fubini theorem can be applied since the integral on $\bb X^-_z$ is in fact a finite sum. 
Denote 
\begin{align*}
J(t) & = \int_{\bb X^-_z}  e^{ -\frac{\bf{i} t}{\sqrt{n}} \check S_n g (y\cdot z)} 
          \widehat{F} \left((T^{-n} y \cdot z)_+, \frac{t}{\sqrt{n}} \right)  \nu^-_z(dy)   \notag\\
& \qquad\qquad   - e^{- \frac{t^2}{2}} \int_{\bb X^+} \widehat{F} \left(z', \frac{t}{\sqrt{n}} \right) \nu^+(dz'). 
\end{align*}
We decompose the integral $I$ into three parts: $I=I_{1}+I_{2} + I_3,$ where
\begin{align*}
I_{1}  
     &   = \frac{1}{2\pi }\int_{|t| \leq \delta \sqrt{n}}  J(t)  dt,   \\
I_{2} &  = \frac{1}{2\pi }\int_{ \delta  \sqrt{n}  < |t| }   \left[  \int_{\bb X^-_z}  e^{ -\frac{\bf{i} t}{\sqrt{n}} \check S_n g (y\cdot z)} 
          \widehat{F} \left((T^{-n} y \cdot z)_+, \frac{t}{\sqrt{n}} \right)  \nu^-_z(dy)  \right]  dt,  \\ 
I_3 & =  \frac{1}{2\pi }\int_{ |t| \leq \delta  \sqrt{n}  }   
    \left[  e^{- \frac{t^2}{2}} \int_{\bb X^+} \widehat{F} \left(z', \frac{t}{\sqrt{n}} \right) \nu^+(dz') \right]  dt. 
\end{align*}

\textit{Estimate of $I_1.$}
Since $\int_{\bb R} \|F(\cdot, u)\|_{\scr B^+_{\alpha}} du < \infty$, 
the function $z \mapsto \widehat{F} \left(z, t \right)$ is H\"older continuous on $\bb X^+$
with H\"older norm at most $\int_{\bb R} \|F(\cdot, u)\|_{\scr B^+_{\alpha}} du$, for any fixed $t \in \bb R$. 
Applying \eqref{perturb001}, we get
\begin{align*}
J(t)  & = \mathcal{L}^n_{\psi + \frac{\bf{i} t}{\sqrt{n}} g}  \widehat{F} \left(\cdot, \frac{t}{\sqrt{n}} \right)(z) 
   - e^{- \frac{t^2}{2}}  \int_{\bb X^+} \widehat{F} \left(z', \frac{t}{\sqrt{n}} \right) \nu^+(dz')   \nonumber\\
&  =  \left( \lambda^{n}_{\frac{t}{\sqrt{n}} } -  e^{- \frac{t^2}{2}}  \right) 
  \Pi_{\frac{t}{\sqrt{n}} }  \widehat{F} \left(\cdot, \frac{t}{\sqrt{n}} \right)(z)   \nonumber\\
& \quad  +  e^{- \frac{t^2}{2}}  \left( \Pi_{\frac{t}{\sqrt{n}} } - \Pi_{0} \right) 
      \widehat{F} \left(\cdot, \frac{t}{\sqrt{n}} \right)(z)       
   + N^n_{\frac{t}{\sqrt{n}} }  \widehat{F} \left(\cdot, \frac{t}{\sqrt{n}} \right)(z)   \nonumber\\
& =:  J_1(t) + J_2(t) + J_3(t). 
\end{align*}
For the first term, by \eqref{decomp-lambda001} and simple calculations, we get
\begin{align*}
\left| J_1(t)  \right|
 \leq c \left| \lambda^{n}_{\frac{t}{\sqrt{n}} } -  e^{- \frac{t^2}{2}} \right|
      \sup_{|t'| \leq \delta} \left\| \widehat{F} \left(\cdot, t' \right) \right\|_{\scr B^+_{\alpha}}
\leq \frac{C}{\sqrt{n}} e^{- \frac{t^2}{4}} \int_{\bb R} \|F(\cdot, u)\|_{\scr B^+_{\alpha}} du. 
\end{align*}
For the second and third terms, using again Lemma \ref{Lem_Perturbation}, we obtain
\begin{align*}
\left| J_2(t)  \right| + \left| J_3(t)  \right|
  \leq  C  \left( \frac{|t|}{\sqrt{n}} e^{- \frac{t^2}{2}} + e^{-c n}  \right) 
     \int_{\bb R} \|F(\cdot, u)\|_{\scr B^+_{\alpha}} du. 
\end{align*}
Therefore, we obtain the following upper bound for $I_1$:
\begin{align}
|I_{1}|  \leq  \left( \frac{C}{\sqrt{n}} + C e^{-cn} \right)  \int_{\bb R} \|F(\cdot, u)\|_{\scr B^+_{\alpha}} du
    \leq  \frac{C}{\sqrt{n}}  \int_{\bb R} \|F(\cdot, u)\|_{ \scr B^+_{\alpha} } du.  \label{I_1 final}
\end{align}

\textit{Estimate of $I_2.$}
Since the function $\widehat{F}(z,\cdot)$ is compactly supported on $K \subset [-C_1, C_1 ]$, where $C_1 >0$ 
is a constant not depending on $z \in \bb X^+$, 
we have 
\begin{align*}
I_2 
 & =  \frac{1}{2\pi } \int_{\bb X^-_z}    
     \left[ \int_{ \delta  \sqrt{n}  < |t| \leq C_1 \sqrt{n} }   e^{ -\frac{\bf{i} t}{\sqrt{n}} \check S_n g (y\cdot z)} 
          \widehat{F} \left((T^{-n} y \cdot z)_+, \frac{t}{\sqrt{n}} \right)   dt  \right]   \nu^-_z(dy)  \\
& = \frac{\sqrt{n}}{2\pi }\int_{ \delta   < |t| \leq C_1 }   \left[  \mathcal{L}^n_{\psi + \bf{i} t f}  \widehat{F} \left(\cdot, t \right)(z) \right]  dt. 
\end{align*}
Note that, for any $t$ satisfying $\delta   < |t| \leq C_1 $,
\begin{align*}
\sup_{z \in \bb X^+}
\left|  \mathcal{L}^n_{\psi + \bf{i} tf}  \widehat{F} \left(\cdot, t \right)(z) \right| 
\leq
\left\|  \mathcal{L}^n_{\psi + \bf{i} tf}  \widehat{F} \left(\cdot, t \right) \right\|_{\mathscr B^+_{\alpha}} \leq 
\left\|  \mathcal{L}^n_{\psi + \bf{i} tf} \right\| _{ \mathcal{L(\mathscr{B}^+_{\alpha}, \mathscr{B}^+_{\alpha})} } 
    \|\widehat{F} (\cdot, t ) \|_{\mathscr B^+_{\alpha}}. 
\end{align*}
Then, by Lemma \ref{Lem_StrongNonLattice}, it follows that
\begin{align}\label{Bound_I2_aa}
|I_{2}|  
&  =    \frac{1}{2\pi }\int_{ \delta  < |t| \leq C_1}   
\sqrt{n}\left\|  \mathcal{L}^n_{\psi + \bf{i} tf} \right\| _{ \mathcal{L(\mathscr{B}^+_{\alpha}, \mathscr{B}^+_{\alpha})} } dt
  \sup_{|t'| \in [\delta,  C_1 ] }  \|\widehat{F} (\cdot, t' ) \|_{\mathscr B^+_{\alpha}}     \notag\\
& \leq  c_K' \sqrt{n} e^{-c_K n}    
   \sup_{|t'| \in [\ee,  C_1 ] }  \|\widehat{F} (\cdot, t' ) \|_{\mathscr B^+_{\alpha}}  
 \leq  c_K' e^{-c_K n}  \int_{\bb R}  \| F (\cdot, t) \|_{\mathscr B^+_{\alpha}}  dt. 
\end{align}
\textit{Estimate of $I_3.$}
Notice that 
\begin{align*}
I_3 & =  \frac{1}{2\pi }\int_{ \bb R  }   \left[  e^{- \frac{t^2}{2}} \int_{\bb X^+} \widehat{F} \left(z, \frac{t}{\sqrt{n}} \right) \nu^+(dz) \right]  dt   \\
&  \quad -  \frac{1}{2\pi }\int_{ |t| > \delta \sqrt{n}  }   \left[  e^{- \frac{t^2}{2}} \int_{\bb X^+} \widehat{F} \left(z, \frac{t}{\sqrt{n}} \right) \nu^+(dz) \right]  dt. 
\end{align*}
For the first term, by the Fourier inversion formula, 
\begin{align}
\frac{1}{2\pi } \int_{\bb R}  e^{- \frac{ t^2}{2}}
   \int_{\bb X^+} \widehat{F} \left(z, \frac{t}{\sqrt{n}} \right) \nu^+(dz) dt 
= \frac{1}{\sqrt{2 \pi n}} \int_{\bb X^+} \int_{\mathbb{R}} e^{- \frac{ t^2}{2 n}}  
      F \left(z, t\right)  dt  \nu^+(dz).  \label{locI-002}
\end{align}
For the second term, using the fact that 
$\widehat{F} \left(z, \frac{t}{\sqrt{n}} \right) \leq \int_{\bb R}  | F (z, u) |  du$, 
we have 
\begin{align}\label{locI-002bis}
& \frac{1}{2\pi } \int_{ |t| > \delta \sqrt{n} }  
\left[  e^{- \frac{t^2}{2}} \int_{\bb X^+} \widehat{F} \left(z, \frac{t}{\sqrt{n}} \right) \nu^+(dz) \right]  dt  \notag\\
& \leq  \frac{1}{2\pi } \int_{ |t| > \delta \sqrt{n} }  e^{- \frac{t^2}{2}}  dt  
    \int_{\bb X^+ \times \bb R}    | F (z, u) |  du  \nu^+(dz)  
 \leq  c e^{- \frac{\delta^2}{4} n}  \int_{\bb R}  \| F (\cdot, u) \|_{\mathscr B^+_{\alpha}}  du.  
\end{align}
Combining \eqref{I_1 final}, \eqref{Bound_I2_aa}, \eqref{locI-002} and \eqref{locI-002bis},
and taking into account that $\delta$ is a fixed constant, we conclude the proof of \eqref{LLT_First}. 
\end{proof}

\subsection{Local limit theorem for $\ee$-dominated target functions}

Let $\ee >0$. Let $f, g$ be functions on $\bb R$.  
We say that the function $g$ $\ee$-dominates the function $f$
(or $f$ $\ee$-minorates $g$) if for any $t \in \bb R$, it holds that
\begin{align*}
f(t) \leq g(t +v),  \quad  \forall  \  |v| \leq \ee.
\end{align*}
In this case we write $f \leq_{\ee} g$ or $g \geq_{\ee} f$.
For any functions $F$ and $G$  on $\bb X^+ \times \bb R$, we say that $F \leq_{\ee} G$ if $F(z, \cdot) \leq_{\ee} G(z, \cdot)$ 
for any $z \in \bb X^+$.

In the proofs we make use of the following assertion. 
Denote by $\rho$ the non-negative density function on $\bb R$, 
which is the Fourier transform of the function $(1 - |t|) \mathds 1_{|t| \leq 1}$ for $t \in \bb R$. 
Set $\rho_{\ee}(u) = \frac{1}{\ee} \rho ( \frac{u}{\ee} )$ for $u \in \bb R$ and $\ee >0$.  
\begin{lemma}\label{smoothing-lemma-001}
Let $\ee \in (0,1/4)$. Let $f: \bb R \to \bb R_+$ and $g: \bb R \to \bb R_+$
 be integrable functions and that $f\leq_{\ee} g$. 
Then for any $u \in \bb R$,
\begin{align*}
f(u) \leq  \frac{1}{1-2\ee} g*\rho_{\ee^2} (u),  \quad  
g(u)  \geq f * \rho_{\ee^2} (u) - \int_{|v| > \ee} f \left( u- v \right) \rho_{\ee^2} (v) d v.
\end{align*}
\end{lemma}



\begin{remark}\label{f is bounded if g is integr -001}
The domination property $\leq_{\ee} $ implies, in particular, that if $f \leq_{\ee} g$ and the function $g$ is integrable, then $f$ is bounded and 
$\lim_{u\to\infty} f(u)=0$, $\lim_{u\to-\infty} f(u)=0$. 
Indeed, since $f \leq_{\ee} g$ and $g$ is an integrable function, 
by Lemma \ref{smoothing-lemma-001} we have $f \leq \frac{1}{1 - 2 \ee} g * \rho_{\ee^2}$.
Since the Fourier transform of $g * \rho_{\ee^2}$ is compactly supported on $[- \frac{1}{\ee^2}, \frac{1}{\ee^2}]$, 
by the Fourier inversion formula, 
\begin{align*}
| g * \rho_{\ee^2} (x) | = \left| \frac{1}{2 \pi} \int_{\bb R} e^{itx}  \widehat{g}(t) \widehat{\rho}_{\ee^2}  (t) dt  \right| \leq c. 
\end{align*}
Therefore, $g * \rho_{\ee^2}$ is bounded on $\bb R$, so that $f$ is bounded on $\bb R$. 
\end{remark}

Below, for any function $F \in \mathscr H_{\alpha}^+$,  we use the notation
\begin{align*}
F * \rho_{\ee^2} (z,t) = \int_{\bb R} F (z, t -v) \rho_{\ee^2}(v) dv,  
\quad  z \in \bb X^+,  \  t \in \bb R, 
\end{align*}
and
\begin{align*} 
\| F \|_{\mathscr H^+_\alpha} =  \int_{\bb R} \| F(\cdot, u) \|_{\mathscr B^+_\alpha}   du,
\qquad
\| F \|_{\nu ^+\otimes \Leb} =  \int_{\bb X^+} \int_{\bb R} | F(z, u) |  du  \nu^+(dz). 
\end{align*}

The following properties will be useful in the proofs:
\begin{lemma}\label{Lem_PropertiesBanachSpace}
Let $F\in \scr H^+_\alpha $ and $\rho\in L^1(\bb R)$. Then $F*\rho \in \scr H^+_\alpha $ and 
$\| F *\rho \|_{\mathscr H^+_\alpha} \leq \| F \|_{\mathscr H^+_\alpha} \| \rho \|_{L^1(\bb R)}$. 
\end{lemma}

\begin{theorem}\label{Lem_LLT_Nonasm}
Let $\alpha \in (0, 1)$ and $g\in \scr B^+_{\alpha}$ be such that $\nu^+ (g)=0$.
Assume that for any $p \neq 0$ and $q \in \bb R$, 
the function $p g + q$ is not cohomologous to a function with values in $\bb Z$.  
Then for any $\ee \in (0, \frac{1}{8})$, 
there exist constants $c, c_{\ee} >0$ such that for any non-negative function $F$ and any function $G \in \mathscr H_{\alpha}^+$
satisfying $F \leq_{\ee} G$, $n\geq 1$ and $z \in \bb X^+$, 
\begin{align}
&   \int_{\bb X^-_z}  
  F \left( (T^{-n} y \cdot z)_+,  \check S_n g (y\cdot z) \right) \nu^-_z(dy)  \notag\\
&    \leq   \frac{1}{\sqrt{n}} \int_{\bb X^+}  \int_{\bb R}  
    \frac{1}{\sigma_g} \phi \left(\frac{u}{\sigma_g \sqrt{n}} \right)  G \left(z', u\right)  du  \nu^+(dz')  
      +  \frac{c\ee}{\sqrt{n}} \| G \|_{\nu ^+\otimes \Leb}  
     +  \frac{c_{\ee}}{ n }  \| G \|_{\mathscr H^+_{\alpha}},    \label{LLT_Upper_aa}
\end{align}
and for any non-negative  function $F$ and non-negative  functions $G, H \in \mathscr H_{\alpha}^+$
satisfying $H \leq_{\ee} F \leq_{\ee} G$, $n\geq 1$ and $z \in \bb X^+$, 
\begin{align}
&  \int_{\bb X^-_z}  
  F \left( (T^{-n} y \cdot z)_+,  \check S_n g (y\cdot z) \right) \nu^-_z(dy)   \notag\\
& \geq  \frac{1}{\sqrt{n}} \int_{\bb X^+} \int_{\mathbb{R}}  
  \frac{1}{\sigma_g} \phi \left(\frac{u}{\sigma_g \sqrt{n}} \right)   H \left(z', u\right)  du  \nu^+(dz')   \notag\\
& \quad      -  \frac{ c \ee}{\sqrt{n}} \| G \|_{\nu ^+\otimes \Leb}  
   -  \frac{c_{\ee}}{ n } 
      \left(  \| G \|_{\mathscr H^+_{\alpha}}  
       +  \| H \|_{\mathscr H^+_{\alpha}} \right).   
       \label{LLT_Lower_aa}
\end{align}
\end{theorem}


\begin{proof}
Without loss of generality, we assume that $\sigma_g =1$. 
We first prove the upper bound \eqref{LLT_Upper_aa}. 
By Lemma \ref{smoothing-lemma-001},  we have $F \leq  (1 + 4 \ee) G * \rho_{\ee^2}$, 
and hence
\begin{align}  \label{LLT_Inequ_Smooth}
&  \int_{\bb X^-_z}  
  F \left( (T^{-n} y \cdot z)_+,  \check S_n g (y\cdot z) \right) \nu^-_z(dy)  \nonumber\\
& \leq  (1 + 4 \ee) 
 \int_{\bb X^-_z}  G * \rho_{\ee^2} \left( (T^{-n} y \cdot z)_+,  \check S_n g (y\cdot z) \right) \nu^-_z(dy).   
\end{align}
By Lemma \ref{Lem_PropertiesBanachSpace},
$\widehat{G * \rho_{\ee^2}}\in \scr H^+_\alpha$, and the support of the function
$\widehat{G * \rho_{\ee^2}} (z, \cdot) = \widehat{G} (z, \cdot) \widehat{\rho}_{\ee^2}(\cdot)$. 
is included in $[- \frac{1}{\ee^2}, \frac{1}{\ee^2}]$,
for all $z \in \bb X^+$. 
Using Theorem \ref{Theor-LLT-bound001}, 
for any $\ee \in (0, \frac{1}{4})$, there exists  $c_{\ee} >0$ 
such that for all $n\geq 1$ and $z \in \bb X^+$, 
\begin{align}
&  \int_{\bb X^-_z}  G * \rho_{\ee^2}  \left( (T^{-n} y \cdot z)_+,  \check S_n g (y\cdot z) \right) \nu^-_z(dy) \notag \\
& \leq  \frac{1}{\sqrt{n}} \int_{\bb X^+} \int_{\mathbb{R}}  
   \phi \left(\frac{u}{\sqrt{n}} \right)  G * \rho_{\ee^2} \left(z, u\right)  du  \nu^+(dz)  
  +   \frac{c_{\ee}}{ n }   \|  G\|_{\scr H^+_{\alpha}}.   
   \label{LLT_Inequa_aa}
\end{align}
By a change of variable and Fubini's theorem, we have for any $z \in \bb X^+$, 
\begin{align} \label{LLT_B_Fubini_aa}
&  \int_{\mathbb{R}}  \phi \left(\frac{u}{\sqrt{n}} \right)  G * \rho_{\ee^2} \left(z, u\right)  du   
    = \sqrt{n}  \int_{\bb R } \phi_{\sqrt{n}} * \rho_{\ee^2} \left( t \right) G(z, t) dt,
\end{align}
where $\phi_{\sqrt{n}}(t) = \frac{1}{\sqrt{2 \pi n}} e^{- \frac{t^2}{2n}}$, $t \in \bb R$. 
For brevity, denote $\psi(t)=\sup_{|v|\leq\ee} \phi_{\sqrt{n}} (t +v)$, $t \in \bb R$. 
Using the second inequality in Lemma \ref{smoothing-lemma-001}, 
we have 
\begin{align*} 
&  \int_{\bb R } \phi_{\sqrt{n}} * \rho_{\ee^2} \left( t \right) G(z, t) dt  \notag\\
& \leq    \int_{\mathbb{R}} \psi \left( t \right)  G(z, t)  dt    
   +   \int_{\mathbb{R}} \int_{\abs{v} \geq \ee} 
        \phi_{\sqrt{n}} \left( t - v \right) \rho_{\ee^2} (v) d v  G(z, t)  dt   
=: J_1+J_2.  
\end{align*}
For $J_1$, by Taylor's expansion and the fact that the function $\phi'$ is bounded on $\bb R$,
we derive that
\begin{align} \label{LLT_B_J1} 
J_1
& =  \frac{1}{\sqrt{n}} \left[ \int_{-\infty}^{-\ee}  \phi \left( \frac{t + \ee}{\sqrt{n}} \right)   G(z, t) dt  
   +    \int_{-\ee}^{\ee} \frac{1}{\sqrt{2\pi}}  G(z, t) dt    
   +   \int_{\ee}^{\infty}  \phi \left( \frac{t - \ee}{\sqrt{n}} \right)  G(z, t) dt  \right]  \notag\\
 & \leq   \frac{1}{\sqrt{n}}  \int_{\bb R}  \phi \left( \frac{t}{\sqrt{n}} \right)   G(z, t) dt  
     +  \frac{c \ee}{\sqrt{n}}  \int_{\bb R}   G (z, t)  dt.  
\end{align}
For $J_2$, since $\phi_{\sqrt{n}} \leq \frac{1}{\sqrt{n}}$ and $\int_{|v| \geq \ee} \rho_{\ee^2} (v) d v \leq 2 \ee$, we get
\begin{align}\label{LLT_B_J2}
J_2 \leq   \frac{1}{\sqrt{n}} \int_{\mathbb{R}} \left( \int_{|v| \geq \ee} \rho_{\ee^2} (v) d v \right)  G(z, t)  dt  
      \leq  \frac{ 2 \ee}{\sqrt{n}}   \int_{\bb R}   G (z, t)  dt. 
\end{align}
From \eqref{LLT_B_J1} and \eqref{LLT_B_J2}, 
together with \eqref{LLT_Inequ_Smooth} and \eqref{LLT_Inequa_aa},  we get \eqref{LLT_Upper_aa}. 

Now we prove the lower bound \eqref{LLT_Lower_aa}. 
Since $F \geq_{\ee} H$, 
using the second inequality in Lemma \ref{smoothing-lemma-001}, 
we get
\begin{align}  \label{LLT_Inequ_Smooth-002}
&  \int_{\bb X^-_z}  
  F \left( (T^{-n} y \cdot z)_+,  \check S_n g (y\cdot z) \right) \nu^-_z(dy)  \nonumber\\
& \geq   \int_{\bb X^-_z}  H * \rho_{\ee^2} \left( (T^{-n} y \cdot z)_+,  \check S_n g (y\cdot z) \right) \nu^-_z(dy)  \notag \\
&  \quad   -  \int_{\bb X^-_z}   \int_{|v| \geq \ee}
  H \left( (T^{-n} y \cdot z)_+,  \check S_n g (y\cdot z) - v \right) \rho_{\ee^2} (v) dv \nu^-_z(dy).  
\end{align}
For the first term,  by Theorem \ref{Theor-LLT-bound001},  
for any $\ee >0$, there exists  $c >0$ 
such that for all $n\geq 1$ and $z \in \bb X^+$, 
\begin{align}
&   \int_{\bb X^-_z}  H * \rho_{\ee^2}   \left( (T^{-n} y \cdot z)_+,  \check S_n g (y\cdot z) \right) \nu^-_z(dy) \notag \\
& \geq  \frac{1}{\sqrt{n}} \int_{\bb X^+} \int_{\mathbb{R}}  
\phi \left(\frac{u}{ \sqrt{n}} \right)   H * \rho_{\ee^2}  \left(z, u\right)  du  \nu^+(dz)
  -   \frac{c_{\ee}}{ n }  \int_{\bb R} \|  H (\cdot, u)\|_{\scr B^+} du.    \label{LLT_Inequa_aa-002}
\end{align}
In the same way as in \eqref{LLT_B_Fubini_aa}, we have
\begin{align} \label{LLT_B_Fubini_bb}
\int_{\mathbb{R}}  \phi \left(\frac{u}{\sqrt{n}} \right)  H * \rho_{\ee^2} \left(z, u\right)  du 
&=   \sqrt{n} \int_{\bb R } \phi_{\sqrt{n}} * \rho_{\ee^2} \left( t \right) H(z, t) dt.
\end{align}
Using the first inequality in Lemma \ref{smoothing-lemma-001}, 
we have $\phi_{\sqrt{n}} *\rho_{\ee^2} (t) \geq (1 - 2 \ee) \psi(t)$, for $t\in \bb R$, 
where $\psi(t) = \inf_{|v| \leq \ee} \phi_{\sqrt{n}} (t + v)$. 
Proceeding in the same way as in \eqref{LLT_B_J1} and \eqref{LLT_B_J2}, 
we obtain that 
\begin{align}\label{LLT_B_Lower_Term1}
& \int_{\bb X^+} \int_{\mathbb{R}}  \phi \left(\frac{u}{\sqrt{n}} \right)   H * \rho_{\ee^2}  \left(z, u\right)  du  \nu^+(dz)  \notag\\
& \geq  \int_{\bb X^+} \int_{\mathbb{R}}  \phi \left(\frac{u}{\sqrt{n}} \right)   H  \left(z, u\right)  du  \nu^+(dz)
  - c \ee  \int_{\bb X^+}  \int_{\bb R}   H (z, u)  du \nu^+(dz).
\end{align}
For the second term on the right hand side of \eqref{LLT_Inequ_Smooth-002}, 
using \eqref{LLT_Upper_aa} and the fact that $H \leq_{\ee} F$ and $\phi \leq 1$, 
we get that there exist constants $c, c_{\ee} >0$ such that for any $v \in \bb R$ and $n \geq 1$, 
\begin{align*}
&  \int_{\bb X^-_z}  
  H \left( (T^{-n} y \cdot z)_+,  \check S_n g (y\cdot z) - v \right)  \nu^-_z(dy)  \notag \\
& \leq  \frac{c}{ \sqrt{n} } \int_{\bb X^+}  \int_{\bb R}  \phi \left(\frac{u}{\sqrt{n}} \right)
  F * \kappa_{\ee^2} \left(z, u - v\right)  du  \nu^+(dz)   
     +  \frac{c_{\ee}}{ n }  \int_{\bb R} \|  F (\cdot, u) \|_{\mathscr B^+_{\alpha}} du  \notag \\
& \leq  \frac{c}{\sqrt{n}}  \int_{\bb X^+}  \int_{\bb R}  F  \left(z, u \right)  du  \nu^+(dz)  
    +  \frac{c_{\ee}}{ n }  \int_{\bb R} \|  F (\cdot, u) \|_{\mathscr B^+_{\alpha}} du. 
\end{align*}
This, together with the fact that $\int_{|v| \geq \ee} \rho_{\ee^2} (v) d v \leq 2 \ee$, implies
\begin{align} \label{LLTlowbnd-004}
&  \int_{\bb X^-_z}   \int_{|v| \geq \ee}
  H \left( (T^{-n} y \cdot z)_+,  \check S_n g (y\cdot z) - v \right) \rho_{\ee^2} (v) dv \nu^-_z(dy) \notag \\
 & \leq  \frac{c \ee}{\sqrt{n}} \int_{\bb X^+}  \int_{\bb R}  F  \left(z, u \right)  du  \nu^+(dz)  
    +  \frac{c_{\ee}}{ n }  \int_{\bb R} \|  F (\cdot, u) \|_{\mathscr B^+_{\alpha}} du.
\end{align}
From \eqref{LLT_Inequ_Smooth-002}, \eqref{LLT_Inequa_aa-002}, \eqref{LLT_B_Lower_Term1}  
and \eqref{LLTlowbnd-004},
and the fact that $H \leq F$, 
we obtain the lower bound \eqref{LLT_Lower_aa}. 
\end{proof}

\section{Effective conditioned local limit theorems} 

\subsection{Formulation of the result}

We will prove the following conditioned local limit theorem for Markov chains which provides a rate of order $n^{-1}$. 
This result will serve as an intermediate step between the conditioned central limit Theorem \ref{Thm-cnodi-lim-theor001}
and the conditioned local limit Theorem \ref{Thm-conLLT}. 
The interest of this result lies in the fact that it is uniform in the function $F$.
In particular, the theorem is effective when the support of the function $F$ moves to infinity with the rate $\sqrt{n}$.  
This strategy is inspired by \cite{Den Wachtel 2011} for random walks in cones of $\bb R^d$, 
see also \cite{GLL20}  for finite Markov chains
and \cite{GX21} for random walks on $\bb R$. 
For a different approach based on the Wiener-Hopf factorisation we refer to \cite{Carav05, Don12, VatWacht09}. 

%


\begin{theorem} \label{t-A 001}
Let $\alpha \in (0,1)$ and $g \in \scr B^+_{\alpha}$ be such that  $\nu^+ (g) = 0$.
Assume that for any $p \neq 0$ and $q \in \bb R$, 
the function $p g + q$ is not cohomologous to a function with values in $\bb Z$.  
Let $t_0 \in \bb R_+$. 
Then, there exist a constant $c>0$ and a sequence $(r_n)$ of positive numbers satisfying $\lim_{n \to \infty} r_n = 0$
with the following properties. 
For any $\ee \in (0,\frac{1}{8})$, 
there exists a constant $c_\ee > 0$ such that for any $n \geq 1$, $z \in \bb X^+$ and $t \leq t_0$, 
\begin{enumerate}[label=\arabic*., leftmargin=*]
\item  
 For any  non-negative function $F$ 
and any function $G \in \mathscr H_{\alpha}^+$ satisfying $F \leq_{\ee} G$,  
\begin{align}\label{eqt-A 001}
&  n \int_{\bb X^-_z}  F \left( (T^{-n} y \cdot z)_+,  t + \check S_n g(y \cdot z)  \right) 
   \mathds 1_{ \left\{ \check \tau_t^g (y \cdot z) >n \right\}} \nu^-_z(dy)    \nonumber\\
& \leq  \frac{2 \check V^{g}(z, t)}{ \sigma_g^2 \sqrt{2\pi} } 
  \int_{\bb X^+}  \int_{\mathbb R}  G \left(z', u' \right)
\phi^+ \left( \frac{u'}{\sigma_g \sqrt{n}} \right) du' \nu^+(dz')   \nonumber\\
& \quad + c \left( \ee^{1/4}  +   \frac{r_n}{\ee^{1/4}} \right) 
  \| G \|_{\nu^+ \otimes \Leb }  
    +   \frac{c_{\ee}}{ \sqrt{n}}  \|  G \|_{\mathscr H^+_{\alpha}}.  
\end{align}
\item 
For any non-negative  function $F$ and non-negative  functions $G, H \in \mathscr H_{\alpha}^+$
satisfying $H \leq_{\ee} F \leq_{\ee} G$, 
\begin{align} \label{eqt-A 002}
&  n \int_{\bb X^-_z}  F \left( (T^{-n} y \cdot z)_+,  t + \check S_n g(y \cdot z)  \right) 
   \mathds 1_{ \left\{ \check \tau_t^g (y \cdot z) >n \right\}} \nu^-_z(dy)    \nonumber\\
&  \geq  \frac{2 \check V^{g}(z, t)}{ \sigma_g^2 \sqrt{2\pi} } 
\int_{\bb X^+}  \int_{\bb R}  H(z', u') \phi^+\left(\frac{u'}{ \sigma_g \sqrt{n} } \right) du' \nu^+(d z')  \notag \\
&\qquad  - c \left( \ee^{1/12}  +   \frac{r_n}{\ee^{1/4}} \right)    \| G \|_{\nu^+ \otimes \Leb } 
  -  \frac{c_{\ee}}{\sqrt{n}} 
  \left( \left\Vert  G  \right\Vert _{\mathscr H^+_{\alpha}} +  \left\Vert  H  \right\Vert _{\mathscr H^+_{\alpha}}  \right). 
\end{align}
\end{enumerate}
\end{theorem}

\subsection{Auxiliary statements}
The normal density of variance $v > 0$ is denoted by 
\begin{align*}
\phi_{v} (x) = \frac{1}{\sqrt{2 \pi v} } e^{- \frac{x^2}{2 v}},  \quad x \in \bb R,  
\end{align*}
and the Rayleigh density with scale parameter $\sqrt{v}$ is denoted by 
\begin{align*}
\phi^+_{v}(x)=\frac{x}{v} e^{-\frac{x^2}{2 v}} \mathds 1_{\mathbb R_+} (x), \quad  x\in \mathbb R. 
\end{align*}
The standard normal density is denoted by $\phi(x) = \phi_1(x)$, $x\in \mathbb R$.
The following lemma from \cite{GX21} shows that when $v$ is small the convolution $\phi_{v} * \phi^+_{1-v}$ 
behaves like the Rayleigh density.
\begin{lemma} \label{t-Aux lemma}
For any $v \in (0,1/2]$ and $x\geq 0$, it holds that
\begin{align*} 
 \sqrt{1-v} \phi^+(x) 
\leq  \phi_{v} * \phi^+_{1- v}(x) 
\leq  \sqrt{1-v} \phi^+(x) +  \sqrt{v}  e^{ -\frac{x^2}{2v} }.  
\end{align*}
\end{lemma}

We need the following inequality of Haeusler \cite[Lemma 1]{Hae84}, 
which is a generalisation of Fuk's inequality for martingales. 

\begin{lemma}\label{Lem_Fuk}
Let $\xi_1, \ldots, \xi_n$ be a martingale difference sequence with respect to the non-decreasing 
$\sigma$-fields $\mathscr F_0, \mathscr F_1, \ldots, \mathscr F_n$.
Then, for all $u, v, w > 0$,
\begin{align*}
\bb P \left(  \max_{1 \leq k \leq n} \left| \sum_{i=1}^k \xi_i \right| \geq u \right)
& \leq  \sum_{i=1}^n \bb P \left( |\xi_i|  > v \right)
  + 2 \bb P \left(  \sum_{i=1}^n \bb E \left( \xi_i^2 |  \mathscr F_{i-1} \right) > w \right)   \\
& \quad  + 2 \exp \left\{ \frac{u}{v} \left( 1 - \log \frac{uv}{w} \right) \right\}. 
\end{align*}
\end{lemma}

In order to control certain natural quantities appearing in the proof,
we shall need the following definitions. 
For $\ee  >0 $,  
\begin{align}\label{Def_chiee}
\chi_{\ee} (u) = 0  \  \mbox{for} \ u \leq -\ee,  
\   \chi_{\ee} (u) = \frac{u+\ee}{\ee}   \  \mbox{for} \ u \in (-\ee,0),
\   \chi_{\ee} (u) = 1  \  \mbox{for} \  u \geq 0.  
\end{align}
Denote $\overline\chi_{\ee}(u) = 1 - \chi_{\ee}(u)$ and note that
\begin{align} \label{bounds-reversedindicators-001} 
\chi_{\ee}  \left( t-\ee \right) \leq  \mathds 1_{(0,\infty)} \left( t \right) \leq \chi_{\ee}  \left( t\right), 
\quad
\overline\chi_{\ee}  \left( t \right) \leq \mathds 1_{(-\infty,0]} \left( t \right) \leq \overline\chi_{\ee}  \left( t-\ee \right).
\end{align}

\begin{lemma}\label{Lem_Inequality_Aoverline}
Let $\alpha \in (0,1)$ and $g \in \scr B^+_{\alpha}$ be such that  $\nu^+ (g) = 0$.
Assume that $g$ is not a coboundary.   
Let $\kappa$ be a smooth compactly supported function on $\bb R$ and $\ee >0$. 
Then there exists a constant $c>0$ such that for any $G \in \scr H^+_{\alpha}$
and any $m \geq 1$,  
the function $A_m$ defined on $\bb X^+ \times \bb R$ by 
\begin{align*}
A_{m} (z,t)    
 := &  \int_{\bb X^-_{z}}   
     G *  \kappa \left( (T^{-m} y \cdot z)_+, t + \check S_m g(y \cdot z) \right)   \notag\\
 &  \times
  \overline\chi_{\ee}  \left( t-\ee + \min_{1 \leq  j \leq m}  \check S_j g(y \cdot z) \right)  \nu^-_{z}(dy), 
\end{align*}
belongs to $\scr H^+_{\alpha}$ and satisfies 
\begin{align*}
\|  A_{m} \|_{\nu^+ \otimes \Leb }  \leq   \int_{\bb R} |\kappa(t)| dt  \| G \|_{\nu^+ \otimes \Leb },  
\qquad 
\| A_{m} \|_{ \scr H^+_{\alpha} }  \leq \frac{c}{\ee}  \| G\|_{ \scr H^+_{\alpha} }.  
\end{align*}
\end{lemma}

\begin{proof}
For the first inequality, we write 
\begin{align*}
|A_{m} (z,t) | \leq  \int_{\bb X^-_{z}}   
    | G *  \kappa |  \left( (T^{-m} y \cdot z)_+, t + \check S_m g(y \cdot z) \right)  \nu^-_{z}(dy),
\end{align*}
which gives 
\begin{align*}
\|  A_{m} \|_{\nu^+ \otimes \Leb }
& \leq  \int_{\bb X \times \bb R}    | G *  \kappa |  \left( (T^{-m} x)_+, t + \check S_m g(x) \right)  \nu(dx) dt  \notag\\
& =   \int_{\bb X \times \bb R}    | G *  \kappa |  \left(  x_+, t \right)  \nu(dx) dt  
 \leq  \int_{\bb R} |\kappa(t)| dt   \| G \|_{\nu^+ \otimes \Leb }.  
\end{align*}
This finishes the proof of the first inequality.  

For the second inequality, 
recall that 
\begin{align*} 
\|  A_{m} \|_{\mathscr H^+_{\alpha}}
 =  \int_{\bb R}   \sup_{z \in \bb X^+}  | A_{m} \left(z, t \right)|  dt
 +  \int_{\bb R} \sup_{z, z' \in \bb X^+}  
    \frac{| A_{m} \left(z, t \right)  - A_{m} \left(z', t \right)  |}{\alpha^{\omega(z, z')}} dt. 
\end{align*}
We pick $c_0 >0$ as in Lemma \ref{Lem_sum_Ine}
and for $t \in \bb R$ we set $\kappa_1(t) = \sup_{|s| \leq c_0} |\kappa(t+s)|$
and $H(t) = \sup_{z \in \bb X^+} |G(z, t)|$. 
We get for $z, z' \in \bb X^+$ with $z_0 = z_0'$ and  $t \in \bb R$, 
\begin{align*}
|A_{m} (z,t) | 
& \leq  \int_{\bb X^-_{z}} H*\kappa_1 \left(   t + \check S_m g(y \cdot z')  \right) \nu^-_{z}(dy).  
\end{align*}
By Lemma \ref{Lem_Absolute_Contin},  we get
\begin{align*}
|A_{m} (z,t) | 
 \leq  c  \int_{\bb X^-_{z'}} H*\kappa_1 \left(   t + \check S_m g(y \cdot z')  \right) \nu^-_{z'}(dy), 
\end{align*}
for some constant $c$.   
By integrating over $z'$, we get 
\begin{align*}
|A_{m} (z,t) | \leq c'   \int_{\bb X} H*\kappa_1 \left(   t + \check S_m g(x)  \right) \nu(dx). 
\end{align*}
By integrating over $t$, it follows that 
\begin{align*}
\int_{\bb R} \sup_{z \in \bb X^+} |A_{m} (z,t) |  dt 
\leq  c' \int_{\bb R}  H*\kappa_1 \left(   t   \right) dt   =  c'  \int_{\bb R} |\kappa_1 (t)| dt  \int_{\bb R}  H(t) dt.  
\end{align*}
Now we dominate the second term in the norm  $\|  A_{m} \|_{\mathscr H^+_{\alpha}}$. 
For $t \in \bb R$, set $\kappa_2(t) = \sup_{|s| \leq c_0} |\kappa'(t+s)|$.
We get for $|t-t'| \leq c_0$ and $z \in \bb X^+$,
\begin{align*}
| G * \kappa(z,t) - G * \kappa(z,t') | \leq  |t-t'|  H * \kappa_2(t).
\end{align*}
Hence for $z, z', z'' \in \bb X^+$ with $z_0 = z_0' = z_0''$ and  $t \in \bb R$,  
\begin{align*}
 I_1(z,z',t) & =: \bigg| A_{m} (z,t) -  \int_{\bb X^-_{z}}   
     G *  \kappa \left( (T^{-m} y \cdot z)_+, t + \check S_m g(y \cdot z') \right)   \notag\\
 &  \qquad\qquad\quad \times
  \overline\chi_{\ee}  \left( t-\ee + \min_{1 \leq  j \leq m}  \check S_j g(y \cdot z) \right)  \nu^-_{z}(dy)  \bigg|  \notag\\
 & \leq  c \alpha^{\omega(z, z')}  \int_{\bb X^-_{z}}  H * \kappa_2 \left( t + \check S_m g(y \cdot z'')  \right)  \nu^-_{z}(dy)  \notag\\
 & \leq  c_1  \alpha^{\omega(z, z')}  \int_{\bb X^-_{z''}}  H * \kappa_2 \left( t + \check S_m g(y \cdot z'')  \right)  \nu^-_{z''}(dy),
\end{align*}
where we have used Lemma \ref{Lem_Absolute_Contin}.  
Again by integrating over $z''$, we get that 
\begin{align}\label{Bound_I1zzt}
I_1(z,z',t)   \leq   c_2 \alpha^{\omega(z, z')}  \int_{\bb X}  H * \kappa_2 \left( t + \check S_m g(x)  \right)  \nu(dx). 
\end{align}
Besides, as $G$ is in $\scr H_{\alpha}^+$, the function $L(t) = \sup_{z, z' \in \bb X^+} \alpha^{-\omega (z.z')} |G(z,t) - G(z',t)|$
is integrable on $\bb R$ and for $z,z' \in \bb X^+$ with $z_0 = z_0'$ and $t \in \bb R$, we have 
\begin{align}\label{Bound_I2zzt}
 I_2(z,z',t) : & = \bigg|   \int_{\bb X^-_{z}}   
   \bigg[  G *  \kappa \left( (T^{-m} y \cdot z)_+, t + \check S_m g(y \cdot z') \right)    \notag\\
 &  \qquad\qquad  -  G *  \kappa \left( (T^{-m} y \cdot z')_+, t + \check S_m g(y \cdot z') \right)  \bigg] \notag\\
 &  \qquad\qquad\quad \times
  \overline\chi_{\ee}  \left( t-\ee + \min_{1 \leq  j \leq m}  \check S_j g(y \cdot z) \right)  \nu^-_{z}(dy)  \bigg|  \notag\\ 
 & \leq  \alpha^{\omega (z.z')} \int_{\bb X^-_{z}}   
   L * \kappa \left(  t + \check S_m g(y \cdot z')  \right)  \nu^-_{z}(dy)   \notag\\
   & \leq c \alpha^{\omega (z.z')} \int_{\bb X}  L * \kappa_1 \left(  t + \check S_m g(x)  \right)  \nu(dx),
\end{align}
where we have again used Lemmas \ref{Lem_Absolute_Contin} and \ref{Lem_sum_Ine}. 

As $\overline\chi_{\ee}$ is $1/\ee$-Lipschitz continuous on $\bb R$, 
by reasoning in the same way and using Corollary \ref{Cor_Holder_minimum}, we get
\begin{align}\label{Bound_I3zzt}
& I_3(z,z',t)   =: \bigg|   \int_{\bb X^-_{z}}   
     G *  \kappa \left( (T^{-m} y \cdot z')_+, t + \check S_m g(y \cdot z') \right)   \notag\\
 &  \times
 \bigg[ \overline\chi_{\ee}  \left( t-\ee + \min_{1 \leq  j \leq m}  \check S_j g(y \cdot z) \right) 
  -  \overline\chi_{\ee}  \left( t-\ee + \min_{1 \leq  j \leq m}  \check S_j g(y \cdot z') \right)  \bigg] \nu^-_{z}(dy)  \bigg| \notag\\
& \leq \frac{c}{\ee}  \alpha^{\omega (z.z')} \int_{\bb X}   H *  \kappa_1 \left(  t + \check S_m g(x) \right) \nu(dx). 
\end{align}

By Lemma \ref{Lem_Absolute_Contin}, we have
\begin{align}\label{Bound_I4zzt}
 I_4(z,z',t)   & =: \bigg|   \int_{\bb X^-_{z}}   
     G *  \kappa \left( (T^{-m} y \cdot z')_+, t + \check S_m g(y \cdot z') \right)   \notag\\
 & \qquad\qquad \times
   \overline\chi_{\ee}  \left( t-\ee + \min_{1 \leq  j \leq m}  \check S_j g(y \cdot z') \right)  \nu^-_{z}(dy) 
   - A_m(z',t) \bigg| \notag\\
& \leq  c \alpha^{\omega (z.z')} \int_{\bb X}   H *  \kappa_1 \left(  t + \check S_m g(x) \right) \nu(dx).  
\end{align}

Putting \eqref{Bound_I1zzt}, \eqref{Bound_I2zzt}, \eqref{Bound_I3zzt} and \eqref{Bound_I4zzt} together,
and integrating over $t\in \bb R$, yields the required domination.  
\end{proof}

\subsection{Proof of the upper bound} 
We prove the inequality \eqref{eqt-A 001} in Theorem \ref{t-A 001}. 
It is enough to prove \eqref{eqt-A 001} only for sufficiently large $n>n_0(\ee)$, 
where $n_0(\ee)$ depends on $\ee$,
otherwise the bound becomes trivial.

Without loss of generality, we assume that $\sigma_g = 1$. 
Let $\ee\in (0,\frac{1}{8})$. With $\delta = \sqrt{\ee}$, set
$m=\left[ \delta n \right]$ and $k = n-m.$ 
Note that $\frac{1}{2}\delta \leq \frac{m}{k} \leq \frac{\delta}{1-\delta}$ for $n \geq \frac{2}{\sqrt{\ee}}$. 
Denote, for $z \in \bb X^+$ and $t \in \bb R$,  
\begin{align*}
\Psi_n (z,t) = \int_{\bb X^-_z}  F \left( (T^{-n} y \cdot z)_+, t + \check S_n g(y \cdot z) \right) 
    \mathds 1_{ \left\{ \check \tau_t^g (y \cdot z) >n \right\}} \nu^-_z(dy). 
\end{align*}
By the Markov property, we have that for any $z \in \bb X^+$ and $t \in \bb R$, 
\begin{align}\label{JJJ-markov property}
\Psi_n (z,t) = \int_{\bb X^-_z} \Psi_m \left( (T^{-k} y \cdot z)_+, t + \check S_k g(y \cdot z) \right) 
  \mathds 1_{ \left\{ \check \tau_t^g (y \cdot z) > k \right\}} 
 \nu^-_z(dy). 
\end{align}
By bounding the indicator function by $1$ in the definition of $\Psi_m$, we get
\begin{align}\label{JJJJJ-1111-000}
\Psi_m (z,t) 
&  \leq  \int_{\bb X^-_{z}} F \left( (T^{-m} y \cdot z)_+, t + \check S_m g(y \cdot z) \right) 
  \nu^-_{z}(dy)   =: J_{m}(z,t). 
\end{align}
By the local limit theorem (cf.\ Theorem \ref{Lem_LLT_Nonasm}), 
 there exist constants $c, c_{\ee} >0$ such that 
for any $m \geq 1$, $z \in \bb X^+$ and $t \in \bb R$, 
\begin{align} \label{JJJJJ-1111-001}
 J_{m}(z, t) 
&\leq H_m(t) +  \frac{ c\ee }{\sqrt{m}} \| G \|_{\nu ^+\otimes \Leb}   
+ \frac{c_{\ee}}{m}  \| G \|_{\mathscr H^+_{\alpha}}, 
\end{align}
where for brevity we set
\begin{align}
\label{JJJ005}
H_m(t) = \int_{\bb X^+}  \int_{\bb R}    
  G \left(z, u\right)  \frac{1}{\sqrt{m}} \phi \left( \frac{u - t}{\sqrt{m}} \right)  du  \nu^+(dz).
\end{align}
Using \eqref{JJJ-markov property}, \eqref{JJJJJ-1111-000} and \eqref{JJJJJ-1111-001},  
and Lemma \ref{Lem_Markov_chain},  
we get that for any $z \in \bb X^+$ and $t \in \bb R$, 
\begin{align} \label{JJJ004}
 \Psi_n (z,t)
 & \leq \int_{\mathbb X_z^-}  H_m \left( t + \check S_k g(y \cdot z) \right)  
 \mathds 1_{\{ \check \tau_t^g (y \cdot z) > k \}}  \nu^-_z (dy) \notag \\  
& \quad +   \frac{c\ee}{\sqrt{mk}}  \| G \|_{\nu ^+\otimes \Leb}
     +  \frac{c_{\ee}}{m \sqrt{k} }  \| G \|_{\mathscr H^+_{\alpha}}. 
\end{align}
Now we deal with the first term on the right hand side of \eqref{JJJ004}. 
Denote $L_m(s) = H_m(\sqrt{k} s)$ for $s \in \mathbb R$. We have 
\begin{eqnarray} \label{JJJ006}
L_m(s) =  \int_{\bb X^+}  \int_{\bb R}  G \left(z, \sqrt{k} u\right)
 \frac{1}{\sqrt{m/k}} \phi \left( \frac{s - u}{\sqrt{m/k}} \right)  du  \nu^+(dz).
\end{eqnarray}
Since the function $s \mapsto L_m(s)$ is differentiable on $\bb R$ and vanishes as $s \to + \infty$, 
using integration by parts, we have, for any $z \in \mathbb X^+$ and $t \in \bb R$,
\begin{align} \label{ApplCondLT-001}
H_{m,k}(z, t): 
& =  \int_{\mathbb X_z^-}  H_m \left( t + \check S_k g(y \cdot z) \right)  
 \mathds 1_{\{ \check \tau_t^g (y \cdot z) > k \}}  \nu^-_z (dy) \notag \\
& =  \int_{\mathbb X_z^-}  L_m \left( \frac{t + \check S_k g(y \cdot z)}{\sqrt{k}} \right)  
\mathds 1_{\{ \check \tau_t^g (y \cdot z) > k \}}  \nu^-_z (dy)  \notag \\
& =  \int_{\mathbb{R}_{+}}  
  L'_m(s)  \nu^-_z \left( \frac{t + \check S_k g(y \cdot z)}{\sqrt{k}}  >  s, 
  \check \tau_t^g (y \cdot z) > k \right) ds.
 \end{align}
Applying the conditioned central limit theorem 
(see Lemma \ref{Lem_CondiCLT}), 
we have  
\begin{align} \label{ApplCondLT-002}
H_{m,k}(z, t)
 & \leq  \frac{2 \check V^{g}(z, t)}{\sqrt{2\pi k}  }  \int_{\mathbb{R}_{+}}  L'_m(s) \left( 1 - \Phi_+(s) \right) ds 
 + \frac{r_k}{k^{1/2}} \int_{\mathbb{R}_{+}}  | L'_m(s) | ds,
\end{align}
where $r_k \to 0$ as $k \to \infty$. 
By \eqref{JJJ006}, we have
\begin{align} \label{JJJ-20001}
\int_{\mathbb{R}_{+}}  | L'_m(s) | ds 
& =  \int_{\bb X^+}  \int_{\bb R}  \int_{\mathbb R}   G \left(z,  \sqrt{m} \frac{u}{\sqrt{m/k}} \right) 
   \phi' \left( \frac{s - u}{\sqrt{m/k}} \right)
  \frac{du}{\sqrt{m/k}} \frac{ds}{\sqrt{m/k}}  \nu^+(dz) \notag\\ 
& =   \int_{\bb X^+}  \int_{\bb R}  \int_{\mathbb R}
  G \left(z,  \sqrt{m} u \right)     
  \phi'\left( s - u \right)  ds du \nu^+(dz)  \notag\\ 
& \leq c  \int_{\bb X^+}  \int_{\bb R}  G \left(z,  \sqrt{m} u \right)  du \nu^+(dz) 
= \frac{c}{\sqrt{m}}  \| G \|_{\nu ^+\otimes \Leb}. 
\end{align}
By integration by parts and a change of variable, we have 
$$
\int_{\mathbb{R}_{+}}  L'_m(s) \left( 1 - \Phi^+(s)  \right)ds
=  \int_{\mathbb{R}_{+}}  H_m(s) \phi^+ \left( \frac{s}{\sqrt{k}} \right) \frac{ds}{\sqrt{k}}.
$$ 
Hence, from \eqref{ApplCondLT-001}, \eqref{ApplCondLT-002} and \eqref{JJJ-20001}, we obtain
\begin{align} \label{Integ_3}
H_{m,k}(z, t) 
& \leq \frac{2 \check V^{g}(z, t)}{\sqrt{2\pi k}  } 
\int_{\mathbb{R}_{+}}  H_m(s) \phi^+ \left( \frac{s}{\sqrt{k}} \right) \frac{ds}{\sqrt{k}} 
  +  \frac{r_k}{ \sqrt{km} }  \| G \|_{\nu ^+\otimes \Leb}.
\end{align}
Implementing this bound into \eqref{JJJ004} 
we get, for any $z \in \mathbb X^+$,
\begin{align}
\label{JJJ201aaa}
\Psi_n (z,t) 
& \leq  \frac{2 \check V^{g}(z, t)}{\sqrt{2\pi } }  I_{m,k}  
  +   \frac{c \ee + r_k}{\sqrt{k m }}  \| G \|_{\nu ^+\otimes \Leb}  
   +  \frac{c_{\ee}}{m\sqrt{k} }  \| G \|_{\mathscr H^+_{\alpha}}, 
\end{align}
where
\begin{align*}
I_{m,k} = \frac{1}{\sqrt{k}} 
 \int_{\mathbb{R}_{+}}  H_m(s) \phi^+ \left( \frac{s}{\sqrt{k}} \right) \frac{ds}{\sqrt{k}}.  
\end{align*}
By the definition of $H_m$ (cf.\ \eqref{JJJ005}) and Fubini's theorem, it follows that
\begin{align*}
I_{m,k} & = 
 \int_{\mathbb{R}_{+}} \int_{\bb X^+}  \int_{\bb R}  \phi_{\sqrt{m}} (u - s)  
  G \left(z, u\right)  du  \nu^+(dz) 
 \phi^+ \left( \frac{s}{\sqrt{k}} \right) \frac{ds}{k} \nonumber \\
& =  \int_{\bb X^+}  \int_{\mathbb R}  G \left(z, u\right)
\left[ \int_{\mathbb{R}_{+}}  \phi_{\sqrt{m}} (u - s)
       \phi^+ \left( \frac{s}{\sqrt{k}} \right)  \frac{ds}{k} \right] du  \nu^+(dz).
\end{align*}
Denote $\delta_n = \frac{m}{n} = \frac{[\delta n]}{n}$. 
By a change of variable, we have 
\begin{align*}
I_{m,k} &= \frac{1}{\sqrt{n}} \int_{\bb X^+}  \int_{\mathbb R}  G \left(z, \sqrt{n} u \right)
\left[ \int_{\mathbb{R}_{+}}
  \phi_{\delta_n} (u - s)
\phi^+_{1 - \delta_n} \left( s \right)  ds \right] du \nu^+(dz) \\
& =  \frac{1}{\sqrt{n}}  \int_{\bb X^+}  \int_{\mathbb R}  G \left(z, \sqrt{n} u \right)
\phi_{\delta_n}*\phi_{1-\delta_n}^+(u) du \nu^+(dz)  \\
& = \frac{ 1 }{n} 
  \int_{\bb X^+}  \int_{\mathbb R}  G \left(z, u \right)
\phi_{\delta_n}*\phi_{1-\delta_n}^+ \left( \frac{u}{\sqrt{n}} \right) du \nu^+(dz).  
\end{align*}
Using Lemma \ref{t-Aux lemma} and the fact that $\sqrt{1 - \delta_n} = \sqrt{\frac{k}{n}}$, it follows that
\begin{align}\label{EqualityPfImkhh}
I_{m,k} &  \leq   \frac{ \sqrt{k} }{n^{3/2}}   \int_{\bb X^+}  \int_{\mathbb R}  G \left(z, u \right)
\phi^+ \left( \frac{u}{\sqrt{n}} \right) du \nu^+(dz)   \notag\\
& \quad  +   \frac{\sqrt{m}}{ n^{3/2}} 
  \int_{\bb X^+}  \int_{\mathbb R}  G \left(z, u \right)  e^{- \frac{u^2}{2m}} du \nu^+(dz)   \notag\\
 & \leq   \frac{ \sqrt{k} }{n^{3/2}}
  \int_{\bb X^+}  \int_{\mathbb R}  G \left(z, u \right)
\phi^+ \left( \frac{u}{\sqrt{n}} \right) du \nu^+(dz)
  +  \frac{\sqrt{m}}{ n^{3/2}}  \| G \|_{\nu ^+\otimes \Leb}.  
\end{align}
Substituting this into \eqref{JJJ201aaa}, and using the fact that $\check V^{g}(z, t) \leq t + c$ gives
\begin{align*}
 \Psi_n (z,t) 
& \leq  \frac{2 \check V^{g}(z, t)}{\sqrt{2\pi} }  \frac{ \sqrt{k} }{n^{3/2}}
  \int_{\bb X^+}  \int_{\mathbb R}  G \left(z, u \right)  \phi^+ \left( \frac{u}{\sqrt{n}} \right) du \nu^+(dz)   \\
& \quad + c \left( \frac{\sqrt{m}}{ n^{3/2}}   +  \frac{c \ee + r_k}{ \sqrt{mk} } \right)  \| G \|_{\nu ^+\otimes \Leb}  
    +  \frac{c_{\ee}}{ m \sqrt{k} }  \| G \|_{\mathscr H^+_{\alpha}}. 
\end{align*}
Since $\ee^{1/2} n \geq  m\geq \frac{1}{2}\ee^{1/2} n$ and $n > k\geq \frac{1}{2}n$,  we obtain 
\begin{align*}
\Psi_n (z,t) 
& \leq  \frac{2 \check V^{g}(z, t)}{\sqrt{2\pi} n}  
  \int_{\bb X^+}  \int_{\mathbb R}  G \left(z, u \right)  \phi^+ \left( \frac{u}{\sqrt{n}} \right) du \nu^+(dz)   \\
& \quad +  \frac{c}{n}  \left( \ee^{1/4}  +   \frac{r_n}{\ee^{1/4}} \right)   \| G \|_{\nu ^+\otimes \Leb}  
    +  \frac{c_{\ee}}{ n^{3/2} }  \| G \|_{\mathscr H^+_{\alpha}}, 
\end{align*}
which finishes the proof of the upper bound \eqref{eqt-A 001}. 

\subsection{Proof of the lower bound}
We now proceed to prove the second assertion \eqref{eqt-A 002} of Theorem \ref{t-A 001}. 
We use the same notation as that in the proof of the upper bound. Recall that $\delta=\sqrt{\ee}$, $m=[\delta n]$ and $k=n-m$. 
For $z \in \bb X^+$, $t \in \bb R$ and $n \geq 1$,  denote
\begin{align*}
\Psi_n (z,t): & = \int_{\bb X^-_z}  F \left( (T^{-n} y \cdot z)_+, t + \check S_n g(y \cdot z) \right) 
   \mathds 1_{ \left\{ \check \tau_t^g (y \cdot z) >n \right\}} \nu^-_z(dy). 
\end{align*} 
By the Markov property, we have that for any $z \in \bb X^+$ and $t \in \bb R$, 
\begin{align}\label{KKK-markov property}
\Psi_n (z,t) = \int_{\bb X^-_z} \Psi_m \left( (T^{-k} y \cdot z)_+, t + \check S_k g(y \cdot z) \right) 
   \mathds 1_{ \left\{ \check \tau_t^g (y \cdot z) > k \right\}}  \nu^-_z(dy). 
\end{align}
We write $\Psi_m$ as a sum of two terms: for any $z \in \bb X^+$ and $t \in \bb R$, 
\begin{align}\label{KKK-1111-000}
\Psi_m (z,t)  =: A_{m} (z,t) - \overline A_{m} (z,t), 
\end{align}
where 
\begin{align}
A_{m} (z,t)   &  =  \int_{\bb X^-_{z}}  
            F \left( (T^{-m} y \cdot z)_+, t + \check S_m g(y \cdot z) \right)  \nu^-_{z}(dy), \label{decompos-psi001}  \\
\overline A_{m} (z,t)  & =  \int_{\bb X^-_{z}}  F \left( (T^{-m} y \cdot z)_+, t + \check S_m g(y \cdot z) \right) 
  \mathds 1_{ \left\{ \check \tau_t^g (y \cdot z) \leq m \right\}}  \nu^-_{z}(dy). \label{decompos-psi002}
\end{align}
This implies that for any $z \in \bb X^+$ and $t \in \bb R$, 
\begin{align} \label{Psi_nee_Decompo}
 \Psi_n(z,t)= J_n(z,t)-K_n(z,t), 
\end{align}
where  
\begin{align}
J_n(z,t)& := \int_{\bb X^-_z}  A_{m} \left( (T^{-k} y \cdot z)_+, t + \check S_k g(y \cdot z) \right) 
   \mathds 1_{ \left\{ \check \tau_t^g (y \cdot z) > k \right\}}  \nu^-_z(dy),  \label{Psi_nee_Decompo001} \\
K_n(z,t)&:=  \int_{\bb X^-_z}  \overline A_{m} \left( (T^{-k} y \cdot z)_+, t + \check S_k g(y \cdot z) \right) 
   \mathds 1_{ \left\{ \check \tau_t^g (y \cdot z) > k \right\}}  \nu^-_z(dy).  \label{Psi_nee_Decompo002}
\end{align}
We proceed to give a lower bound for the term $J_n(z,t)$.
 It can be handled as the case of the upper bound, but here the situation is more complicated. 
 By the local limit theorem (cf.\ Theorem \ref{Lem_LLT_Nonasm}), 
we get that there exist constants $c, c_{\ee} >0$ such that 
for any  $m \geq 1$, $z \in \bb X^+$ and $t \in \bb R$, 
\begin{align} \label{JJJJJ-1111-001B}
 A_{m}(z, t) 
& \geq  H_m(t)
     -   \frac{c \ee}{\sqrt{m}}  \| G \|_{\nu ^+\otimes \Leb}   
       -  \frac{c_{\ee}}{ m }   \left(  \| G \|_{\mathscr H^+_{\alpha}}   +  \| H \|_{\mathscr H^+_{\alpha}} \right), 
\end{align}
where for brevity we set
\begin{align}
\label{JJJ005B}
H_m(t) = \int_{\bb X^+}  \int_{\bb R} \frac{1}{\sqrt{m}} \phi \left( \frac{u - t}{\sqrt{m}} \right)  
  H \left(z, u\right)  du  \nu^+(dz).
\end{align}
Using \eqref{decompos-psi001}, \eqref{Psi_nee_Decompo001} and \eqref{JJJJJ-1111-001B},  
and Lemma \ref{Lem_Markov_chain}, 
we get that for any $z \in \bb X^+$ and $t \in \bb R$, 
\begin{align} \label{JJJ004B}
  J_n (z,t)
& \geq   \int_{\bb X^-_z}  H_{m} \left( t + \check S_k g(y \cdot z) \right) 
   \mathds 1_{ \left\{ \check \tau_t^g (y \cdot z) > k \right\}}  \nu^-_z(dy)  \notag \\  
& \quad  -  \frac{c\ee}{\sqrt{km}}   \| G \|_{\nu ^+\otimes \Leb}  
     -  \frac{c_{\ee}}{\sqrt{k} m }  \left(  \| G \|_{\mathscr H^+_{\alpha}}   +  \| H \|_{\mathscr H^+_{\alpha}} \right).  
\end{align}
For the first term on the right hand side of \eqref{JJJ004B}, 
proceeding in the same way as that in the proof of  \eqref{Integ_3}, 
using the lower bound in the conditioned central limit theorem (see Lemma \ref{Lem_CondiCLT}), 
one can verify that 
\begin{align} \label{Integ_3B}
&  \int_{\bb X^-_z}  H_{m} \left( t + \check S_k g(y \cdot z) \right) 
   \mathds 1_{ \left\{ \check \tau_t^g (y \cdot z) > k \right\}}  \nu^-_z(dy)   \notag\\
& \geq \frac{2 \check V^{g}(z, t)}{\sqrt{2\pi k}  } 
\int_{\mathbb{R}_{+}}  H_m(s) \phi^+ \left( \frac{s}{\sqrt{k}} \right) \frac{ds}{\sqrt{k}}  
   -  \frac{r_k}{ \sqrt{km} } \| H \|_{\nu ^+\otimes \Leb}.
\end{align}
Implementing this bound into \eqref{JJJ004B},
we get that for any $z \in \mathbb X^+$,
\begin{align}
\label{JJJ201aaaB}
  J_n (z,t) 
&  \geq  \frac{2 \check V^{g}(z, t)}{\sqrt{2\pi} }  I_{m,k}  
 -  \frac{r_k}{ \sqrt{km} }  \| H \|_{\nu ^+\otimes \Leb}  \notag\\
& \quad  -  \frac{c\ee}{\sqrt{km}}   \| G \|_{\nu ^+\otimes \Leb}  
     -  \frac{c_{\ee}}{\sqrt{k} m }  \left(  \| G \|_{\mathscr H^+_{\alpha}}   +  \| H \|_{\mathscr H^+_{\alpha}} \right).    
\end{align}
where, in the same way as in the proof of \eqref{EqualityPfImkhh}, 
\begin{align*}
I_{m,k} : & = \int_{\mathbb{R}_{+}}  H_m(s) \phi^+ \left( \frac{s}{\sqrt{k}} \right) \frac{ds}{k}  
 = \frac{ \sqrt{k} }{n^{3/2}}
  \int_{\bb X^+}  \int_{\mathbb R}  H \left(z, u \right)
\phi^+ \left( \frac{u}{\sqrt{n}} \right) du \nu^+(dz).  
\end{align*}
Substituting this into \eqref{JJJ201aaaB}, and using the fact that $H \leq_{\ee} G$ and $\check V^{g}(z, t) \leq t + c$,
we get 
\begin{align*}
  J_n (z,t) 
& \geq  \frac{2 \check V^{g}(z, t)}{\sqrt{2\pi} }  \frac{ \sqrt{k} }{n^{3/2}}
  \int_{\bb X^+}  \int_{\mathbb R}  H \left(z, u \right)
\phi^+ \left( \frac{u}{\sqrt{n}} \right) du \nu^+(dz)   \\
&  \quad  -  \frac{c\ee + r_k}{\sqrt{km}}   \| G \|_{\nu ^+\otimes \Leb}  
     -  \frac{c_{\ee}}{\sqrt{k} m }  \left(  \| G \|_{\mathscr H^+_{\alpha}}   +  \| H \|_{\mathscr H^+_{\alpha}} \right). 
\end{align*}
Since $\sqrt{\frac{n}{k}}\leq 1+ c \ee^{1/4}$, $m\geq \frac{1}{2} \ee^{1/2} n$ and $k\geq \frac{1}{2}n$, 
using again $H \leq_{\ee} G$
we deduce that for $n$ sufficiently large, 
\begin{align} \label{final bounf Jm(x)}
J_n (z,t)
& \geq  \frac{2 \check V^{g}(z, t)}{\sqrt{2\pi} n} 
  \int_{\bb X^+}  \int_{\mathbb R}  H \left(z, u \right)
\phi^+ \left( \frac{u}{\sqrt{n}} \right) du \nu^+(dz)   \notag\\
&  \quad  -  \frac{c}{n}  \left( \ee^{1/4} + \frac{r_n}{\ee^{1/4}}  \right)  \| G \|_{\nu ^+\otimes \Leb}  
     -  \frac{c_{\ee}}{ n^{3/2} }  \left(  \| G \|_{\mathscr H^+_{\alpha}}   +  \| H \|_{\mathscr H^+_{\alpha}} \right). 
\end{align}

The bound of the term $K_n(z,t)$ (cf.\ \eqref{Psi_nee_Decompo002}) 
is rather long and needs to employ the duality (cf.\ Lemma \ref{Corollary_Duality}).  
We start by splitting $K_n(z,t)$ into two parts 
according to whether the values of $t + \check S_k g(y \cdot z)$ 
are less or larger than $\ee \sqrt{n}$: for $z \in \bb X^+$ and $t \in \bb R$, 
\begin{align}\label{KKK-111-001}
K_n(z,t) = K_1 + K_2, 
\end{align}
where
\begin{align*} 
K_1
& = \int_{\bb X^-_z}  \overline A_{m} \left( (T^{-k} y \cdot z)_+, t + \check S_k g(y \cdot z) \right)  
    \mathds 1_{ \left\{  t + \check S_k g(y \cdot z) \leq \ee^{1/6}\sqrt{n} \right\} }
     \mathds 1_{ \left\{ \check \tau_t^g (y \cdot z) > k \right\}}  \nu^-_z(dy),   \\
K_2 
& = \int_{\bb X^-_z}  \overline A_{m} \left( (T^{-k} y \cdot z)_+, t + \check S_k g(y \cdot z) \right)  
    \mathds 1_{ \left\{  t + \check S_k g(y \cdot z) > \ee^{1/6}\sqrt{n} \right\} }
     \mathds 1_{ \left\{ \check \tau_t^g (y \cdot z) > k \right\}}  \nu^-_z(dy). 
\end{align*}
For $K_1$, since $\overline\chi_{\ee} \leq 1$, using the upper bound in the local limit theorem 
(cf.\ Theorem \ref{Lem_LLT_Nonasm}) and taking into account that $\phi \leq 1$,  we get
\begin{align*}
\overline A_{m} (z, t) 
\leq  \frac{L_m(\ee)}{\sqrt{m}}, \quad \mbox{where}  \   
 L_m(\ee) =  c \| G \|_{\nu ^+\otimes \Leb}     +  \frac{c_{\ee}}{ \sqrt{m} }  \| G \|_{\mathscr H^+_{\alpha}}. 
\end{align*}
This, together with \eqref{bounds-reversedindicators-001}, \eqref{identity_for_indicators001}
and the fact that $\sqrt{\frac{n}{k}} \leq c$, 
implies that 
\begin{align*} 
K_1  & \leq  \frac{L_m(\ee)}{\sqrt{m}} 
    \int_{\bb X^-_z}   \mathds 1_{ \left\{  t + \check S_k g(y \cdot z) \leq \ee^{1/6}\sqrt{n} \right\} }
     \mathds 1_{ \left\{ \check \tau_t^g (y \cdot z) > k \right\}}  \nu^-_z(dy)    \\
& \leq  \frac{L_m(\ee)}{\sqrt{m}} \,  \nu^-_z \left(  \frac{t + \check S_k g(y \cdot z)}{\sqrt{k}} \leq c\ee^{1/6} , \, 
  \check \tau_{t}^g (y \cdot z) > k \right). 
\end{align*}
Using Lemma \ref{Lem_CondiCLT}, we get that uniformly in $z \in \bb X^+$, 
\begin{align} \label{K1-final bound}
K_1
& \leq  \frac{L_m(\ee)}{\sqrt{m}}  \, 
\left( \frac{2 \check V^{g}(z, t)}{ \sqrt{2\pi k} } 
  \int_{0}^{ c\ee^{1/6}} \phi^+ \left(t'\right) dt' 
  +  \frac{o(1)}{k^{1/2}} \right) \notag \\
&\leq  \frac{L_m(\ee)}{\sqrt{mk}} \, 
\left(\int_{0}^{ c\ee^{1/6}} \phi^+ \left(t'\right) dt' +  o(1) \right) 
\leq  \frac{L_m(\ee)}{n} \, 
\frac{1}{\ee^{\frac{1}{4}}} \left(c\ee^{1/3}  +  o(1) \right)  \notag \\
&\leq   \frac{L_m(\ee)}{n} \,  \left(c\ee^{\frac{1}{12}}  +  \frac{o(1)}{\ee^{\frac{1}{4}}} \right).
\end{align}

We proceed to give an upper bound for $K_{2}$, see \eqref{KKK-111-001}.
Note that the function 
\begin{align*}
& \overline A_{m} (z,t)   =  \int_{\bb X^-_{z}}  F \left( (T^{-m} y \cdot z)_+, t + \check S_m g(y \cdot z) \right) 
  \mathds 1_{ \left\{ \check \tau_t^g (y \cdot z) \leq m \right\}}  \nu^-_{z}(dy) 
\end{align*}
does not in general belong to the space $\scr H^+_{\alpha}$.  
We start by smoothing the indicator functions in \eqref{Psi_nee_Decompo002} and \eqref{decompos-psi002}. 
Define for $\ee>0$, 
\begin{align*}
\overline A_{m, \ee}  (z,t) := &  \int_{\bb X^-_{z}}   
     G *  \kappa_{\ee/2} \left( (T^{-m} y \cdot z)_+, t + \check S_m g(y \cdot z) \right)   \notag\\
 &  \times
  \overline\chi_{\ee}  \left( t-\ee + \min_{1 \leq  j \leq m}  \check S_j g(y \cdot z) \right)  \nu^-_{z}(dy),
\end{align*}
where $\chi_{\ee}$ is the same as in \eqref{Def_chiee} and $\overline\chi_{\ee} = 1 - \chi_{\ee}$.  
Note that $G * \kappa_{\ee/2}$ $\ee/2$-dominates the function $F$. 
By the identity
\begin{align} \label{identity_for_indicators001}
\mathds 1_{ \left\{ \check \tau_t^g (y \cdot z) > m \right\}} 
= \mathds 1_{[0,\infty)} \left( t + \min_{1 \leq j \leq m}  \check S_j g(y \cdot z)  \right),   
\end{align}  
using the bounds \eqref{bounds-reversedindicators-001} and $F(z, \cdot) \leq  G * \kappa_{\ee/2} (z, \cdot)$,  
 we get that  
the function $\overline A_{m, \ee}$ $\ee/2$-dominates the function $\overline A_{m}$. 
Moreover, by Lemma \ref{Lem_Inequality_Aoverline}, 
there exists a constant $c_{\ee}$ such that for any $m \geq 1$, 
the function $\overline A_{m, \ee}$ belongs to $\scr H^+_{\alpha}$ and satisfies 
\begin{align*}
\| \overline A_{m, \ee} \|_{ \scr H^+_{\alpha} }  \leq c_{\ee}  \| G\|_{ \scr H^+_{\alpha} },
\qquad  
\| \overline A_{m, \ee} \|_{\nu^+ \otimes \Leb }  \leq   \| G \|_{\nu^+ \otimes \Leb }. 
\end{align*}
Denote 
\begin{align}\label{Def_m_zt}
W_{m,\ee} (z, t) = \overline A_{m, \ee} \left( z, t  \right)   \mathds 1_{ \{ t \geq \ee^{1/6}\sqrt{n} \} }. 
\end{align}
Using the upper bound \eqref{eqt-A 001} and the fact that $\phi^+ \leq 1$, we obtain 
\begin{align}\label{eqt-A 001_low_bb}
K_{2} & \leq \int_{\bb X^-_z}  W_m \left( (T^{-k} y \cdot z)_+,  t + \check S_k g(y \cdot z)  \right) 
   \mathds 1_{ \left\{ \check \tau_t^g (y \cdot z) >k \right\}} \nu^-_z(dy)    \nonumber\\
& \leq  \left(\frac{2 \check V^{g}(z, t)}{\sqrt{2\pi} k}  + \frac{c}{k} \left( \ee^{1/4}  +   \frac{r_n}{\ee^{1/4}} \right) \right)
   \| W_{m,\ee} \|_{\nu^+ \otimes \Leb }
    +   \frac{c_{\ee}}{ \sqrt{n} k}  \|  W_{m,\ee} \|_{\mathscr H^+_{\alpha}}.  
\end{align}
For the first term on the right hand side of \eqref{eqt-A 001_low_bb}, 
by the definition of $W_m$ and Fubini's theorem, we have
\begin{align*}
  \int_{\bb X^+}  \int_{\mathbb R}  W_{m,\ee} \left(z', u \right)  & du \nu^+(dz')   
\leq  \int_{\bb X}  \int_{\mathbb R}  
  \bigg[   G *  \kappa_{\ee/2} \left( T^{-m} x,  u + \check S_m g(x) \right)   \notag\\
& \qquad\qquad  \times  \mathds 1_{ \left\{   u  + \min_{1 \leq  j \leq m}  \check S_j g(x)  \leq 0 \right\} } 
\mathds 1_{ \{ u \geq \ee^{1/6}\sqrt{n} \} }    \bigg]  \nu(dx)  du   =: U. 
\end{align*}
Using the duality (Lemma \ref{Corollary_Duality}) yields that
\begin{align*}
& U  =  \int_{\bb X}  \int_{\mathbb R}  
  \bigg[   G *  \kappa_{\ee/2} \left( x',  u'  \right)    \mathds 1_{ \left\{ u' - \min_{1 \leq j \leq m }  S_j g(x')   \leq 0  \right\}  }  
\mathds 1_{ \left\{ u' - S_m g(x') \geq  \ee^{1/6}\sqrt{n} \right\}  } \bigg]  \nu(dx')  du'.  
\end{align*}
Since the measure $\nu$ is $T$-invariant, it follows that 
\begin{align}\label{Lower_F_ee_kkk}
U  & =  \int_{\bb X}  \int_{\mathbb R}  
  \bigg[   G *  \kappa_{\ee/2} \left( T^{-m} x',  u'  \right)    \mathds 1_{ \left\{ u' - \min_{1 \leq j \leq m }  S_j g(T^{-m} x')   \leq 0  \right\}  } 
     \notag\\ 
& \qquad \qquad\qquad
   \times  \mathds 1_{ \left\{ u' - S_m g(T^{-m} x') \geq  \ee^{1/6}\sqrt{n} \right\}  } \bigg]  \nu(dx')  du'  \notag\\
& =  \int_{\bb X}  \int_{\mathbb R}  
  \bigg[   G *  \kappa_{\ee/2} \left( T^{-m} x',  u'  \right)    
  \mathds 1_{ \left\{ u' - \check S_{m} g(x') + \max_{1 \leq j \leq m } ( \check S_{m-j} g(x') )  \leq 0  \right\}  }       \notag\\ 
& \qquad \qquad\qquad
   \times  \mathds 1_{ \left\{ u' - \check S_m g(x') \geq  \ee^{1/6}\sqrt{n} \right\}  } \bigg]  \nu(dx')  du'  \notag\\
& \leq    \int_{\mathbb R}    \| G *  \kappa_{\ee/2} \left( \cdot,  u'  \right)\|_{\infty} du'  
 \,  \nu \left( x' \in \bb X:     \max_{1 \leq j \leq m } (\check S_{j} g(x') )  \leq - \ee^{1/6}\sqrt{n}  \right). 
\end{align}
By Propositions \ref{lemma-martingale001} and \ref{lemma-martingale002},  
there exists a H\"older continuous function $g_0$ on $\bb X^+$ satisfying $\mathcal L_{\psi} g_0 = 0$
such that $\{ \check S_k g_0 (y' \cdot z) \}_{k \geq 0}$ is a martingale
and $\sup_{k \geq 0} | \check S_k g_0 - \check S_k g| \leq c$ for some constant $c>0$. 
Therefore, for any $t' \geq  \ee^{1/6}\sqrt{n}$, 
by  Fuk's inequality for martingales (Lemma \ref{Lem_Fuk}), 
we have that with $u = \ee^{1/6}\sqrt{ n} $ and $v =  c \ee^{1/3} \sqrt{n}$ ($c$ is a sufficiently large constant), 
uniformly in $z \in \bb X^+$, 
\begin{align} \label{Fuk-Nagaev-001}
  \nu^-_{z} \left( y \in \bb X^-_z:  \max_{1 \leq j \leq m}  | \check S_j g(y \cdot z)| \geq   u  \right) 
& \leq  \exp\left(- \frac{u}{v} \log \left(1 + \frac{uv}{c[\frac{m}{2}]}\right)\right)   \notag\\
& \leq  \exp\left( - \frac{c}{\ee^{1/6}} \right),  
\end{align} 
which implies that
\begin{align}\label{Proba_002}
  \nu \left( x' \in \bb X:   \max_{1 \leq j \leq m}  | \check S_j g(x')| \geq   u  \right)  
 \leq  \exp\left( - \frac{c}{\ee^{1/6}} \right).  
\end{align}
Combining \eqref{Proba_002} and \eqref{Lower_F_ee_kkk},  we obtain
\begin{align}\label{Bound_n_U}
U & \leq  \exp\left(- c \ee^{-1/6} \right)    
   \int_{\mathbb R}    \| G *  \kappa_{\ee/2} \left( \cdot,  u'  \right)\|_{\infty} du'   \notag\\
& \leq  \exp\left(- c \ee^{-1/6} \right)   \int_{\mathbb R}    \| G \left( \cdot,  u'  \right)\|_{\infty} du'. 
\end{align}
For the second term on the right hand side of \eqref{eqt-A 001_low_bb},
by \eqref{Def_m_zt} and  Lemma \ref{Lem_Inequality_Aoverline}, we get
\begin{align}\label{Bound_n_Wmee}
\|  W_{m,\ee} \|_{\mathscr H^+_{\alpha}}  
\leq \| \overline A_{m, \ee} \|_{ \scr H^+_{\alpha} }
\leq  c_\ee \|  G \|_{\mathscr H^+_{\alpha}}. 
\end{align}
Therefore, from \eqref{eqt-A 001_low_bb}, \eqref{Bound_n_U} and \eqref{Bound_n_Wmee},  
we derive the upper bound for $K_2$:
\begin{align}\label{Final_Bound_K2}
K_{2} 
& \leq  \left(\frac{2 \check V^{g}(z, t)}{\sqrt{2\pi} n}  + \frac{c}{n} \left( \ee^{1/4}  +   \frac{r_n}{\ee^{1/4}} \right) \right)
 \exp\left(- c \ee^{-1/6} \right)   \int_{\mathbb R}    \| G \left( \cdot,  u'  \right)\|_{\infty} du'    \nonumber\\
& \quad  +   \frac{c_{\ee}}{ n^{3/2} }  \|  G \|_{\mathscr H^+_{\alpha}}.  
\end{align}
Combining \eqref{final bounf Jm(x)}, \eqref{K1-final bound} and \eqref{Final_Bound_K2}, the lower bound \eqref{eqt-A 002} follows.

\section{Proof of Theorem \ref{Thm-conLLT}} 

As for Theorems \ref{Thm-exit-time001} and \ref{Thm-cnodi-lim-theor001},  
we first establish the result when $g$ is in $\scr B^+$.  
The general case of a function $g$ in $\scr B$ will follow using the same method as in Section \ref{Sect_CCLT}.

\begin{theorem}\label{Lem_CondiLLT_001}
Let $g \in \scr B^+$ be such that  $\nu^+ (g) = 0$.
Assume that for any $p \neq 0$ and $q \in \bb R$, 
the function $p g + q$ is not cohomologous to a function with values in $\bb Z$.  
Let $F$ be a continuous compactly supported function on $\bb X^+ \times \bb R$. 
Then, we have, uniformly in $z \in \bb X^+$ and $t$ in a compact subset of $\bb R$,  
\begin{align*}
& \lim_{n \to \infty}  n^{3/2} \int_{\bb X^-_z}  F \left( (T^{-n} y \cdot z)_+,  t + \check S_n g(y \cdot z)  \right) 
   \mathds 1_{ \left\{ \check \tau_t^g (y \cdot z) >n - 1 \right\}} \nu^-_z(dy)    \nonumber\\
& =   \frac{2 \check V^{g}(z, t)}{\sqrt{2\pi} \sigma_g^3 }  
  \int_{\bb X \times \mathbb R}  F \left(x_+', t' \right)  \mu^{(-g)}(dx', dt'). 
 \end{align*}
\end{theorem}

In the proof of this theorem, we will make use of several technical lemmas which are stated below. 
We say that a function $G$ on $\bb X^+ \times \bb R$
is $\alpha$-regular if there is a constant $c$ such that for any $(z, t)$ and $(z', t')$ in $\bb X^+ \times \bb R$,
we have $|G(z,t) - G(z',t')| \leq c( |t-t'| + \alpha^{\omega(z,z')} )$. 
In other words, a function is $\alpha$-regular if and only if it is Lipschitz continuous on $\bb X^+ \times \bb R$
when $\bb R$ is equipped with the standard distance and $\bb X^+$ is equipped with the distance 
$(z, z') \mapsto \alpha^{\omega(z,z')}$. 
The following result is similar to Lemma \ref{Lem_Inequality_Aoverline}. 
It will allow us to smooth certain functions appearing in the proof of Theorem \ref{Lem_CondiLLT_001}
in order to be able to apply Theorem \ref{t-A 001}. 
Recall that for $\ee \in (0,1)$,  
 $\chi_{\ee} (u) = 0$ for $u \leq -\ee$, 
$\chi_{\ee} (u) = \frac{u+\ee}{\ee}$ for $u \in (-\ee,0)$, 
and $\chi_{\ee} (u) = 1$ for $u \geq 0$. 

\begin{lemma}\label{Lem_HolderNormPsi}
Let $\alpha \in (0,1)$ and $g \in \scr B^+_{\alpha}$ be such that  $\nu^+ (g) = 0$.
Assume that $g$ is not a coboundary. 
Let $G$ be an $\alpha$-regular function with compact support on $\bb X^+ \times \bb R$.
For $(z,t) \in \bb X^+ \times \bb R$, $m \geq 1$ and $\ee >0$, define 
\begin{align*}
\overline \Psi_{m,\ee}(z, t) :  = 
\int_{\bb X^-_{z}}  & G \left( (T^{-m} y \cdot z)_+, t + \check S_m g(y \cdot z) \right)     \notag\\
 & \qquad \times   \chi_{\ee}  \left( t + \ee + \min_{1 \leq  j \leq m}  \check S_j g(y \cdot z) \right)  \nu^-_{z}(dy). 
\end{align*}
Then $\overline \Psi_{m,\ee} \in \scr H_{\alpha}^+$ and  $\| \overline \Psi_{m,\ee} \|_{\scr H_{\alpha}^+} \leq \frac{c}{\ee \sqrt{m}}.$
\end{lemma}

\begin{proof}
Recall that 
\begin{align*} 
\|  \overline \Psi_{m,\ee} \|_{\mathscr H^+_{\alpha}}
 =  \int_{\bb R}   \sup_{z \in \bb X^+}  | \overline \Psi_{m,\ee} \left(z, t \right)|  dt
 +  \int_{\bb R} \sup_{z, z' \in \bb X^+}  
    \frac{| \overline \Psi_{m,\ee} \left(z, t \right)  - \overline \Psi_{m,\ee} \left(z', t \right)  |}{\alpha^{\omega(z, z')}} dt. 
\end{align*}
By Corollary \ref{Cor_CoarseBound},  the first term is dominated by $c/\sqrt{m}$ for some constant $c>0$. 

For the second term, we start by noticing that
by Lemma \ref{Lem_sum_Ine}, there exists a constant $c_0>0$ such that
 for any $z, z' \in \bb X^+$ with $z_0 = z_0'$, 
$t \in \bb R$ and $y \in \bb X^-_{z}$,
\begin{align*}
& \chi_{\ee}  \left( t + \ee + \min_{1 \leq  j \leq m}  \check S_j g(y \cdot z) \right) 
  - \chi_{\ee}  \left( t  + \ee+ \min_{1 \leq  j \leq m}  \check S_j g(y \cdot z') \right)  \notag\\
& = \left( \chi_{\ee}  \left( t + \ee + \min_{1 \leq  j \leq m}  \check S_j g(y \cdot z) \right) 
  - \chi_{\ee}  \left( t + \ee  + \min_{1 \leq  j \leq m}  \check S_j g(y \cdot z') \right)  \right)  \notag\\
& \quad \times  \mathds 1_{\{ t  + \min_{1 \leq  j \leq m}  \check S_j g(y \cdot z) \geq - c_0 \}}   \notag\\
& \leq \frac{1}{\ee} 
\left|  \min_{1 \leq  j \leq m}  \check S_j g(y \cdot z) - \min_{1 \leq  j \leq m}  \check S_j g(y \cdot z') \right| 
\mathds 1_{\{ t  + \min_{1 \leq  j \leq m}  \check S_j g(y \cdot z) \geq - c_0 \}}   \notag\\
& \leq \frac{c_1}{\ee}  \alpha^{\omega(z,z')}  
  \mathds 1_{\{ t  + \min_{1 \leq  j \leq m}  \check S_j g(y \cdot z) \geq - c_0 \}},
\end{align*}
where in the last inequality we used Corollary \ref{Cor_Holder_minimum}.  
It follows that
\begin{align*}
& \int_{\bb X^-_{z}}   G \left( (T^{-m} y \cdot z)_+, t + \check S_m g(y \cdot z) \right)     \notag\\
 & \quad \times  \left| \chi_{\ee}  \left( t + \ee + \min_{1 \leq  j \leq m}  \check S_j g(y \cdot z) \right) - 
 \chi_{\ee}  \left( t + \ee + \min_{1 \leq  j \leq m}  \check S_j g(y \cdot z') \right) \right| \nu^-_{z}(dy)   \notag\\
& \leq  \frac{c_1}{\ee}  \alpha^{\omega(z,z')}
\int_{\bb X^-_{z}}  G \left( (T^{-m} y \cdot z)_+, t + \check S_m g(y \cdot z) \right) 
\mathds 1_{\{ t  + \min_{1 \leq  j \leq m}  \check S_j g(y \cdot z) \geq - c_0 \}} \nu^-_{z}(dy).
\end{align*}
By using again Corollary \ref{Cor_CoarseBound}, we get
\begin{align}\label{HolderBound1}
& \int_{\bb R} \sup_{z, z' \in \bb X^+: z_0 = z_0'}  \alpha^{-\omega(z,z')}
\int_{\bb X^-_{z}}   G \left( (T^{-m} y \cdot z)_+, t + \check S_m g(y \cdot z) \right)     \notag\\
 & \quad \times  \left| \chi_{\ee}  \left( t + \ee + \min_{1 \leq  j \leq m}  \check S_j g(y \cdot z) \right) - 
 \chi_{\ee}  \left( t + \ee + \min_{1 \leq  j \leq m}  \check S_j g(y \cdot z') \right) \right| \nu^-_{z}(dy)   \notag\\
 & \leq \frac{c_2}{\ee \sqrt{m}}.  
\end{align}
Besides, as $G$ is $\alpha$-regular and has compact support, we have 
 for any $z, z' \in \bb X^+$ with $z_0 = z_0'$, and $t \in \bb R$, 
 by Lemma \ref{Lem_sum_Ine},  
\begin{align*}
&  \int_{\bb X^-_{z}}  
\left| G \left( (T^{-m} y \cdot z)_+, t + \check S_m g(y \cdot z) \right)  
   -  G \left( (T^{-m} y \cdot z')_+, t + \check S_m g(y \cdot z') \right)  \right|  \notag\\
 & \qquad \times   \chi_{\ee}  
   \left( t + \ee + \min_{1 \leq  j \leq m}  \check S_j g(y \cdot z') \right)  \nu^-_{z}(dy)  \notag\\
& \leq c_3 \alpha^{\omega(z,z')} H(t)  
\nu^-_{z}  \left( y \in \bb X^-_{z}: t + c' + \min_{1 \leq  j \leq m}  \check S_j g(y \cdot z) \right), 
\end{align*}
for some compactly supported continuous function $H$ on $\bb R$.  
Again by Corollary \ref{Cor_CoarseBound}, we get 
\begin{align}\label{HolderBound2}
& \int_{\bb R}  \sup_{z, z' \in \bb X^+: z_0 = z_0'}  \alpha^{-\omega(z,z')}  \notag\\
& \quad \times  \int_{\bb X^-_{z}}  
\left| G \left( (T^{-m} y \cdot z)_+, t + \check S_m g(y \cdot z) \right)  
   -  G \left( (T^{-m} y \cdot z')_+, t + \check S_m g(y \cdot z') \right)  \right|   \notag\\
   & \quad \times   \chi_{\ee}  
   \left( t + \ee + \min_{1 \leq  j \leq m}  \check S_j g(y \cdot z') \right)  \nu^-_{z}(dy)  dt 
 \leq \frac{c_4}{\sqrt{m}}.  
\end{align}
Finally, for any $z, z' \in \bb X^+$ with $z_0 = z_0'$, 
$t \in \bb R$, we have 
\begin{align*}
& \int_{\bb X^-_{z}}  G \left( (T^{-m} y \cdot z')_+, t + \check S_m g(y \cdot z') \right) 
\chi_{\ee}  \left( t + \ee + \min_{1 \leq  j \leq m}  \check S_j g(y \cdot z') \right)  \nu^-_{z}(dy)  \notag\\
& = \int_{\bb X^-_{z'}}  G \left( (T^{-m} y \cdot z')_+, t + \check S_m g(y \cdot z') \right) 
\chi_{\ee}  \left( t + \ee + \min_{1 \leq  j \leq m}  \check S_j g(y \cdot z') \right) e^{\theta(y, z', z)}  \nu^-_{z'}(dy), 
\end{align*}
where $\theta$ is as in Lemma \ref{Lem_Absolute_Contin}.  
By the H\"older continuous domination of $\theta$ in Lemma \ref{Lem_Absolute_Contin}, we derive that
\begin{align}\label{HolderBound3}
& \int_{\bb R} \sup_{z, z' \in \bb X^+: z_0 = z_0'}  \alpha^{-\omega(z,z')}   \notag\\
 & \quad  \bigg| \int_{\bb X^-_{z}}  G \left( (T^{-m} y \cdot z')_+, t + \check S_m g(y \cdot z') \right) 
\chi_{\ee}  \left( t + \ee + \min_{1 \leq  j \leq m}  \check S_j g(y \cdot z') \right)  \nu^-_{z}(dy)   \notag\\
&  -  \int_{\bb X^-_{z'}}  G \left( (T^{-m} y \cdot z')_+, t + \check S_m g(y \cdot z')  \right)
\chi_{\ee}  \left( t + \ee + \min_{1 \leq  j \leq m}  \check S_j g(y \cdot z') \right)  \nu^-_{z'}(dy)  \bigg| dt  \notag\\
& \leq c_4    \int_{\bb R} 
  \int_{\bb X^-_{z'}}  G \left( (T^{-m} y \cdot z')_+, t + \check S_m g(y \cdot z') \right) 
\chi_{\ee}  \left( t + \ee + \min_{1 \leq  j \leq m}  \check S_j g(y \cdot z') \right)  \nu^-_{z'}(dy) dt  \notag\\
& \leq  \frac{c_5}{\sqrt{m}} \alpha^{\omega(z,z')}, 
\end{align}
where the last inequality follows from Corollary \ref{Cor_CoarseBound}. 
Putting together \eqref{HolderBound1}, \eqref{HolderBound2} and \eqref{HolderBound3} gives
\begin{align*}
\int_{\bb R} \sup_{z, z' \in \bb X^+: z_0 = z_0'}  
    \frac{| \overline \Psi_{m,\ee} \left(z, t \right)  - \overline \Psi_{m,\ee} \left(z', t \right)  |}{\alpha^{\omega(z, z')}} dt
    \leq  \frac{c_6 }{\ee \sqrt{m}}. 
\end{align*}
The lemma follows. 
\end{proof}

Now we write a technical version of Theorem \ref{Lem_CondiLLT_001}. 

\begin{lemma} \label{t-BB001}
Let $\alpha \in (0,1)$ and $g \in \scr B^+_{\alpha}$ be such that  $\nu^+ (g) = 0$.
Assume that for any $p \neq 0$ and $q \in \bb R$, 
the function $p g + q$ is not cohomologous to a function with values in $\bb Z$.  
Let $t \in \bb R$. 
Then,  for any $\ee \in (0,\frac{1}{8})$ and $z \in \bb X^+$,  
and for any non-negative  function $F$ 
and non-negative $\alpha$-regular compactly supported functions $G, H$ 
satisfying $H \leq_{\ee} F \leq_{\ee} G$, 
we have 
\begin{align}\label{eqt-BB001}
& \limsup_{n \to \infty} n^{3/2} \int_{\bb X^-_z}  F \left( (T^{-n} y \cdot z)_+,  t + \check S_n g(y \cdot z)  \right) 
   \mathds 1_{ \left\{ \check \tau_t^g (y \cdot z) >n - 1 \right\}} \nu^-_z(dy)    \nonumber\\
& \leq   \frac{2 \check V^{g}(z, t)}{\sigma_g^3 \sqrt{2\pi} }  
  \int_{\bb X} \int_{\bb R}   G(x_+, t)  \mu^{(-g)}(dx, dt)
\end{align}
and
\begin{align} \label{eqt-BB002}
&  \liminf_{n \to \infty} n^{3/2} \int_{\bb X^-_z}  F \left( (T^{-n} y \cdot z)_+,  t + \check S_n g(y \cdot z)  \right) 
   \mathds 1_{ \left\{ \check \tau_t^g (y \cdot z) >n -1 \right\}} \nu^-_z(dy)    \nonumber\\
&  \geq   \frac{2 \check V^{g}(z, t)}{\sigma_g^3 \sqrt{2\pi} }  
  \int_{\bb X} \int_{\bb R}   H(x_+, t)  \mu^{(-g)}(dx, dt). 
\end{align}
\end{lemma}

\begin{proof}
We first prove \eqref{eqt-BB001}.  
As in \eqref{KKK-markov property}, denote, for $z \in \bb X^+$ and $t \in \bb R$,  
\begin{align*}
\Psi_n (z,t) = \int_{\bb X^-_z}  F \left( (T^{-n} y \cdot z)_+, t + \check S_n g(y \cdot z) \right) 
    \mathds 1_{ \left\{ \check \tau_t^g (y \cdot z) >n - 1 \right\}} \nu^-_z(dy). 
\end{align*}
Set $m=\left[ n/2 \right]$ and $k = n-m.$   
By the Markov property we have that for any $z \in \bb X^+$ and $t \in \bb R$, 
\begin{align}\label{JJJ-markov property_thmB}
\Psi_n (z,t) = \int_{\bb X^-_z} \Psi_m \left( (T^{-k} y \cdot z)_+, t + \check S_k g(y \cdot z) \right) 
  \mathds 1_{ \left\{ \check \tau_t^g (y \cdot z) > k \right\}} 
 \nu^-_z(dy). 
\end{align}
For any $z' \in \bb X^+$ and $t' \in \bb R$, 
we set 
\begin{align*} 
\overline \Psi_m(z', t') :  = 
\int_{\bb X^-_{z'}}  & G \left( (T^{-m} y \cdot z')_+, t' + \check S_m g(y \cdot z') \right)     \notag\\
 & \qquad \times   \chi_{\ee}  \left( t' + \ee + \min_{1 \leq  j \leq m}  \check S_j g(y \cdot z') \right)  \nu^-_{z'}(dy). 
\end{align*}
By using $F \leq_{\ee} G$, we get that $\Psi_m \leq_{\ee}  \overline \Psi_m$.
Note that by Lemma \ref{Lem_HolderNormPsi}, 
the function $\overline \Psi_m$ belongs to the space $\mathscr H_{\alpha}^+$,
so that we are exactly in the setting of Theorem \ref{t-A 001}.  
Therefore, using the bound \eqref{eqt-A 001} of Theorem \ref{t-A 001}, we get 
\begin{align*}
\Psi_n (z,t) 
& \leq  \frac{2 \check V^{g}(z, t)}{\sigma_g^2 \sqrt{2\pi} k } 
  \int_{\bb X^+}  \int_{\mathbb R_+}  \overline \Psi_m  \left(z', u' \right)
\phi^+ \left( \frac{u'}{\sigma_g \sqrt{k}} \right) du' \nu^+(dz')   \nonumber\\
& \quad + \frac{c}{k} \left( \ee^{1/4}  +   \frac{r_k}{\ee^{1/4}} \right)  \| \overline \Psi_m \|_{\nu^+ \otimes \Leb }  
    +   \frac{c_{\ee}}{ k^{3/2} }  \|  \overline \Psi_m \|_{\mathscr H^+_{\alpha}}   \notag\\
&  =: J_1 + J_2 + J_3. 
\end{align*}

For $J_1$, applying the duality (Lemma \ref{Corollary_Duality}), we deduce that
\begin{align*}
&  \int_{\bb X^+}  \int_{\mathbb R_+}  \overline \Psi_m  \left(z, u \right)
\phi^+ \left( \frac{u}{\sigma_g \sqrt{k}} \right) du \nu^+(dz)  \\
& \leq   \int_{\bb X^+} \int_{\bb R_+}
\int_{\bb X^-_{z}}  G \left( (T^{-m} y \cdot z)_+,  u + \check S_m g(y \cdot z) \right)   \notag\\
& \qquad  \times     \mathds 1_{ \left\{ \check \tau_{u + 2\ee}^g (y \cdot z) > m - 1 \right\}} \nu^-_{z}(dy) 
     \phi^+ \left( \frac{u}{\sigma_g \sqrt{k}} \right)  du \nu^+(dz) \notag\\
 & =   \int_{\bb X} \int_{\bb R_+}
   G \left( (T^{-m} x)_+,  u + \check S_m g(x) \right) \phi^+ \left( \frac{u}{\sigma_g \sqrt{k}} \right)
    \mathds 1_{ \left\{ \check \tau_{u + 2\ee}^g (x) > m - 1 \right\}} 
       du  \nu(dx)   \notag\\
  & =  \int_{\bb X} \int_{\bb R}
   G \left( x_+,  t - \ee \right)  \phi^+ \left( \frac{t - S_m g(x) - 2\ee }{\sigma_g \sqrt{k}} \right)
      \mathds 1_{ \left\{  \tau_{t}^{-g} (x) > m - 1 \right\}}    dt  \nu(dx).  
\end{align*}
Using the conditioned central limit theorem (Theorem \ref{Thm-cnodi-lim-theor001}),
we get 
\begin{align*}
& \lim_{n \to \infty}  \sigma_g  \sqrt{2 \pi m} 
\int_{\bb X} \int_{\bb R} 
   G \left( x_+,  t - 2\ee \right)  \phi^+ \left( \frac{t - S_m g(x) - \ee }{\sigma_g \sqrt{k}} \right)
      \mathds 1_{ \left\{  \tau_{t }^{-g} (x) > m - 1 \right\}}    dt  \nu(dx)  \notag\\
 & =   2  \int_{\bb X}  \int_{\bb R}   G(x_+, t - 2\ee)  \mu^{-g}(dx, dt)
       \int_{\bb R_+}   \left(  \phi^+(t')  \right)^2 dt'   \notag\\
 & = \frac{\sqrt{\pi}}{2} \int_{\bb X} \int_{\bb R}   G(x_+, t - 2\ee)  \mu^{(-g)}(dx, dt).  
\end{align*}
Therefore, we obtain 
\begin{align*}
\lim_{n \to \infty} n^{3/2} J_1
=  \frac{2 \check V^{g}(z, t)}{\sigma_g^3 \sqrt{2\pi} }  
  \int_{\bb X} \int_{\bb R}   G(x_+, t - 2\ee)  \mu^{(-g)}(dx, dt).  
\end{align*}

For $J_2$,  by Corollary \ref{Cor_CoarseBound}, we have 
\begin{align*} 
 \| \overline \Psi_m \|_{\nu^+ \otimes \Leb }  
 \leq \frac{c}{\sqrt{m}}. 
\end{align*}
Taking into account that $m=[n/2]$ and $k=n-m, $ we get $\limsup_{n \to \infty} n^{3/2} J_2   \leq c \ee^{1/4}.$

For $J_3$, by Lemma \ref{Lem_HolderNormPsi}, we have $\lim_{n \to \infty} n^{3/2} J_3 = 0$. 
This finishes the proof of the upper bound. 
The proof of the lower bound can be carried out in the same way. 
\end{proof}

From Lemma \ref{t-BB001}, we get Theorem \ref{Lem_CondiLLT_001} by a standard approximation procedure.

\begin{lemma}\label{Lem_Approximation_FGH}
Fix $\alpha \in (0,1)$. 
Let $F$ be a non-negative continuous compactly supported function on $\bb X^+ \times \bb R$.
Then, there exist a decreasing sequence $(G_n)_{n \geq 1}$ and an increasing sequence $(H_n)_{n \geq 1}$
of compactly supported $\alpha$-regular functions, 
such that $H_n \leq_{1/n} F \leq_{1/n} G_n$ for any $n \geq 1$, and 
$G_n$ and $H_n$ converge uniformly to $F$ as $n \to \infty$. 
\end{lemma}

\begin{proof}[Proof of Theorem \ref{Lem_CondiLLT_001}]
This follows directly from Lemmas \ref{t-BB001} and \ref{Lem_Approximation_FGH}.  
\end{proof}

From Theorem \ref{Lem_CondiLLT_001}
 we deduce a new lemma in which the target function $F$ may depend on the past coordinates.

\begin{lemma}\label{Lem_CondiLLT_thm1}
Let $g \in \scr B^+$ be such that  $\nu^+ (g) = 0$.
Assume that for any $p \neq 0$ and $q \in \bb R$, 
the function $p g + q$ is not cohomologous to a function with values in $\bb Z$.  
Let 
$F$ be a continuous compactly supported function on $\bb X \times \bb R$. 
Then, uniformly in $z \in \bb X_+$ and $t$ in a compact subset of $\bb R$,  
\begin{align*}
& \lim_{n \to \infty}  n^{3/2} \int_{\bb X^-_z}  F \left( T^{-n} (y \cdot z),  t + \check S_n g(y \cdot z)  \right) 
   \mathds 1_{ \left\{ \check \tau_t^g (y \cdot z) >n - 1 \right\}} \nu^-_z(dy)    \nonumber\\
& =   \frac{2 \check V^{g}(z, t)}{\sqrt{2\pi} \sigma_g^3}  
  \int_{\bb X \times \mathbb R}  F \left(x', t' \right)  \mu^{(-g)}(dx', dt'). 
  \end{align*}

\end{lemma}

\begin{proof}
As in the proof of Lemma \ref{Lem_CondiCLT_target1}, 
it suffices to prove this result when $F$ is of the form 
$(x,t) \mapsto \mathds 1_{\{x_{-m} = a_{-m}, \ldots, x_{-1} = a_{-1} \} } F_1(x_+, t)$,
where $a_{-m}, \ldots, a_{-1}$ are in $A$
such that  $M(a_{i-1}, a_{i}) = 1$ for $-m+1 \leq i \leq -1$,
 and $F_1$ is a continuous compactly supported function on 
$\bb X^+\times \bb R$. 
For such a function, we have 
\begin{align*}
& \int_{\bb X^-_z} 
F \left( T^{-n}(y \cdot z),  \check S_n g (y \cdot z) \right)
    \mathds 1_{ \left\{ \check \tau_t^g(y \cdot z) >n -1 \right\} }   \nu^-_z(dy)   \notag\\
& =  \int_{\bb X^-_z} 
F_2 \left( \left( T^{-n}(y \cdot z) \right)_+,  \check S_n g (y \cdot z)  \right)
    \mathds 1_{ \left\{ \check \tau_t^g(y \cdot z) >n -1   \right\} }   \nu^-_z(dy), 
\end{align*}
where, for $(z', t') \in \bb X^+ \times \bb R$, 
\begin{align*}
F_2(z',t') =  \exp(- S_m \psi(a_{-m} \ldots a_{-1} \cdot z')) F_1(z', t')
=  \int_{\bb X_{z'}^-} F(y\cdot z', t') \nu^-_{z'}(dy). 
\end{align*}
Since $F_2(\cdot, t')$ depends only on the future, we can apply Theorem \ref{Lem_CondiLLT_001}, 
which gives 
\begin{align*}
\lim_{n \to \infty}  n^{3/2}   \int_{\bb X^-_z} 
&F_2 \left( \left( T^{-n}(y \cdot z) \right)_+,  \check S_n g (y \cdot z) \right)
    \mathds 1_{ \left\{ \check \tau_t^g(y \cdot z) >n - 1 \right\} }   \nu^-_z(dy) \nonumber \\
&= \frac{2 \check V^{g}(z, t)}{\sqrt{2\pi} \sigma_g^3}  
  \int_{\bb X \times \mathbb R}  F_2 \left(x'_+, t' \right)  \mu^{(-g)}(dx', dt'). 
\end{align*}
To conclude, it remains to show that 
\begin{align*}
\int_{\bb X \times \mathbb R}  F_2 \left(x_+, t \right)  \mu^{(-g)}(dx, dt)
= \int_{\bb X \times \mathbb R}  F \left(x, t \right)  \mu^{(-g)}(dx, dt). 
\end{align*}
Indeed, by the definition of the measure $\mu^{(-g)}$ (see Theorem \ref{Thm_Radon_Measure}) 
and by using Lemma \ref{Lem_Fubini},  we get
\begin{align*}
& \int_{\bb X \times \mathbb R}  F \left(x, t \right)  \mu^{(-g)}(dx, dt)  \notag\\
& = \lim_{n \to \infty} \int_{\bb X \times \mathbb R}  
  F \left(x, t \right) (- S_n g(x) ) \mathds 1_{\{ \tau_t^{(-g)} (x) > n \}} \nu(dx) dt  \notag\\
& =  \lim_{n \to \infty} \int_{\bb X_+} \int_{\mathbb R} (- S_n g(z) )  \mathds 1_{\{ \tau_t^{(-g)}(z) > n \}} 
   \int_{\bb X_z^-} F(y\cdot z, t) \nu^-_z(dy) dt  \nu^+(dz)  \notag\\
& =  \lim_{n \to \infty} \int_{\bb X_+} \int_{\mathbb R} (-S_n g(z) )  \mathds 1_{\{ \tau_t^{(-g)}(z) > n \}}  
   F_2(z,t)  dt  \nu^+(dz)  \notag\\
& =  \int_{\bb X \times \mathbb R}  F_2 \left(x_+, t \right)  \mu^{(-g)}(dx, dt), 
\end{align*}
which ends the proof of the lemma. 
\end{proof}

Now we will put a target on the starting point $y \in \bb X^-_z$. 

\begin{lemma}\label{Lem_CondiCLLT003}
Let $g \in \scr B^+$ be such that  $\nu^+ (g) = 0$.
Assume that for any $p \neq 0$ and $q \in \bb R$, 
the function $p g + q$ is not cohomologous to a function with values in $\bb Z$.  
Then, 
for any $(z, t) \in \bb X^+ \times \bb R$
and any continuous compactly supported function $F$ on $\bb X_z^- \times \bb X \times \bb R$, 
we have 
\begin{align*}
& \lim_{n \to \infty}  n^{3/2} \int_{\bb X^-_z}  F \left(y,  T^{-n} y \cdot z,  t + \check S_n g(y \cdot z)  \right) 
   \mathds 1_{ \left\{ \check \tau_t^g (y \cdot z) >n - 1 \right\}} \nu^-_z(dy)    \nonumber\\
& =   \frac{2 \check V^{g}(z, t)}{\sqrt{2\pi} \sigma_g^3}  
  \int_{\bb X_z^- \times \bb X \times \bb R}  F \left(y', x', t' \right)  
   \mu^{(-g)}(dx', dt') \check \mu^{g,-}_{z,t} (dy'). 
  \end{align*}

\end{lemma}

\begin{proof}
As usual, it suffices to prove the lemma when $F$ is of the form 
$(y,x,t') \mapsto \mathds 1_{\bb C_{a,z}}(y) G(x,t')$,
where $a \in A_z^m$ and $G$ is a continuous compactly supported function on $\bb X \times \bb R$.

If  $t + S_k g(T^{m-k} (a \cdot z)) \geq 0$ for every $1 \leq k \leq m$, we have that for $n > m$,
\begin{align*}
&  n^{3/2}  \int_{\bb X^-_z}  
  F \left(y, T^{-n} y \cdot z,  t + \check S_n g(y \cdot z)  \right) 
 \mathds 1_{ \left\{ \check \tau_t^g(y \cdot z) >n - 1 \right\} } \nu^-_z(dy)   \notag\\
&  =  n^{3/2}  \exp(-S_m \psi(a \cdot z))    \notag\\
& \qquad \times   \int_{\bb X^-_{a \cdot z} }   
  G \left(T^{-(n-m)} y \cdot (a \cdot z),  t +  S_m g(a \cdot z) + \check S_{n-m} g (y \cdot (a \cdot z)) \right)
   \notag\\
& \qquad \times 
   \mathds 1_{ \left\{ \check \tau_{t + S_m g( a \cdot z )}^g(y \cdot (a \cdot z)) >n - m-1 \right\} } 
     \nu^-_{a \cdot z}(dy). 
\end{align*}
By Lemma \ref{Lem_CondiLLT_thm1}, as $n \to \infty$, the latter quantity converges to 
\begin{align*}
\frac{2 \check V^{g}(a \cdot z, t + S_m g( a \cdot z) )}{\sqrt{2\pi} \sigma_g^3}  
  \int_{\bb X \times \mathbb R}  G \left(x', t' \right)  \mu^{(-g)}(dx', dt') 
  \exp(-S_m \psi(a \cdot z)), 
\end{align*}
which, by the definition of measure $\check \mu^{g,-}_{z,t}$ (see \eqref{Measure_rho}), is equal to 
\begin{align*}
& \frac{2 \check V^{g}(z, t)}{\sqrt{2\pi} \sigma_g^3}  \check \mu^{g,-}_{z,t} (\bb C_{a,z}) 
  \int_{\bb X \times \bb R}  G \left(x', t' \right)  \mu^{(-g)}(dx', dt')  \notag\\
&= \frac{2 \check V^{g}(z, t)}{\sqrt{2\pi} \sigma_g^3}  
  \int_{\bb X_z^- \times \bb X \times \bb R}  F \left(y', x', t' \right)  
   \mu^{(-g)}(dx', dt') \check \mu^{g,-}_{z,t} (dy'). 
\end{align*}

If there exists $1 \leq k \leq m$ with $t + S_k g(T^{m-k} (a \cdot z)) <0$, 
we have $\check \mu^{g,-}_{z,t} (\bb C_{a,z}) = 0$
and
\begin{align*}
\int_{\bb X^-_z}  
  F \left(y, T^{-n} y \cdot z,  t + \check S_n g(y \cdot z)  \right) 
 \mathds 1_{ \left\{ \check \tau_t^g(y \cdot z) >n \right\} } \nu^-_z(dy)  = 0
\end{align*}
for $n > k$. 
The conclusion follows. 
\end{proof}

As usual, from Lemma \ref{Lem_CondiCLLT003},  
we want to deduce the analogous result for functions which depend only on finitely many negative coordinates. 
We shall use the following easy formula that relates the measures $\mu^g$ and $\mu^{g \circ T}$. 

\begin{lemma}\label{Lem_Shift_Invariance_mu}
Let $g \in \scr B$ be such that $\nu(g) = 0$ and $g$ is not a coboundary.
Then, for any continuous compactly supported function $F$ on $\bb X \times \bb R$, we have 
\begin{align*}
\int_{\bb X \times \bb R} F(x, t) \mu^{g \circ T} (dx, dt)
= \int_{\bb X \times \bb R} F(T^{-1} x, t) \mu^{g} (dx, dt).  
\end{align*}
\end{lemma}

\begin{proof}
By using the relation $\tau_t^{g \circ T} = \tau_t^{g} \circ T$, we get
\begin{align*}
\int_{\bb X \times \bb R} F(x, t) \mu^{g \circ T} (dx, dt)
& = \lim_{n \to \infty}  \int_{\bb X \times \bb R}  F(x, t)  S_n (g \circ T)(x)  \mathds 1_{\{ \tau_t^{g \circ T}(x) > n\}}
 \nu(dx) dt  \notag\\
 & = \lim_{n \to \infty}  \int_{\bb X \times \bb R} F(T^{-1} x, t)  S_n g (x)   \mathds 1_{\{ \tau_t^{g}(x) > n \}}
 \nu(dx) dt   \notag\\
& = \int_{\bb X \times \bb R} F(T^{-1} x, t) \mu^{g} (dx, dt), 
\end{align*}
as desired.
\end{proof}

\begin{lemma}\label{Lem_CondiCLLT004}
Let $g \in \scr B$ be such that $\nu(g) = 0$ and there exists $m \geq 0$ with $g \circ T^m \in \scr B^+$. 
Assume that for any $p \neq 0$ and $q \in \bb R$, 
the function $p g + q$ is not cohomologous to a function with values in $\bb Z$.  
Then, 
for any $(z, t) \in \bb X^+ \times \bb R$
and any continuous compactly supported function $F$ on $\bb X_z^- \times \bb X \times \bb R$, 
we have 
\begin{align*}
& \lim_{n \to \infty}  n^{3/2} \int_{\bb X^-_z}  F \left(y,  T^{-n} y \cdot z,  t + \check S_n g(y \cdot z)  \right) 
   \mathds 1_{ \left\{ \check \tau_t^g (y \cdot z) >n - 1 \right\}} \nu^-_z(dy)    \nonumber\\
& =   \frac{2 \check V^{g}(z, t)}{\sqrt{2\pi} \sigma_g^3}  
  \int_{\bb X_z^- \times \bb X \times \bb R}  F \left(y', x', t' \right)  
   \mu^{(-g)}(dx', dt') \check \mu^{g,-}_{z,t} (dy'). 
  \end{align*}
\end{lemma}

\begin{proof}
As in Lemma \ref{Lem_Harmonic_Function_rho}, 
for $a \in A_z^m$, set $F_a$ to be the function on $\bb X^-_{a \cdot z} \times \bb X \times \bb R$
defined by $F_a (y, x, t) = F (y \cdot a, T^m x, t)$.  
Then we have, by setting $h = g \circ T^m$,  
\begin{align*}
& n^{3/2} \int_{\bb X^-_z}  F \left(y,  T^{-n}(y \cdot z),  \check S_n g (y \cdot z)  \right)
  \mathds 1_{ \left\{ \check \tau_t^g(y \cdot z) >n - 1 \right\} } \nu^-_z(dy)  \notag\\
&  = n^{3/2} \sum_{a \in A^m_z} \exp(- S_m \psi(a \cdot z) )  
  \int_{\bb X^-_{a \cdot z}} 
  F_a \left(y,  T^{-n}( y \cdot (a \cdot z) ),  \check S_n h ( y \cdot (a \cdot z) )  \right)    \notag\\
& \qquad \times \mathds 1_{ \left\{ \check \tau_t^h (y \cdot (a \cdot z)) >n - 1 \right\} } \nu^-_{a \cdot z}(dy).   
\end{align*}
By Lemma \ref{Lem_CondiCLLT003},  as $n \to \infty$, this converges to 
\begin{align*}
&   \sum_{a \in A^m_z} \exp(- S_m \psi(a \cdot z) ) 
\frac{2 \check V^{h}(a \cdot z, t)}{\sqrt{2\pi} \sigma_g^3}    \notag\\
& \qquad \times   \int_{\bb X_{a \cdot z}^- \times \bb X \times \bb R}  F_a \left(y', x', t' \right)  
   \mu^{(-h)}(dx', dt') \check \mu^{h,-}_{a \cdot z,t} (dy'). 
\end{align*}
By \eqref{Def_rho_gzt}, the latter quantity is equal to 
\begin{align*}
\int_{\bb X_{z}^- \times \bb X \times \bb R}  F \left(y', T^m x', t' \right)  
   \mu^{(-h)}(dx', dt') \check \mu^{g,-}_{z,t} (dy'). 
\end{align*}
As $h = g \circ T^m$, 
the conclusion now follows from Lemma \ref{Lem_Shift_Invariance_mu}. 
\end{proof}

Now we can give a result for any function $g$ in $\scr B$. 

\begin{lemma}\label{Lem_CondiCLLT005}
Let $g \in \scr B$ be such that $\nu(g) = 0$. 
Assume that for any $p \neq 0$ and $q \in \bb R$, 
the function $p g + q$ is not cohomologous to a function with values in $\bb Z$.  
Then, 
for any continuous compactly supported function $F$ on $\bb X \times \bb X \times \bb R \times \bb R$, 
we have 
\begin{align*}
& \lim_{n \to \infty}  n^{3/2} \int_{\bb X \times \bb R}  
F \left(x,  T^{-n} x, t,  t + \check S_n g(x)  \right) 
   \mathds 1_{ \left\{ \check \tau_t^g (x) >n - 1 \right\}} \nu(dx) dt    \nonumber\\
& =   \frac{2}{ \sqrt{2\pi}  \sigma_g^3 } 
  \int_{\bb X \times \bb X \times \bb R \times \bb R}  F \left(x, x', t, t' \right)  
   \mu^{(-g)}(dx', dt')  \check \mu^{g} (dx, dt). 
  \end{align*}
\end{lemma}

\begin{proof}
We can assume that the function $F$ is non-negative. 
For $(z,t) \in \bb X^+ \times \bb R$, denote
\begin{align*}
W_n (z,t) =   n^{3/2} \int_{\bb X_z^-}  
F \left(y \cdot z,  T^{-n} (y \cdot z), t,  t + \check S_n g(y \cdot z)  \right) 
   \mathds 1_{ \left\{ \check \tau_t^g (y \cdot z) >n - 1 \right\}} \nu^-_z(dy). 
\end{align*} 
Let $(g_m)_{m \geq 0}$ be the sequence of H\"{o}lder continuous functions as in Lemma \ref{Lem_Appro_g}. 
For any $n, m \geq 0$, we set 
\begin{align*}
F_m^+ (x, x', t, t') & = \sup_{|s'| \leq 2 c_1 \alpha^m} F (x, x', t - 2 c_1 \alpha^m, t'+s'),   \\
F_m^- (x, x', t, t')  & = \inf_{|s'| \leq 2 c_1 \alpha^m} F (x, x', t + 2 c_1 \alpha^m, t'+s'),
\end{align*}
and 
\begin{align*}
& W_{n,m}^+ (z,t) =   n^{3/2} \int_{\bb X_z^-}  
F_m^+ \left(y \cdot z,  T^{-n} (y \cdot z), t,  t + \check S_n g_m(y \cdot z)  \right)   \notag\\
 & \qquad\qquad\qquad\qquad
 \times  \mathds 1_{ \left\{ \check \tau_t^{g_m} (y \cdot z) >n - 1 \right\}} \nu^-_z(dy),    \\
& W_{n,m}^- (z,t) =   n^{3/2} \int_{\bb X_z^-}  
F_m^- \left(y \cdot z,  T^{-n} (y \cdot z), t,  t + \check S_n g_m (y \cdot z)  \right)   \\
& \qquad\qquad\qquad\qquad
  \times  \mathds 1_{ \left\{ \check \tau_t^{g_m} (y \cdot z) >n - 1 \right\}} \nu^-_z(dy). 
\end{align*}
For $z \in \bb X^+$ and $t \in \bb R$, it holds that
\begin{align*}
W_{n,m}^- (z,t - 2 c_1 \alpha^m) 
\leq W_n (z,t)  
\leq  W_{n,m}^+ (z,t + 2 c_1 \alpha^m). 
\end{align*} 
By taking the limit as $n \to \infty$, we get by Lemma \ref{Lem_CondiCLLT004}, 
\begin{align}\label{CondiLLT_Fatou}
& \frac{2 \check V^{g_m}(z, t - 2 c_1 \alpha^m )}{\sqrt{2\pi} \sigma_g^3}  
  \int_{\bb X_z^- \times \bb X \times \bb R}  F_m^- \left(y \cdot z, x', t - 2 c_1 \alpha^m, t'  \right)  \notag\\
 & \qquad\qquad\qquad\qquad\qquad\qquad
  \times  \mu^{(-g_m)}(dx', dt') \check \mu^{g_m,-}_{z,t - 2 c_1 \alpha^m } (dy)   \notag\\
 & \leq \liminf_{n \to \infty} W_n (z,t)  
\leq \limsup_{n \to \infty} W_n (z,t)    \notag\\
& \leq  \frac{2 \check V^{g_m}(z, t + 2 c_1 \alpha^m )}{\sqrt{2\pi} \sigma_g^3}  
  \int_{\bb X_z^- \times \bb X \times \bb R}  F_m^+ \left(y \cdot z, x', t + 2 c_1 \alpha^m, t'  \right)   \notag\\
 & \qquad\qquad\qquad\qquad\qquad\qquad
  \times    \mu^{(-g_m)}(dx', dt') \check \mu^{g_m,-}_{z,t + 2 c_1 \alpha^m } (dy). 
\end{align}
On the one hand, after integrating over $\bb X^+ \times \bb R$ in \eqref{CondiLLT_Fatou} with respect to the product of $\nu^+$ with the Lebesgue measure,  
by Fatou's lemma and Lemma \ref{Lem_CondiCLLT004}, we get 
\begin{align}\label{CondiLLT_Fatou_Lower}
& \frac{2}{ \sqrt{2\pi}  \sigma_g^3 }   \int_{\bb X \times  \bb X \times \bb R \times \bb R} 
   F_m^- \left(x, x', t, t'  \right)   \mu^{(-g_m)}(dx', dt') \check \mu^{g_m} (dx, dt)    \notag\\
& \leq  \liminf_{n \to \infty} 
  n^{3/2} \int_{\bb X \times \bb R}  
F \left(x,  T^{-n} x, t,  t + \check S_n g(x)  \right) 
   \mathds 1_{ \left\{ \check \tau_t^g (x) >n - 1 \right\}} \nu(dx) dt. 
\end{align}
To conclude, we need to show the reverse Fatou property holds. 
To this aim, we choose a non-negative continuous compactly supported function $G$ on $\bb R$
such that for any $(x, x', t, t') \in \bb X \times  \bb X \times \bb R \times \bb R$, 
one has 
$
F_0^+(x, x', t, t') \leq G(t) G(t'). 
$
Then, we get for $(z,t) \in \bb X^+ \times \bb R$, 
\begin{align*}
W_{n,0}^+ (z,t) \leq U_n(z,t) := n^{3/2} G(t) \int_{\bb X_z^-}  G(t + \check S_n g_0(y \cdot z)) 
\mathds 1_{ \left\{ \check \tau_t^{g_0} (y \cdot z) >n - 1 \right\}} \nu^-_z(dy). 
\end{align*}
By Lemma \ref{Lem_CondiLLT_001},  $U_n(z,t)$ converges uniformly in $(z,t) \in \bb X^+ \times \bb R$. 
Therefore, by applying Fatou's lemma to the sequence $U_n(z,t) - W_{n} (z,t)$, we get
by integrating over $\bb X^+ \times \bb R$ in \eqref{CondiLLT_Fatou} with respect to the product of $\nu^+$ with the Lebesgue measure, 
\begin{align}\label{CondiLLT_Fatou_Upper}
& \frac{2}{ \sqrt{2\pi}  \sigma_g^3 }     \int_{\bb X \times  \bb X \times \bb R \times \bb R} 
   F_m^+ \left(x, x', t, t'  \right)   \mu^{(-g_m)}(dx', dt') \check \mu^{g_m} (dx, dt)    \notag\\
& \geq  \limsup_{n \to \infty} 
  n^{3/2} \int_{\bb X \times \bb R}  
F \left(x,  T^{-n} x, t,  t + \check S_n g(x)  \right) 
   \mathds 1_{ \left\{ \check \tau_t^g (x) >n - 1 \right\}} \nu(dx) dt. 
\end{align}
By letting $m \to \infty$,
the conclusion follows from \eqref{CondiLLT_Fatou_Lower}, \eqref{CondiLLT_Fatou_Upper}
and  Lemma \ref{Lem_Continuity_rho_g}.  
\end{proof}

\begin{proof}[Proof of Theorem \ref{Thm-conLLT}]
By the duality lemma (Lemma \ref{Corollary_Duality}), we have 
\begin{align*}
&    \int_{\bb X \times \bb R} F(x, T^nx, t, t + S_n f(x))
    \mathds 1_{\left\{ \tau_t^f(x) >n - 1 \right\} }  \nu(dx) dt   \notag\\
  & \quad  =  \int_{\bb X \times \bb R} F(T^{-n} x,  x, t - \check S_n f(x), t )   
        \mathds 1_{\left\{ \check \tau_{t }^{(-f)}(x) >n - 1 \right\} }  \nu(dx) dt. 
\end{align*}
Now the conclusion follows from Lemma \ref{Lem_CondiCLLT005}. 
\end{proof}


\end{document}